\documentclass[12pt,a4paper,twoside,reqno]{amsart}
\title[Gauge theory for string algebroids]{Gauge theory for string algebroids}
\author[M. Garcia-Fernandez]{Mario Garcia-Fernandez}
\author[R. Rubio]{Roberto Rubio}
\author[C. Tipler]{Carl Tipler}

\address{Dep. Matem\'aticas, Universidad Aut\'onoma de Madrid, and Instituto de Ciencias Matem\'aticas (CSIC-UAM-UC3M-UCM), Cantoblanco, 28049 Madrid, Spain}
\email{mario.garcia@icmat.es}

\address{Universitat Aut\`onoma de Barcelona, 08193 Barcelona, Spain}
\email{roberto.rubio@uab.es}

\address{Univ Brest, UMR CNRS 6205, Laboratoire de Math\'ematiques de Bretagne Atlantique, France}
\email{carl.tipler@univ-brest.fr}

\thanks{This project has received funding from the European Union’s Horizon 2020 research and innovation programme under the Marie Sk\l odowska-Curie grant agreement No 750885 GENERALIZED.
M.G.F. was partially supported by the Spanish Ministry of Science and Innovation, through the `Severo Ochoa Programme for Centres of Excellence in R\&D' (SEV-2015-0554 and CEX2019-000904-S), and under grants PID2019-109339GA-C32 and EUR2020-112265.
R.R. was supported by a Marie Sk\l odowska-Curie Individual Fellowship. C.T. was supported by the French government ``Investissements d'Avenir'' programme ANR-11-LABX-0020-01 and ANR project EMARKS No ANR-14-CE25-0010.}

\usepackage{hyperref,url}
\usepackage{amssymb,latexsym}
\usepackage{amsmath,amsthm}
\usepackage{enumerate}
\usepackage[all]{xy} \CompileMatrices
\usepackage{amscd}
\usepackage{leftidx}

\usepackage{slashed} 

\SelectTips{cm}{12} 
\theoremstyle{plain}
\newtheorem{theorem}{Theorem}[section]

\newtheorem{lemma}[theorem]{Lemma}
\newtheorem{corollary}[theorem]{Corollary}
\newtheorem{proposition}[theorem]{Proposition}
\newtheorem{conjecture}[theorem]{Conjecture}
\theoremstyle{definition}
\newtheorem{definition}[theorem]{Definition}
\newtheorem{definition-theorem}[theorem]{Definition-Theorem}
\newtheorem{example}[theorem]{Example}

\newtheorem{innercustomcond}{Condition}
\newenvironment{customcond}[1]
{\renewcommand\theinnercustomcond{#1}\innercustomcond}
{\endinnercustomcond}

\theoremstyle{remark}
\newtheorem{remark}[theorem]{Remark}

\numberwithin{equation}{section} 

\newcommand{\tr}{\operatorname{tr}}
\newcommand{\Id}{\operatorname{Id}}

\newcommand{\End}{\operatorname{End}}

\newcommand{\Ker}{\operatorname{Ker}}
\newcommand{\Img}{\operatorname{Im}}
\newcommand{\ad}{\operatorname{ad}}

\newcommand{\Aut}{\operatorname{Aut}}
\newcommand{\dbar}{\bar{\partial}}

\newcommand{\CC}{{\mathbb C}}

\newcommand{\RR}{{\mathbb R}}

\newcommand{\rk}{\operatorname{rk}}

\newcommand{\surj}{\to\kern-1.8ex\to}

\newcommand{\lra}[1]{\stackrel{#1}{\longrightarrow}}

\newcommand{\cA}{\mathcal{A}}
\newcommand{\cC}{\mathcal{C}}

\newcommand{\cM}{\mathcal{M}}

\newcommand{\cF}{\mathcal{F}}

\newcommand{\cW}{\mathcal{W}}

\newcommand{\cG}{\mathcal{G}}
\newcommand{\cL}{\mathcal{L}}
\newcommand{\cO}{\mathcal{O}}

\newcommand{\cR}{\mathcal{R}}
\newcommand{\cS}{\mathcal{S}}

\newcommand{\Lie}{\operatorname{Lie}}

\newcommand{\cH}{\mathcal{H}} 

\def\om{\omega}
\def\Om{\Omega}

\def\Lie{\mathrm{Lie}}

\def\Im{\mathrm{Im}}
\def\Id{\mathrm{Id}}

\def\cA{\mathcal{A}}

\def\cF{\mathcal{F}}
\def\cG{\mathcal{G}}
\def\cH{\mathcal{H}}

\def\cR{\mathcal{R}}

\newcommand{\st}{\;|\;}

\newcommand{\R}{{\mathbb{R}}}
\newcommand{\C}{{\mathbb{C}}}

\newcommand{\la}{\langle}
\newcommand{\ra}{\rangle}

\newcommand{\E}{\mathrm{E}}

\renewcommand{\Re}{\mathrm{Re}}
\renewcommand{\Im}{\mathrm{Im}}

\begin{document}

\begin{abstract}
We introduce a moment map picture for {\it holomorphic string algebroids} where the Hamiltonian gauge action is described by means of inner automorphisms of Courant algebroids. 
The zero locus of our moment map is given by the solutions of the {\it Calabi system}, a coupled system of equations which provides a unifying framework for the classical Calabi problem and the Hull-Strominger system. Our main results are concerned with the geometry of the moduli space of solutions, and assume a technical condition which is fulfilled in examples. We prove that the moduli space carries a pseudo-K\"ahler metric with K\"ahler potential given by the \emph{dilaton functional}, a topological formula for the metric, and an infinitesimal Donaldson-Uhlenbeck-Yau type theorem. 
\end{abstract}

\maketitle

\tableofcontents

\section{Introduction}
\label{sec:intro}

Back to the work of Atiyah and Bott \cite{AB}, the interaction of 
Yang-Mills theory with symplectic geometry and, in particular, the idea of \emph{moment map}, has had an important impact in our understanding of the moduli theory for holomorphic vector bundles in algebraic geometry. The seed relation between stable bundles on a Riemann surface and flat unitary connections observed in \cite{AB,NS} was largely expanded with the Donaldson-Uhlenbeck-Yau Theorem \cite{Don,UY}. This important result, initially conjectured by Hitchin and Kobayashi, establishes a correspondence between the moduli space of solutions of the Hermite-Yang-Mills equations and the moduli space of slope-stable bundles on a compact K\"ahler manifold. A key upshot is that certain moduli spaces in algebraic geometry, constructed via Mumford's theory of stability, are endowed with natural symplectic structures.

Our main goal in the present work is to explore a new scenario where the `moment map picture' arises, tightly bound up with higher gauge theory. Inspired by the Atiyah and Bott construction, our starting point is a class of holomorphic bundle-like objects on a compact complex manifold $X$, known as \emph{string algebroids} \cite{grt2}. A string algebroid $Q$ is a special class of holomorphic Courant algebroid, which can be thought of as the `higher Atiyah algebroid' of a holomorphic principal bundle for the (complexified) \emph{string group} \cite{Xu}. In the case of our interest, the geometric content of $Q$ comprises, in particular, a holomorphic principal $G$-bundle $P$ over $X$ with vanishing first Pontryagin class
$p_1(P) = 0$
and a holomorphic extension
\begin{equation*}\label{eq:defstringintro}
	\xymatrix{
		0 \ar[r] & T^*X \ar[r] & Q \ar[r] & A_P \ar[r] & 0
	}
\end{equation*}
of the holomorphic Atiyah algebroid $A_P$ of $P$ by the holomorphic cotangent bundle. We assume $G$ to be a complex reductive Lie group with a fixed non-degenerate invariant pairing $\la\, ,\ra$ on its Lie algebra.

In this work we shall study gauge theoretical aspects of holomorphic string algebroids. For this, we start by developing basic aspects of the theory, such as gauge symmetries and a \emph{Chern correspondence} in our setting. A natural complex gauge group is introduced in Section \ref{sec:Morita}, given by \emph{inner symmetries} of a smooth complex string algebroid $E$ (the analogous concept in the smooth category) canonically associated to a string algebroid $Q$ (see Lemma \ref{lemma:bricks}). 
The Chern correspondence (Lemmas \ref{lem:Cherncorr} and \ref{lem:Cherncorrsets}) requires the study of a notion of \emph{compact form} for $Q$ by means of \emph{real string algebroids}
$$
E_\RR \subset E.
$$
A compact form $E_\RR$ determines a reduction $P_h \subset P$ to a maximal compact subgroup $K \subset G$ (see Definition \ref{def:realform}). 
The Chern correspondence associates to each compact form on $Q$ a horizontal subspace 
$$
W \subset E_\RR,
$$
which provides the analogue of the Chern connection in our context. In agreement with structural properties of connections in higher gauge theory \cite{Gastel,SWolf}, 
any such $W \subset E_\RR$ determines the classical Chern connection $\theta^h$ of $P_h \subset P$ and a real $(1,1)$-form $\omega$ satisfying a structure equation (see Proposition \ref{propo:Chernclassic}).


We move on to study the geometry of the infinite-dimensional space of horizontal subspaces $W$ on a fixed compact form $E_\RR$ whose associated $(1,1)$-form $\omega$ is hermitian. Via the Chern correspondence, this space has a (possibly degenerate) pseudo-K\"ahler structure for each choice of a smooth volume form $\mu$ on $X$ and \emph{level} $\ell \in \RR$. There is a global K\"ahler potential given by $- \log$ of the \emph{dilaton functional} $M_\ell$, that is,
\begin{equation}\label{eq:Kintro}
- \log M_\ell := - \log \int_X e^{-\ell f_\omega} \frac{\omega^n}{n!},
\end{equation}
where $f_\omega := \tfrac{1}{2}\log (\frac{\omega^n}{n! \mu})$. In Proposition \ref{prop:mmapCalabil} we prove that there is a natural Hamiltonian action for a subgroup of $\Aut(E)$ preserving the compact form, with zero locus for the moment map given by solutions of the coupled equations
	\begin{equation}\label{eq:Calabilintro}
	\begin{split} 
	F\wedge \omega^{n-1} & = 0, \qquad  \qquad \qquad \;\;\;\, F^{0,2}  = 0,\\
	d (e^{-\ell f_\omega}\omega^{n-1}) & = 0, \qquad  dd^c \om + \la F\wedge F \ra  = 0.
	\end{split}
	\end{equation}
Here $F$ is the curvature of a connection on the principal $K$-bundle underlying $E_\RR$, which is determined by $W \subset E_\RR$. 

The equations \eqref{eq:Calabilintro} were first found in \cite{grst} for $\ell = 1$ in a holomorphic setting, in relation to the critical locus of the dilaton functional $M_1$. By Proposition \ref{prop:mmapCalabil}, they can be regarded as a natural analogue of the Hermite-Yang-Mills equations for string algebroids. Following \cite{grst}, we will refer to \eqref{eq:Calabilintro} as the \emph{Calabi system}. These moment map equations provide a unifying framework for the classical Calabi problem, which is recovered when $K$ is trivial (see Section \ref{sec:mmapHS}), and the Hull-Strominger system \cite{HullTurin,Strom} (see \cite{Fei,GF2,PPZ2} for recent surveys covering this topic). For the latter, we assume that $X$ is a (non-necessarily K\"ahler) Calabi-Yau threefold with holomorphic volume form $\Omega$ and we take $\ell = 1$ and
\begin{equation}\label{eq:muOmega2intro}
\mu = (-1)^{\frac{n(n-1)}{2}}i^n\Omega \wedge \overline{\Omega}.
\end{equation}
To our knowledge, Corollary \ref{cor:mmapHS} provides the first moment map interpretation of the Hull-Strominger system in the mathematics literature (see \cite{Waldram} for an alternative construction in the physics literature). As a matter of fact, this was our original motivation when we initiated the present work.

Our main results, discussed briefly over the next section, are devoted to the geometry of the moduli space of solutions of \eqref{eq:Calabilintro}. Assuming a technical Condition \ref{ConditionA} which is fulfilled in examples (see Section \ref{sec:example}), we shall prove that the moduli space carries a (possibly degenerate) pseudo-K\"ahler metric with K\"ahler potential \eqref{eq:Kintro} (see  Theorem \ref{thm:metric}), a topological formula for the metric (see Theorem \ref{thm:metricfibre}), and an infinitesimal Donaldson-Uhlenbeck-Yau type Theorem (see Theorem \ref{thm:DUYinfinitesimal}). Interestingly, the non-degeneracy of the metric seems to be very sensitive to the level $\ell \in \RR$.

\subsection*{Main results}

Throughout this section we fix a solution $W$ of the Calabi system \eqref{eq:Calabilintro} on a compact form $E_\RR$. Via the Chern correspondence, $W$ determines a string algebroid $Q$ with underlying holomorphic principal $G$-bundle $P$. In addition, $W$ determines two cohomological quantities which play an important role in the present paper, namely, a balanced class and an `Aeppli class of $Q$' 
\begin{equation}\label{eq:balancedAeppliintro}
\mathfrak{b}:= \frac{1}{(n-1)!}[e^{-\ell f_\omega} \omega^{n-1}] \in H^{n-1,n-1}_{BC}(X,\RR), \qquad \mathfrak{a} = [E_\RR] \in \Sigma_A(Q,\RR).
\end{equation}
The space $\Sigma_A(Q,\RR)$ is constructed via Bott-Chern secondary characteristic classes and is affine for a subspace of $H^{1,1}_A(X,\RR)$ (see Equation \eqref{eqannex:partialmap}). 

For the sake of clarity, we will assume throughout this introduction that $G$ is semisimple.
On the other hand, our main results assume Condition \ref{ConditionA}. In a nutshell, this technical condition states that any element in the kernel of the linearization of \eqref{eq:Calabilintro} along the Aeppli class $\mathfrak{a}$ determines an infinitesimal automorphism of $Q$ (see Remark \ref{rem:ConditionAgeometry}). This is very natural, as it typically follows for geometric PDE with a moment map interpretation. In Proposition \ref{prop:deformation} we discuss a class of non-K\"ahler examples of solutions of \eqref{eq:Calabilintro} where Condition \ref{ConditionA} applies, obtained via deformation of a K\"ahler metric.



Our main theorem relies on a
gauge fixing mechanism for infinitesimal variations $(\dot \omega,\dot b,\dot a ) \in \Omega^{1,1}_\RR \oplus \Omega^2 \oplus \Omega^1(P_h)$ of the Calabi system \eqref{eq:Calabilintro}, which requires Condition \ref{ConditionA} (see Proposition \ref{prop:Calabillineargaugefixed}). To state the result, we use the decomposition $\dot \omega = \dot \omega_0 + (\Lambda_\omega \dot \omega) \omega /n$ into primitive and non-primitive parts with respect to the hermitian form $\omega$. Denote by $\cM_\ell$ the moduli space of solutions of the Calabi system (see Section \ref{sec:gaugefixing}). A precise statement is given in Theorem \ref{thm:metric}.

\begin{theorem}\label{thm:metricintro}
Assume Condition \ref{ConditionA}. Then, the tangent space to $\cM_\ell$ at $[W]$ inherits a pseudo-K\"ahler structure with (possibly degenerate) metric
\begin{equation}\label{eq:glWmoduliintro}
	\begin{split}
	g_\ell (\dot \omega, \dot b,\dot a) & {} =  \frac{\ell - 2}{M_\ell} \int_X \la \dot a \wedge J \dot a \ra \wedge e^{-\ell f_\omega} \frac{\omega^{n-1}}{(n-1)!}\\
&  \phantom{ {} = } +  \frac{2 - \ell}{2 M_\ell} \int_X (|\dot \omega_0|^2 + |\dot b_0^{1,1}|^2) e^{-\ell f_\omega} \frac{\omega^{n}}{n!}\\
	& \phantom{ {} = } + \frac{2 - \ell}{2 M_\ell}\Bigg{(}\frac{\ell}{2} - \frac{n-1}{n}\Bigg{)} \int_X (|\Lambda_\omega \dot b|^2 + | \Lambda_\omega \dot \omega|^2) e^{-\ell f_\omega} \frac{\omega^{n}}{n!}\\
	& \phantom{ {} = }  + \left(\frac{2-\ell}{2M_\ell}\right)^2\left(\left(\int_X \Lambda_\omega \dot \omega e^{-\ell f_\omega} \frac{\omega^n}{n!}\right)^2 + 
	 \left(\int_X \Lambda_\omega\dot b e^{-\ell f_\omega} \frac{\omega^n}{n!}\right)^2\right).
	\end{split}
\end{equation}
\end{theorem}

Ignoring topological issues, the significance of our main theorem is that the smooth locus of the moduli space $\cM_\ell$
inherits a (possibly degenerate) pseudo-K\"ahler metric $g_\ell$ with K\"ahler potential \eqref{eq:Kintro}. An interesting upshot of our formula for the moduli space metric is that along the `bundle directions', given formally by the first line in formula \eqref{eq:glWmoduliintro}, the metric is conformal to the Atiyah-Bott-Donaldson pseudo-K\"ahler metric on the moduli space of Hermite-Yang-Mills connections with fixed hermitian metric $\omega$ (see \cite{AB,Don,lt}). 
This statement must be handled very carefully, since the hermitian metric $\omega$ in our picture varies in a complicated way from point to point in the moduli space.

Motivated by this observation, in Theorem \ref{thm:metricfibre} we study the structure of the metric \eqref{eq:glWmoduliintro} along the fibres of a natural map from $\cM_\ell$ to the moduli space of holomorphic principal $G$-bundles with $h^{0}(\ad P) = 0$, proving the following formula:
\begin{equation}\label{eq:metricfibreintro}
\begin{split}
g_\ell & = 
\frac{2 - \ell}{2 M_\ell}  \Bigg{(} \frac{2-\ell}{2M_\ell} (\Re \; \dot{\mathfrak{a}} \cdot \mathfrak{b})^2 - \Re \; \dot{\mathfrak{a}} \cdot \Re \; \dot{\mathfrak{b}} + \frac{2-\ell}{2M_\ell} (\Im \; \dot{\mathfrak{a}} \cdot \mathfrak{b})^2 - \Im \; \dot{\mathfrak{a}} \cdot\Im \; \dot{\mathfrak{b}} \Bigg{)}.
\end{split}
\end{equation}
Here, $\dot{\mathfrak{b}}\in H_{BC}^{n-1,n-1}(X)$ and $\dot{\mathfrak{a}} \in H^{1,1}_A(X)$ are `complexified variations' of the Bott-Chern class and the Aeppli class of the solution in \eqref{eq:balancedAeppliintro}, obtained via gauge fixing (see Lemma \ref{lem:Calabillinearfibre}). Formula \eqref{eq:metricfibreintro} shows that the moduli space metric \eqref{eq:glWmoduliintro} is `semi-topological', in the sense that fibre-wise it can be expressed in terms of classical cohomological quantities.

When the structure group $K$ is trivial, $X$ is a K\"ahler Calabi-Yau threefold, we take the volume form as in \eqref{eq:muOmega2intro} and $\ell = 1$, equation \eqref{eq:metricfibreintro} matches Strominger's formula for the special K\"ahler metric on the `complexified K\"ahler moduli' of $X$ \cite[Eq. (4.1)]{CD}. As a consequence of our Theorem \ref{thm:metricintro}, this classical moduli space is recovered, along with its special K\"ahler metric, by pseudo-K\"ahler reduction in Corollary \ref{cor:Calabimoduli}. As an application of \eqref{eq:metricfibreintro}, in Section \ref{sec:example} we show that any stable vector bundle $V$ over $X$ satisfying
$$
c_1(V) = 0, \qquad c_2(V) = c_2(X)
$$
determines a deformation of the special K\"ahler geometry of the complexified K\"ahler moduli of $X$ to an explicit family of pseudo-K\"ahler metrics (see Example \ref{example:CY}).

On a (non-necessarily K\"ahler) Calabi-Yau threefold $(X,\Omega)$, \eqref{eq:Calabilintro} is equivalent to the Hull-Strominger system \cite{HullTurin,Strom} provided that $\ell = 1$ and we take $\mu$ as in \eqref{eq:muOmega2intro}. For this interesting case, the physics of string theory predicts that the fibre-wise moduli metric \eqref{eq:metricfibreintro} should be positive definite (see Conjecture \ref{conj:Kahlerpotential0}). This way, we obtain a physical prediction relating the variations of the Aeppli classes and balanced classes of solutions.

\begin{conjecture}\label{conj:theorem2intro}
If $(X,P)$ admits a solution of the Hull-Strominger system, then \eqref{eq:metricfibreintro} is positive definite. In particular, the variations of the Aeppli and balanced classes of nearby solutions must satisfy
\begin{equation}\label{eq:ineqmoduliintro}
\Re \; \dot{\mathfrak{a}} \cdot \Re \; \dot{\mathfrak{b}} < \frac{1}{2\int_X \|\Omega\|_\omega \frac{\omega^3}{6}} (\Re \; \dot{\mathfrak{a}} \cdot \mathfrak{b})^2.
\end{equation}
\end{conjecture}

Formula \eqref{eq:ineqmoduliintro} provides a potential obstruction to the existence of solutions of the Hull-Strominger system around a given a solution. 
We expect this phenomenon to be related to some obstruction to the global existence, which goes beyond the slope stability of the bundle and the balanced property of the manifold (cf. \cite{Yau2}). It would be interesting to obtain a physical explanation for the inequality \eqref{eq:ineqmoduliintro}. 

Our last result can be regarded as an infinitesimal Donaldson-Uhlenbeck-Yau type theorem, relating the moduli space of solutions of the Calabi system with a Teichm\"uller space for string algebroids (see Section \ref{sec:DUY}). A precise formulation can be found in Theorem \ref{thm:DUYinfinitesimal}.

\begin{theorem}\label{thm:DUYinfinitesimalintro}
Assume Condition \ref{ConditionA} and $h^{0,1}_A(X) = h^{0}(\ad P) = 0$. Then, the tangent space to the moduli space $\cM_\ell$ at $[W]$ is canonically isomorphic to the tangent space to the Teichm\"uller space for string algebroids at $[Q]$.
\end{theorem}

This strongly suggests that---if we shift our perspective and consider the Calabi system as equations for a compact form on a fixed Bott-Chern algebroid $Q$ along a fixed Aeppli class---the existence of solutions should be related to a stability condition in the sense of Geometric Invariant Theory. This was essentially the point of view taken in \cite{grst}. The precise relation with stability in our context is still unclear, as the balanced class $\mathfrak{b} \in H^{n-1,n-1}_{BC}(X,\RR)$ of the solution varies in the moduli space $\cM_\ell$. The conjectural stability condition which characterizes the existence of solutions should be related to the \emph{integral of the moment map}, given by the dilaton functional $M_\ell$. We speculate that there is a relation between this new form of stability and the conjectural inequality \eqref{eq:ineqmoduliintro}. The global structure of the moduli space $\cM_\ell$ is also a mistery to us. An important insight for future studies of this structure might be provided by the moduli space metric in our Theorem \ref{thm:metricintro}.

The moduli space of solutions of the Hull-Strominger system has been an active topic of research over the last years, both in the mathematics and physics literature (see \cite{AGS,AOMSS,BeckerTsengYau,OssaSvanes,grt} and references therein). 
Recent advances on the moduli metric for heterotic flux compactifications including bundles, as described by the Hull-Strominger system, can be found in \cite{AQS,CDM,GLM,McSisca}. In particular, our formula for the K\"ahler potential \eqref{eq:K0} can be compared with \cite[Eq. (1.3)]{CDM}, where a construction of the metric was provided. While we were finishing the present work, an alternative moment map interpretation of the Hull-Strominger system using spinorial methods has been independently obtained in the physics literature \cite{Waldram}. There is indeed a match between their formula for the moduli K\"ahler potential and our formula \eqref{eq:K0}.  We also expect a relation between our gauge theory picture for string algebroids and the generalization of the holomorphic Casson invariant proposed in \cite{AOMSS}. 



The paper is organized as follows. In Section 2 we recall the definition of string algebroids and their classification, we introduce a way to describe them in terms of smooth data by using liftings, and give some explicit models. In Section 3 we discuss reductions of complex string algebroids and define the complex gauge group for the theory. Section 4 is devoted to Bott-Chern algebroids and the Chern correspondence. In Section 5, we give the moment map picture for the Calabi system. Finally, in Section 6 we describe, provided the technical hypothesis, a pseudo-K\"ahler metric on the space of solutions. Appendix A deals with the relation between compact forms of a complex string algebroids and the complex gauge group, whereas Appendix B offers an explanation to the variation of complexified Aeppli classes in the pseudo-K\"ahler metric.

\textbf{Acknowledgments:} The authors would like to thank Luis \'Alvarez-C\'onsul, Vestislav Apostolov, Jean-Michel Bismut, Marco Gualtieri, Nigel Hitchin, Fernando Marchesano, Jock McOrist, Carlos Shahbazi and Martin Ziegler for helpful conversations.

\section{String algebroids and liftings}

\subsection{Holomorphic string algebroids}\label{sec:string}

Let $X$ be a complex manifold of dimension $n$. We denote by $\cO_X$ and $\underline{\CC}$ the sheaves of holomorphic functions and $\CC$-valued locally constant functions on $X$, respectively. A holomorphic Courant algebroid $(Q,\la\, , \ra,[\,,],\pi)$ over $X$ consists of a holomorphic vector bundle $Q \to X$, with sheaf of sections denoted also by $Q$, together with a holomorphic non-degenerate symmetric bilinear form $\la\, ,\ra$, 
a holomorphic vector bundle morphism $\pi:Q\to TX$, and a homomorphism of sheaves of $\underline{\CC}$-modules
$$
[\, ,] \colon Q \otimes_{\underline{\CC}} Q \to Q, 
$$
satisfying natural axioms, 
for sections $u,v,w$ of $Q$ and $\phi\in \cO_X$,
\begin{itemize}
	\item[(D1):] $[u,[v,w]] = [[u,v],w] + [v,[u,w]]$,
	\item[(D2):] $\pi([u,v])=[\pi(u),\pi(v)]$,
	\item[(D3):] $[u,\phi v] = \pi(u)(\phi) v + \phi[u,v]$,
	\item[(D4):] $\pi(u)\la v, w \ra = \la [u,v], w \ra + \la v,
	[u,w] \ra$,
	\item[(D5):] $[u,v]+[v,u]=\mathcal{D}\la u, v\ra$,
  \end{itemize}
where $\mathcal{D} \colon \mathcal{O}_X \to Q$ denotes the composition of the exterior differential, the natural map $\pi^* \colon T^*X \to Q^*$, and the isomorphism $Q^* \to Q$ provided by $\la\cdot,\cdot\ra$.

Given a holomorphic Courant algebroid $Q$ over $X$ with surjective anchor map $\pi$, there is an associated holomorphic Lie algebroid 
$$
A_Q := Q/(\Ker \pi)^\perp.
$$ 
Furthermore, the holomorphic subbundle
$$
\ad_Q := \Ker \pi/(\Ker \pi)^\perp \subset A_Q
$$
inherits the structure of a holomorphic bundle of quadratic Lie algebras.

Let $G$ be a 
complex Lie group with Lie algebra $\mathfrak{g}$, and consider an invariant non-degenerate pairing $\la\, , \ra : \mathfrak{g} \otimes \mathfrak{g} \to \mathbb{C}$. 
Let $p \colon P \to X$ be a holomorphic principal $G$-bundle over $X$.
Consider the holomorphic Atiyah Lie algebroid $A_{P} := TP/G$ of $P$, with anchor map $dp \colon A_P \to TX$ and bracket induced by the Lie bracket on $TP$. The holomorphic bundle of Lie algebras 
$
\Ker dp = \ad P \subset A_{P}
$
fits into the short exact sequence of holomorphic Lie algebroids
$$
0 \to \ad P \to A_{P} \to TX \to 0.
$$

\begin{definition}\label{def:Courant}
	A \emph{string algebroid} is a tuple $(Q,P,\rho)$, where $P$ is a holomorphic principal $G$-bundle over $X$, $Q$ is a holomorphic Courant algebroid over $X$, and $\rho$ is a bracket-preserving morphism inducing a short exact sequence
	\begin{equation}\label{eq:defstring}
	\xymatrix{
		0 \ar[r] & T^*X \ar[r] & Q \ar[r]^\rho & A_P \ar[r] & 0,
	}
	\end{equation}
	such that the induced map of holomorphic Lie algebroids $\rho \colon A_Q \to A_P$ is an isomorphism restricting to an isomorphism $\ad_Q \cong (\ad P,\la\,,\ra)$.
\end{definition}

We are interested in the classification of these objects up to isomorphism, as given in the following definition.

\begin{definition}[\cite{grt2}]\label{def:stringholCourmor}
	A morphism from $(Q,P,\rho)$ to $(Q',P',\rho')$ is a pair $(\varphi,g)$, where $\varphi \colon Q \to Q'$ is a morphism of holomorphic Courant algebroids and $g \colon P \to P'$ is a homomorphism of holomorphic principal bundles covering the identity on $X$, such that the following diagram is commutative.
	\begin{equation*}\label{eq:defstringiso}
	\xymatrix{
		0 \ar[r] & T^*X \ar[r] \ar[d]^{id} & Q \ar[r]^\rho \ar[d]^{\varphi} & A_P \ar[r] \ar[d]^{dg} & 0,\\
		0 \ar[r] & T^*X \ar[r] & Q' \ar[r]^{\rho'} & A_{P'} \ar[r] & 0.
	}
	\end{equation*}
	We say that $(Q,P,\rho)$ is isomorphic to $(Q',P',\rho')$ if there exists a morphism $(\varphi,g)$ such that $\varphi$ and $g$ are isomorphisms.
\end{definition}

To recall the basic classification result that we need, we introduce now some notation which will be used in the rest of the paper. Given a holomorphic principal $G$-bundle $P$ over $X$, denote by $\mathcal{A}_{P}$ the space of connections $\theta$ on the underlying smooth bundle $\underline{P}$ whose curvature $F_\theta$ satisfies $F_\theta^{0,2} = 0$ and whose $(0,1)$-part induces $P$. Given $\theta \in \mathcal{A}_{P}$, the Chern-Simons three-form $CS(\theta)$ is a $G$-invariant complex differential form of degree three on the total space of $\underline{P}$ defined by
\begin{equation*}\label{eq:CS}
CS(\theta) = -  \frac{1}{6}\la \theta\wedge [\theta,\theta]\ra  + \la F_\theta \wedge \theta\ra \in \Omega^3_\CC(\underline{P}),
\end{equation*}
which satisfies 
\begin{equation*}\label{eq:dCS}
d CS(\theta) = \la F_\theta \wedge F_\theta \ra.
\end{equation*}

\begin{proposition}[\cite{grt2}, Prop. 2.9]\label{lemma:deRhamC}
	The isomorphism classes of string algebroids are in one-to-one correspondence  with the set
	\begin{equation*}\label{eq:lescEind3}
	H^1(\cS) = \{(P,H,\theta): (H,\theta) \in \Omega^{3,0} \oplus \Omega^{2,1} \times \cA_{P} \; | \; dH + \la F_\theta \wedge F_\theta \ra = 0\}/ \sim,
	\end{equation*}
	where $(P,H,\theta)\sim (P',H',\theta')$ if there exists an isomorphism $g \colon P \to P'$ of holomorphic principal $G$-bundles and, for some $B\in \Omega^{2,0}$,
	\begin{equation}\label{eq:anomaly}
	H'  = H + CS(g\theta) - CS(\theta') - d\la g\theta \wedge \theta'\ra + dB.
	\end{equation}
\end{proposition}

\begin{remark}
	The notation $H^1(\cS)$ comes from the fact that the isomorphism classes can be understood as the first cohomology of a certain sheaf $\cS$ (see \cite[Sec. 3.1]{grt2} for more details). Implicitly, we shall use a smooth version of this sheaf (and its first cohomology) in Proposition \ref{lemma:deRhamCsmooth}.
\end{remark}

\begin{remark}\label{rem:CSexp}
Recall that given a pair of connections $\theta$, $\theta'$ on a smooth principal $G$-bundle $\underline P$ over $X$, there is an equality (see e.g. \cite{grt2})
	$$
	CS(\theta') - CS(\theta) - d\la \theta' \wedge \theta\ra = 2 \la a,F_\theta\ra + \la a, d^\theta a \ra + \frac{1}{3} \la a, [a,a] \ra \in \Omega^3_\CC
	$$
	where $a = \theta' - \theta$ is a smooth $1$-form with values in the adjoint bundle of $\underline P$. 
	By an abuse of notation, we omit the pullback of the right-hand side to the total space of $\underline P$.
\end{remark}

\subsection{Liftings}

Our next goal is to understand string algebroids in terms of smooth data. For this, we will extend the \emph{lifting plus reduction} method introduced in \cite{G2}. Our construction can be regarded as an 
analogue of the well-known construction of holomorphic vector bundles in terms of Dolbeault operators.

Let $X$ be a complex manifold. We denote by $\underline{X}$ the underlying smooth manifold. A smooth complex Courant algebroid $(E,\la\, , \ra,[\,,],\pi)$ over $X$ consists of a smooth complex vector bundle $E \to X$ together with a smooth non-degenerate symmetric bilinear form $\la\, ,\ra$, a smooth vector bundle morphism $\pi:E\to T\underline X\otimes \C$ and a bracket $[\, ,]$ on smooth sections 
satisfying the same axioms (D1)-(D5) as a holomorphic Courant algebroid (see Section \ref{sec:string}), replacing $\cO_X$ by the sheaf of smooth $\mathbb{C}$-valued functions on $X$.

We fix the data $G$, $\mathfrak{g}$, $\la\,,\ra$ as in the previous section. Let $\underline{P}$ be a smooth principal $G$-bundle over $X$ with
vanishing first Pontryagin class
$$
p_1(\underline{P}) = 0 \in H^4(X,\CC).
$$
We consider the Atiyah Lie algebroid $A_{\underline P}$, fitting into the short exact sequence of smooth complex Lie algebroids
\begin{equation}\label{eq:Atiyahseq}
0 \to \ad \underline{P} \to A_{\underline{P}} \to T\underline{X} \otimes \CC \to 0,
\end{equation}
where $T\underline{X} \otimes \CC$ denotes the complexified smooth tangent bundle of $X$. Recall that $A_{\underline P}$ is defined as the quotient of the complexification of the real Atiyah algebroid of $\underline{P}$, regarded as a principal bundle with real structure group, by the ideal $\ad P^{0,1}$, whereby $\ad \underline{P} \cong \ad P^{1,0}$ in \eqref{eq:Atiyahseq}.

\begin{definition}\label{def:Courantcx}
	A \emph{complex string algebroid} is a tuple $(E,\underline P,\rho_c)$, where $\underline P$ is a smooth principal $G$-bundle over $X$, the bundle $E$ is a smooth complex Courant algebroid over $X$, and $\rho_c$ is a bracket-preserving morphism inducing a short exact sequence
	\begin{equation*}\label{eq:defstringcomplex}
	\xymatrix{
		0 \ar[r] & T^*\underline X \otimes \CC \ar[r] & E \ar[r]^{\rho_c} & A_{\underline P} \ar[r] & 0,
	}
	\end{equation*}
	such that the induced map of complex Lie algebroids $\rho_c \colon A_E \to A_{\underline P}$ is an isomorphism restricting to an isomorphism $\ad_E \cong (\ad \underline P,\la\, ,\ra)$.
\end{definition}

Here, the notion of morphism is analogous to Definition \ref{def:stringholCourmor}, and it is therefore omitted. The basic device to produce a string algebroid out of a complex string algebroid is provided by the following definition.

\begin{definition}\label{def:lifting}
	Given $(E,\underline{P},\rho_c)$ a complex string algebroid, a lifting of $T^{0,1} X$ to $E$ is an isotropic, involutive subbundle $L \subset E$ mapping isomorphically to $T^{0,1} X$ under $\pi \colon E \to T\underline{X} \otimes \CC$.
\end{definition}

Our next result shows how to obtain a string algebroid from any lifting $L \subset E$. Given a smooth vector bundle $Q$ over X, we will denote by $\Gamma(Q)$ the corresponding vector space of global sections.

\begin{proposition}\label{prop:QL}
	Let $(E,\underline{P},\rho_c)$ be a complex string algebroid. Then, a lifting $L \subset E$ of $T^{0,1} X$ determines a string algebroid $(Q_L,P_L,\rho_L)$, with
	\begin{equation*}\label{eq:QL}
	Q_L = L^\perp / L
	\end{equation*}
	where $L^\perp$ denotes the orthogonal complement of $L \subset E$.
\end{proposition}

\begin{proof}
We will follow closely \cite[App. A]{G2}. The reduction of $E$ by $L$, defined as the quotient $Q_L=L^\perp/L$, inherits a non-degenerate pairing. Then, $Q_L$ inherits a structure of holomorphic vector bundle given by the Dolbeault operator
$$
\dbar^L_{V} u = [s(V),\tilde u] \qquad
\operatorname{mod} \;\; L,
$$
where $V \in \Gamma(T^{0,1}X)$, $\tilde u \in \Gamma(L^\perp)$ is any lift of $u \in \Gamma(Q_L)$ to $L^\perp$, and $s = \pi_{|L}^{-1} \colon T^{0,1}X \to L$. By construction, the surjective (anchor) map
\begin{align*}
		\pi_{Q_L} \colon Q_L & \to T^{1,0}X \cong TX \\ [u] & \mapsto \pi(u),
\end{align*}
for $\pi$ the anchor map of $E$, is holomorphic. Finally, on the sheaf of holomorphic sections of $Q_L$, we define a bracket via
$$
[u,v]_{Q_L} = [\tilde u,\tilde v] \qquad
\operatorname{mod} \;\; L.
$$
Using that $\dbar^L u  = \dbar^L v = 0$ and $[L,L] \subset L$, combined with the axioms of Courant algebroid for $E$, it is not difficult to see that $(Q_L,[,]_{Q_L},\pi_{Q_L})$ satisfies the axioms of a transitive holomorphic Courant algebroid over $X$. 

To endow $Q_L$ with the structure of a string algebroid, we note that the image of $\rho_{L} := \rho_{E|L^\perp}$ is an involutive subbundle of $A_{\underline{P}}$. This determines uniquely a $G$-invariant (integrable) almost complex structure on $\underline{P}$, such that $T^{1,0}P/G = \Im \; \rho_L$ and the induced map $A_{Q_L} \to T^{1,0}P/G$ is an isomorphism of holomorphic Lie algebroids.
\end{proof}

In the following lemma we observe that every string algebroid comes from reduction.

\begin{lemma}\label{lem:red-for-all-Q}
Let $(Q,P,\rho)$ be a string algebroid.

\begin{enumerate}[i)]

\item There is a structure of complex string algebroid with lifting on
\begin{equation*}\label{eq:EQ}
E_Q=Q\oplus T^{0,1} X \oplus (T^{0,1}X)^*, \qquad \qquad L=T^{0,1} X,
\end{equation*}
such that, for any $e,q \in \Gamma(Q)$, $V,W \in \Gamma(T^{0,1}X)$, $\xi,\eta \in \Omega^{0,1}$, the anchor map, the pairing, the bracket, and the bracket-preserving map are given respectively by
\begin{equation*}\label{eq:EQstructure}
\begin{split}
\pi(e + V + \xi) & := \pi(e) + V,\\
\la e + V + \xi, e + V + \xi\ra & := \la e,e\ra + \xi(V),\\
[e + V + \xi,q + W + \eta]  & := [e,q] + \dbar^{Q}_V q - \dbar^{Q}_W e + 2 \dbar \langle e,q \rangle - 2 \langle e,\dbar^{Q} q \rangle\\
& {} + L_{\pi(e)}\eta - i_{\pi(q)}d\xi + [V,W] + L_V \eta - i_W d\xi,\\
\rho(e + V + \xi) & := \rho(e) + \theta^{0,1}V,
\end{split}
\end{equation*}
where $\theta^{0,1}$ denotes the partial connection on $\underline{P}$ determined by the holomorphic principal bundle $P$ and $\pi$, $\la\,,\ra$, $[\,,]$ and $\rho$ in the right-hand side refer to the natural extensions to smooth sections of $Q$ (cf. Definition \ref{def:Q0}).
\item The reduced string algebroid $Q_L$ is canonically isomorphic to $Q$ via the map induced by the natural projection $L^\perp = Q \oplus T^{0,1} X \to Q$.

\end{enumerate}

\end{lemma}

\begin{proof}
A direct proof of $i)$ follows by a laborious but straightforward check using the axioms in Definition \ref{def:Courant} and it is omitted (see Remark \ref{rem:matchedpair} and Remark \ref{rem:red-for-all-Q} below for an alternative, shorter proof). Part $ii)$ follows easily from Proposition \ref{prop:QL}.
\end{proof}

\begin{remark}\label{rem:matchedpair}
	The construction of $E_Q$ in Lemma \ref{lem:red-for-all-Q} boils down to the fact that $Q$ forms a \emph{matched pair} with the standard Courant structure on $T^{0,1} X \oplus (T^{0,1} X)^*$ (cf. \cite{GrSt,G2}).
\end{remark}

To finish this section, we recall the classification of complex string algebroids. Given a smooth principal $G$-bundle, we denote by $\cA_{\underline{P}}$ the space of connections on $\underline{P}$. 

\begin{proposition}[{\cite[App. A]{grt2}}]\label{lemma:deRhamCsmooth}
	The isomorphism classes of complex string algebroids are in one-to-one correspondence  with the set
	\begin{equation}\label{eq:lescEind3smooth}
	H^1(\underline \cS) = \{(\underline P,H_c,\theta_c): (H_c,\theta_c) \in \Omega^3_\CC \times \cA_{\underline P} \; | \; dH_c + \la F_{\theta_c} \wedge F_{\theta_c} \ra = 0\}/ \sim,
	\end{equation}
	where $(\underline P,H_c,\theta_c) \sim (\underline P',H'_c,\theta'_c)$ if there exists an isomorphism $g \colon \underline P \to \underline P'$ of smooth principal $G$-bundles and \eqref{eq:anomaly} is satisfied for some $B\in \Omega^2_\CC$.
\end{proposition}

\subsection{Explicit models}

We describe now concrete models for string algebroids, either in the holomorphic or smooth categories, which will be used throughout the paper. We refer to \cite[Prop. 2.4]{grt2} for the fact that the model in the next definition satisfies the axioms in Definition \ref{def:Courant}.

\begin{definition}\label{def:Q0}
	For any triple $(P, H,\theta)$ as in Proposition \ref{lemma:deRhamC}, we denote by 
	\begin{equation*}\label{eq:Qexpb}
	Q_0=T^{1,0}X\oplus\ad P\oplus (T^{1,0} X)^*
	\end{equation*}
	the string algebroid with Dolbeault operator
	$$\dbar_0 (V + r + \xi)  = \dbar V + i_V F_\theta^{1,1} + \dbar^\theta r + \dbar \xi + i_{V}H^{2,1} + 2\la F_\theta^{1,1}, r\ra,$$ non-degenerate symmetric bilinear form, or pairing,
	$$\langle V + r + \xi , V + r + \xi \rangle_0  = \xi(V) + \la r, r\ra,$$
	bracket on holomorphic sections defined by
	\begin{equation*}
	\begin{split}
	[V+ r + \xi,W + t + \eta]_0   = {} & [V,W] - F^{2,0}_\theta(V,W) + \partial^\theta_V t - \partial^\theta_W r - [r,t]\\
	& {} + i_V \partial \eta + \partial (\eta(V)) - i_W\partial \xi + i_Vi_W H^{3,0},\\
	& {} + 2\la \partial^\theta r, t\ra + 2\la i_V F_\theta^{2,0}, t\ra - 2\la i_W F_\theta^{2,0}, r\ra,	
	\end{split}
	\end{equation*}
	anchor map $\pi_0(V+ r + \xi) = V$, and bracket-preserving map $\rho_0(V+ r + \xi) = V + r,$
	where we use the connection $\theta$ to identify $A_P \cong T^{1,0}X\oplus\ad P$.
\end{definition}

We turn next to the case of complex string algebroids. Since this case has not been considered previously in the literature, we give a few more details of the construction. Given a triple $(\underline P, H_c,\theta_c)$ as in Proposition \ref{lemma:deRhamCsmooth}, we can associate a complex string algebroid as follows: consider the smooth complex vector bundle
\begin{equation*}\label{eq:Eexp}
E_0 = (T\underline{X} \otimes \CC) \oplus \ad \underline{P} \oplus  (T^*\underline{X} \otimes \CC)
\end{equation*}
with the $\CC$-valued pairing 
\begin{equation}\label{eq:pairing3}
\la V + r + \xi, V + r + \xi\ra = \xi(V) + \la r,r \ra
\end{equation}
and anchor map $\pi(V + r + \xi)= V$. Endowed with the bracket
\begin{equation}\label{eq:bracket3}
\begin{split}
[V+ r + \xi,W + t + \eta]  = {} & [V,W] - F_{\theta_c}(V,W) + d^{\theta_c}_V t - d^{\theta_c}_W r - [r,t]\\
& + L_V \eta - i_W d\xi + i_Vi_W H_c\\
& + 2\la d^{\theta_c} r,t\ra  + 2\la i_V F_{\theta_c},t \ra - 2\la i_W F_{\theta_c},r\ra,
\end{split}
\end{equation}
the bundle $E_0$ becomes a smooth complex Courant algebroid (the Jacobi identity for the bracket is equivalent to the four-form equation \eqref{eq:lescEind3smooth} in Proposition \ref{lemma:deRhamCsmooth}). The connection $\theta_c$ gives a splitting of the Atiyah sequence \eqref{eq:Atiyahseq}, so that $A_{\underline P} \cong (T\underline{X} \otimes \CC) \oplus \ad \underline{P}$, and in this splitting the Lie bracket on sections of $A_{\underline P}$ is
$$
[V+ r,W + t]  = [V,W] - F_{\theta_c}(V,W) + d^{\theta_c}_V t - d^{\theta_c}_W r - [r,t].
$$
Then, one can readily check that 
\begin{equation}\label{eq:rho0}
\rho_0(V+ r + \xi) = V + r
\end{equation} 
defines a structure of complex string algebroid $(\underline P, E_0,\rho)$, in the sense of Definition \ref{def:Courantcx}, where we again use $\theta_c$ to identify $A_{\underline P} \cong (T\underline{X} \otimes \CC) \oplus \ad \underline{P}$.

\begin{definition}\label{def:E0}
	For any triple $(\underline P, H_c,\theta_c)$ as in \eqref{eq:lescEind3smooth}, we denote by 
	\begin{equation*}\label{eq:Eexpb}
	E_0 = (T\underline{X} \otimes \CC) \oplus \ad \underline{P} \oplus  (T^*\underline{X} \otimes \CC)
	\end{equation*} the complex string algebroid described by the pairing \eqref{eq:pairing3}, the bracket \eqref{eq:bracket3}, and the bracket-preserving map \eqref{eq:rho0}.
\end{definition}

\begin{remark}\label{rem:red-for-all-Q}
By using the explicit models $Q_0$ and $E_0$ in Definition \ref{def:Q0} and Definition \ref{def:E0}, combined with  Propositions \ref{lemma:deRhamC} and \ref{lemma:deRhamCsmooth}, one can obtain a short proof of Lemma \ref{lem:red-for-all-Q}.
\end{remark}

We next obtain explicit characterizations of liftings of $T^{0,1} X$ in terms of differential forms. Given $(\gamma,\beta) \in \Omega^{2}_\CC \oplus \Omega^{1}(\ad \underline P)$ we can define an orthogonal automorphism $(\gamma,\beta)$ of $E_0$  by (see \cite{grt2})
\begin{equation}\label{eq:BA}
(\gamma,\beta)(V + r + \xi) = V + i_V \beta + r + i_V \gamma - \la i_V\beta,\beta  \ra - 2\la \beta,r\ra + \xi.
\end{equation}

\begin{lemma}\label{lemma:liftings}
	Let $E_0$ be the complex string algebroid determined by a triple $(\underline P,H_c,\theta_c)$, as in Definition \ref{def:E0}. There is a one-to-one correspondence between liftings of $T^{0,1} X$ to $E_0$ and elements
	$$
	(\gamma,\beta) \in \Omega^{1,1 + 0,2} \oplus \Omega^{0,1}(\ad \underline P)
	$$
	satisfying 
	\begin{equation}\label{eq:liftingcond}
	\begin{split}
	\Big{(}H_c + d \gamma - 2 \la \beta,F_{\theta_c}\ra - \la \beta, d^{\theta_c} \beta \ra - \frac{1}{3} \la \beta, [\beta,\beta] \ra\Big{)}^{1,2 + 0,3} & = 0,\\
	F_{\theta_c}^{0,2} + \dbar^{\theta_c} \beta + \frac{1}{2}[\beta,\beta] & = 0.\\
	\end{split}
	\end{equation}
	More precisely, given $(\gamma,\beta)$ satisfying \eqref{eq:liftingcond}, the lifting is
	\begin{equation}\label{eq:L}
	L = \{(-\gamma,-\beta)(V^{0,1}), \; V^{0,1} \in T^{0,1} X\},
	\end{equation}
	and, conversely, any lifting is uniquely expressed in this way. 
\end{lemma}
\begin{proof}
	An isotropic subbundle $L \subset E_0$ mapping isomorphically to $T^{0,1} X$ under $\pi$ is necessarily of the form \eqref{eq:L} for a suitable $(\tilde \gamma,\tilde \beta) \in \Omega^{1,1 + 0,2} \oplus \Omega^1(\ad \underline P)$ (see \cite[Sec. 3.1]{GF}). Observe that, for any $ V^{0,1} \in T^{0,1} X$,
$$
(-\tilde \gamma,-\tilde \beta)(V^{0,1}) = (-\gamma,-\beta)(V^{0,1}),
$$
where $\beta = \tilde \beta^{0,1}$ and $\gamma = \tilde \gamma + \la \tilde{\beta}^{0,1} \wedge \tilde \beta^{1,0} \ra$, and the pair 
$$
(\gamma,\beta) \in \Omega^{1,1 + 0,2} \oplus \Omega^{0,1}(\ad \underline P)
$$ 
is uniquely determined by $L$. By the proof of \cite[Prop. 4.3]{grt} we have
	\begin{equation}\label{eq:BAshift}
	(\gamma,\beta)[(-\gamma,-\beta) \cdot ,(-\gamma,-\beta)\cdot ]_{\theta_c, H_c} = [\cdot,\cdot]_{\theta_c + \beta,H'_c}
	\end{equation}
	where $[\cdot,\cdot]_{\theta_c, H_c}$ denotes the Dorfman bracket \eqref{eq:bracket3} and
	\begin{equation}\label{eq:Hc'}
	H'_c = H_c + d\gamma - 2\la \beta, F_{\theta_c} \ra - \la \beta, d^{\theta_c} \beta \ra - \frac{1}{3}\la \beta, [\beta,\beta] \ra.
	\end{equation}
	Then, by formula \eqref{eq:bracket3} for the bracket, $L$ is involutive if and only if
	\begin{equation}\label{eq:liftingcond2}
	F_{\theta_c + \beta}^{0,2} = F_{\theta_c}^{0,2} + \dbar^{\theta_c} \beta + \frac{1}{2}[\beta,\beta] = 0, \qquad (H'_c)^{1,2 + 0,3}= 0,
	\end{equation}
	and the proof follows.
\end{proof}

We describe now the isomorphism class of the reduced string algebroid in Proposition \ref{prop:QL}, in terms of the explicit model in the previous lemma.

\begin{proposition}\label{prop:QLexp}
	Let $E_0$ be the complex string algebroid determined by a triple $(\underline P,H_c,\theta_c)$, as in Definition \ref{def:E0}. If $L=(-\gamma,-\beta)T^{0,1} X$, as in \eqref{eq:L}, then the isomorphism class of $(Q_L,P_L,\rho_L)$ is (see Proposition \ref{lemma:deRhamC})
	\begin{equation}\label{eq:QLisom}
	[(P_L,H_c^{3,0 + 2,1} + \partial \gamma^{1,1} - 2\la \beta, F_{\theta_c}^{2,0} \ra ,\theta_c + \beta)] \in H^1(\cS),
	\end{equation}
	where $P_L$ denotes $\underline{P}$ endowed with the holomophic structure $\theta^{0,1}_c+\beta$.
\end{proposition}

\begin{proof} 
	By the second equation in \eqref{eq:liftingcond} it follows that $\theta^{0,1}_c+\beta$ induces the structure of a holomorphic principal bundle on $\underline{P}$, called $P_L$. 
	Now, we have 
	\begin{equation*}\label{eq:Lperp}
	L^\perp = \{(-\gamma,-\beta)(W + t + \eta^{1,0}) \st \; W \in T\underline{X} \otimes \CC, \; t \in \ad \underline{P}, \; \eta^{1,0} \in (T^{1,0} X)^* \}
	\end{equation*}
	and therefore there is a smooth bundle isomorphism
	\begin{align}\label{eq:QLstd}
	Q_L &\to T^{1,0} X\oplus \ad \underline{P} \oplus (T^{1,0} X)^* \nonumber \\
	[(-\gamma,-\beta)(W + t + \eta^{1,0})] &\mapsto W^{1,0} + t + \eta^{1,0}.
	\end{align}
	Let us now express the holomorphic Courant structure in terms of \eqref{eq:QLstd}. Firstly, note that (see \eqref{eq:BAshift})
	\begin{align*}
	(\gamma,\beta)[(-\gamma,-\beta)(V^{0,1}),(-\gamma,-\beta)(W&^{1,0} + t + \eta^{1,0})]  \\
	{} = {} & \dbar_{V^{0,1}} W^{1,0} - F_{\theta_c'}(V^{0,1},W^{1,0}) + \dbar^{\theta_c'}_{V^{0,1}} t \\
	& + \dbar_{V^{0,1}} \eta^{1,0} + i_{V^{0,1}} i_{W^{1,0}} H'_c + 2\la i_{V^{0,1}} F_{\theta_c'},t \ra,
	\end{align*} 
	where $H_c'$ is as in \eqref{eq:Hc'} and $\theta_c' = \theta_c + \beta$. Since $L$ is involutive, we have $(H_c')^{1,2 + 0,3} = 0$ (see \eqref{eq:liftingcond2}), and
	\begin{align*}
	\dbar^L(W^{1,0} + t + \eta^{1,0}) = {} & \dbar W^{1,0} + i_{W^{1,0}}F_{\theta_c'}^{1,1} + \dbar^{\theta_c'} t\\
	& + \dbar\eta^{1,0} + i_{W^{1,0}} (H_c'^{2,1}) + 2\la F_{\theta_c'}^{1,1},t \ra.
	\end{align*} 
	Therefore, using the connection $\theta_c'$ to identify
	$$
	A_{P_L} = T^{1,0} X \oplus \ad \underline{P}
	$$
	it follows that
	\begin{align*}
	\rho_L\colon Q_L &\to A_{P_L}\\ [(-\gamma,-\beta)(W + t + \eta^{1,0})] &\mapsto W^{1,0} + t
	\end{align*}
	is holomorphic, and hence $Q_L$ is a string algebroid. 
	To finish, arguing as for the Dolbeault operator, we notice that, in terms of \eqref{eq:QLstd}, the bracket of $Q_L$ is given by
	\begin{equation*}\label{eq:bracketQL}
	\begin{split}
	[V+ r + \xi,W + t + \eta]  = {} & [V,W] - F_{\theta_c'}(V,W) + \partial^{\theta_c'}_V t - \partial^{\theta_c'}_W r - [r,t]\\
	& + \partial (i_V \eta) + i_V\partial \eta  - i_W\partial \xi + i_Vi_W H_c'^{3,0}\\
	& + 2\la \partial^{\theta_c'} r,t\ra  + 2\la i_V F_{\theta_c'}^{2,0},t \ra - 2\la i_W F_{\theta_c'}^{2,0},r\ra,
	\end{split}
	\end{equation*}
	for $V+ r + \xi,W + t + \eta$ holomorphic sections of $T^{1,0} X\oplus \ad \underline{P} \oplus (T^{1,0} X)^*$. Then, by \cite[Prop. 2.4]{grt2} it follows that the isomorphism class of $(Q_L,P_L,\rho_L)$ is \eqref{eq:QLisom}, as claimed.
\end{proof}

\section{Complex gauge group}\label{sec:Morita}

\subsection{Classification of reduction diagrams}

We introduce next the complex gauge group of our theory as a suitable subgroup of automorphisms of a complex string algebroid. For this, in the present section we show that, essentially, a holomorphic string algebroid can be obtained by reduction in a unique way (see Proposition \ref{prop:QL}). This will lead us to a canonical notion of complex gauge group, which will be introduced in Section \ref{sec:Ham}. We follow the notation in Proposition \ref{prop:QL}. For simplicity, when it is clear from the context, we will denote a complex string algebroid $(E,\underline{P},\rho_c)$ (resp. string algebroid $(Q,P,\rho)$) over a complex manifold $X$ simply by $E$ (resp. $Q$). We fix the structure group of all our principal bundles (either smooth or holomorphic) to be a complex Lie group $G$.

\begin{definition}\label{def:Moritabrick}
	Let $Q$ be a string algebroid. A \emph{reduction diagram} for  $Q$ is a tuple $(E,L,\psi)$, given by a complex string algebroid $E$, a lifting $L \subset E$ of $T^{0,1} X$ and a string algebroid isomorphism $\psi \colon Q_L \to Q$, fitting in a diagram
\begin{equation}\label{eq:Morita}
	\xymatrix{
E  \ar@{-->}[dr] & & \\
 & Q_{L} \ar[r]^{\psi} & Q,
	}
\end{equation}
where the discontinuous arrows refer to the partial map $ L^\perp\to Q_L=L^\perp/L$.
	
\end{definition}

By Lemma \ref{lem:red-for-all-Q}, there always exists a reduction diagram (from now on, simply a \emph{diagram}) for a given string algebroid $Q$, given by
\begin{equation}\label{eq:Brickmodel}
E = E_Q, \qquad L= T^{0,1}X, \qquad \psi = \Id_Q.
\end{equation}
Here $\Id_Q$ denotes the isomorphism $Q_L := L^\perp/L \to Q$ induced by the natural projection $L^\perp = Q \oplus T^{0,1}X \to Q$. Furthermore, as we will see shortly, this is essentially the unique diagram, up to the right notion of isomorphism. 

 To introduce the following definition, observe that given pairs $(E,L)$ and $(E',L')$, an isomorphism 
$$
\underline f \colon E \to E' \qquad \textrm{such that} \qquad \underline f(L) = L'
$$
induces, upon restriction to $L^\perp$, an isomorphism $f \colon Q_L \to Q_{L'}$ and a commutative diagram
\begin{equation}\label{eq:Morita2.0}
	\xymatrix{
		E \ar@{-->}[r] \ar[d]_{\underline f} &  Q_{L} \ar[d]_{f} \\
		E' \ar@{-->}[r] & Q_{L'}.
	}
\end{equation}

\begin{definition}\label{def:brickiso}
	We say that two diagrams $(E,L,\psi)$, $(E',L',\psi')$ for the same string algebroid $Q$ are isomorphic if there exists an isomorphism $\underline f:E\to E'$ such that $\underline{f}(L)=L'$, thus inducing an isomorphism $f:Q_L\to Q_{L'}$, and $\psi=\psi'\circ f$. That is, the following diagram commutes
	\begin{equation*}
	\xymatrix @R=.5pc {
		E \ar@{-->}[dr] \ar[dd]_{\underline f} & & \\
		& Q_{L} \ar[rd]^{\psi} \ar[dd]_{f} & \\ 
		E' \ar@{-->}[dr] & & Q. \\
		& Q_{L'} \ar[ru]_{\psi'}  & 
	}
	\end{equation*}
\end{definition}

In the following result, we observe that there is a natural forgetful map from the set of isomorphism classes of string algebroids $H^1(\cS)$ to the set of isomorphism classes of complex string algebroids $H^1(\underline \cS)$ (see Proposition \ref{lemma:deRhamC} and Proposition \ref{lemma:deRhamCsmooth}). This provides a lift of the map which sends a holomorphic principal $G$-bundle $P$ to the underlying smooth principal bundle $\underline{P}$ (see Remark \ref{rem:smap}).

\begin{lemma}\label{lem:smap}
Using the notation in Lemma \ref{lem:red-for-all-Q}, there is a well-defined map
\begin{align}\label{eq:forgetful}
s \colon  H^1(\cS) &\to H^1(\underline \cS) \nonumber\\ [Q]  & \mapsto [E_Q].
\end{align}
\end{lemma}

\begin{proof}
Given an isomorphism $f \colon Q \to Q'$, we can define an induced isomorphism of complex string algebroids
$$
\underline f : = f \oplus \Id_{T^{0,1} X} \oplus \Id_{(T^{0,1} X)^*} \colon E_Q \to E_{Q'}.
$$
\end{proof}

\begin{remark}\label{rem:smap}
Alternatively, relying on the classification in Proposition \ref{lemma:deRhamC} and Proposition \ref{lemma:deRhamCsmooth}, we can also write \eqref{eq:forgetful} as
$$
s([(P,H,\theta)]) = [(\underline P, H,\theta)],
$$
where $\underline P$ denotes the smooth complex principal $G$-bundle underlying $P$.
\end{remark}

We are now ready to prove the uniqueness of diagrams \eqref{eq:Morita} up to (unique) isomorphism.

\begin{lemma}\label{lemma:bricks}
Let $Q$ be a string algebroid. Given a diagram $(E,L,\psi)$ for $Q$, there exists a unique isomorphism 
	\begin{equation}\label{eq:Moritab}
	\xymatrix @R=.5pc {
		E \ar@{-->}[dr] \ar[dd]_{\underline f} & & \\
		& Q_{L} \ar[rd]^{\psi} \ar[dd]_{f} & \\ 
		E_Q \ar@{-->}[dr] & & Q. \\
		& Q \ar[ru]_{\Id_Q}  & 
	}
	\end{equation}
to the diagram $(E_Q,T^{0,1}X,\Id_Q)$ in \eqref{eq:Brickmodel}. Consequently, any diagram $(E,L,\psi)$ for $Q$ satisfies, for the map $s$ in \eqref{eq:forgetful},
$$
[E] = s([Q]) \in H^1(\underline \cS).
$$
\end{lemma}

\begin{proof} The isotropic splitting $L\subset E$ gives a decomposition of $E$ into $L^\perp$, which contains $L$, and $(T^{0,1}X)^*\subset E$. Combining this with $L\cong T^{0,1}X$, the definition of $Q_L$ and $\psi$, we get
$$E= L^\perp \oplus (T^{0,1} X)^*\cong  Q_L \oplus T^{0,1} X \oplus (T^{0,1} X)^*\cong_{\psi} Q \oplus T^{0,1} X \oplus (T^{0,1} X)^*=E_Q.$$
This is an isomorphism of Courant algebroids with the Courant algebroid structure given in Lemma \ref{lem:red-for-all-Q}. This isomorphism tautologically sends $L$ to $T^{0,1}X$, and induces a map from $Q_L$ to $Q$ which makes the diagram \eqref{eq:Moritab} commutative.

The uniqueness follows from the fact that the first isomorphism above is the only one that sends $L$ to $T^{0,1} X$ via projection, and the second one is induced by $\psi$. Finally, the last statement follows from the fact that $[E]=[E_Q]=s([Q])$, as defined in Lemma \ref{lem:smap}.
\end{proof}

Building on Lemma \ref{lemma:bricks}, we show next that
a pair of diagrams for $Q$ admit a unique isomorphism.

\begin{lemma}\label{lemma:bricksglue}
Let $Q$ be a string algebroid and a pair of diagrams $(E_1,L_1,\psi_1)$ and $(E_2,L_2,\psi_2)$ for $Q$. Then, there exists a unique isomorphism $\underline f \colon E_1 \to E_2$ such that $\underline{f}(L_1) = L_2$ making the following diagram commutative
\begin{equation*}\label{eq:gluingbricks2OLD}
  \xymatrix{
\ar@{-->}[dr]  E_1 \ar[rrrr]^{\underline f} & &  & & E_2.  \ar@{-->}[dl]\\\
 & Q_{L_1} \ar[r]^{\psi_1}  \ar@/^2pc/[rr]^{f} & Q  & \ar[l]_{\psi_2} Q_{L_2}  & \\\
  }
\end{equation*}
\end{lemma}

\begin{proof}
The statement follows as a direct consequence of Lemma \ref{lemma:bricks}. 
\end{proof}


\subsection{Automorphisms}\label{sec:picard}

Let $(Q,P,\rho)$ be a string algebroid with complex structure group $G$ over a complex manifold $X$. As usual, $(Q,P,\rho)$ will be denoted simply by $Q$. Let $\underline{P}$ be the smooth $G$-bundle underlying $P$, and let $\cG_{\underline{P}}$ be the corresponding gauge group. In this section we study a group of gauge symmetries canonically associated to $Q$. By Lemma \ref{lem:red-for-all-Q}, it is natural to consider the group of automorphisms of $E_Q$. This is an infinite-dimensional complex Lie group with a natural action on liftings of $T^{0,1}X$, which will be used for the definition of the complex gauge group of our theory in the next section.

Let $Q$ be a string algebroid and denote by $\Aut(E_Q)$ the group of automorphisms of the complex string algebroid $E_Q$ in Lemma \ref{lem:red-for-all-Q}. Recall from \cite[App. A]{grt2} that there is a group homomorphism
$$
\sigma_{\underline{P}} \colon \cG_{\underline{P}} \to H^3(X,\mathbb{C}),
$$
defined by
$$
\sigma_{\underline{P}}(g) = [CS(g \theta_c) - CS(\theta_c) - d \la g \theta_c \wedge \theta_c\ra] \in H^3(X,\mathbb{C})
$$
for any choice of connection $\theta_c$ on $\underline{P}$. This defines a short exact sequence of groups (cf. \cite[Prop. 2.12]{grt2})
\begin{equation*}
  \xymatrix{
0 \ar[r] & \Omega^2_{\CC,cl} \ar[r] & \Aut(E_Q) \ar[r] &  \ar[r] \Ker \sigma_{\underline{P}} & 1,
  }
\end{equation*}
where $\Omega^2_{\CC,cl}$ is the additive group of closed complex $2$-forms on $X$. The proof of the next result is immediate.

\begin{corollary}\label{cor:PicQ}
There is a canonical exact sequence
\begin{equation}\label{eq:Picardses}
  \xymatrix{
0 \ar[r] & \Omega^2_{\CC,cl} \ar[r] & \Aut(E_Q) \ar[r] &  \ar[r] \Ker \sigma_{\underline{P}} &  \ar[r]^{\sigma_{\underline{P}} \qquad} \cG_{\underline{P}} & H^3(X,\CC).
  }
\end{equation}
\end{corollary}

To obtain a more explicit description of $\Aut(E_Q)$, we choose a representative $[(P,H,\theta)] = [Q] \in H^1(\cS)$ and consider the model $Q_0 \cong Q$ in Definition \ref{def:Q0}. Then, we have an identification (see Lemma \ref{lem:red-for-all-Q})
$$
E_{Q_0} = E_0,
$$
for the complex string algebroid $E_0$ determined by $(\underline P,H,\theta)$ (see Definition \ref{def:E0}). Relying on \cite[Cor. 4.2]{grt}---which characterizes $\Aut(E_0)$ in terms of differential forms (cf. \cite[Lem. 2.10]{grt2})---, we obtain the following result.

\begin{lemma}\label{lem:HomQQ}
Let $Q_0$ be given by $(P,H,\theta)$. There is a canonical bijection between $\Aut(E_0)$ and the set of pairs $(g,\tau) \in \cG_{\underline{P}} \times \Omega^2_\CC$ satisfying
\begin{equation}\label{eq:MoritaPicexp}
d \tau = CS(g^{-1}\theta) - CS(\theta) - d \la g^{-1} \theta \wedge \theta \ra,
\end{equation}
where $(g,\tau)$ acts on $V + r + \xi \in E_0$ by
\begin{equation}\label{eq:gtauaction}
(g,\tau) \cdot (V + r + \xi) = V + g(r + i_V a^g) + \xi + i_V \tau - \la i_V a^g,a^g \ra - 2\la a^g,r \ra
\end{equation}
for $a^g := g^{-1}\theta - \theta$. Via this bijection, the group structure on $\Aut(E_0)$ reads
\begin{equation*}\label{eq:productPic}
(g,\tau)(g',\tau') = (gg',\tau + \tau' + \la g'^{-1}a^g \wedge a^{g'} \ra).
\end{equation*}
\end{lemma}

The following result---characterizing the Lie algebra of $\Aut(E_0)$---has been stated in \cite{grt,grt2} without a proof. As it is key for our development in Section \ref{sec:Ham}, we include a detailed proof here. We follow the notation in Lemma~\ref{lem:HomQQ}. 

\begin{lemma}\label{lem:LiePic}
Let $Q_0$ be given by $(P,H,\theta)$. There is a canonical bijection
\begin{equation*}\label{eq:LiePic}
\Lie \Aut(E_0) = \{(s,B) \; | \; d(B - 2\la s, F_\theta \ra) = 0 \} \subset \Omega^0(\ad \underline{P}) \times \Omega^2_\CC.
\end{equation*}
Via this bijection, the adjoint action of $ \Aut(E_0)$ reads
\begin{equation}\label{eq:AdPic}
(g,\tau)(s,B) = (gs,B - \la a^g \wedge [s,a^g]\ra - 2 \la d^\theta s \wedge a^g \ra),
\end{equation}
for any $(g,\tau) \in  \Aut(E_0)$, and the Lie bracket structure is
\begin{equation}\label{eq:bracket}
[(s_0,B_0),(s_1,B_1)] = ([s_0,s_1],2 \la d^\theta s_0 \wedge d^\theta s_1 \ra).
\end{equation}
\end{lemma}
\begin{proof}
Let $(g_t,\tau_t)$ be a one-parameter family in $ \Aut(E_0)$ with $(g_0,\tau_0) = (\Id_{\underline{P}},0)$. Set $a_t = a^{g_t}$, and note that $(\dot a_t)_{|t = 0} = d^\theta s$. Taking derivatives in \eqref{eq:MoritaPicexp} at $t = 0$, it follows that 
\begin{equation}\label{eq:sB}
(s,B) := (\dot g_t, \dot \tau_t)_{|t=0} \in \Omega^0(\ad \underline{P}) \times \Omega^2_\CC
\end{equation}
satisfies 
\begin{equation}\label{eq:dBF}
d(B - 2 \la s, F_\theta \ra) = 0
\end{equation}
(see Remark \ref{rem:CSexp}). Conversely, given $(s,B) \in \Omega^0(\ad \underline{P}) \times \Omega^2_\CC$ satisfying \eqref{eq:dBF}, we define 
$$
(g_t,\tau_t) \in \cG_{\underline{P}} \times \Omega^2_\CC
$$
by $g_t = e^{ts}$ and $\tau_t = t(B - 2 \la s, F_{\theta}\ra) + \mu_t$, where
$$
\mu_t = \int_0^t (2 \la s, F_{\theta_u} \ra + \la a_u \wedge d^{\theta_u} s \ra) du
$$
and $\theta_t = g_t^{-1}\theta$. Notice that $(\dot \tau_t)_{|t = 0} = B$, as required. Setting
$$
C_t := CS(\theta_t) - CS(\theta) - d\la \theta_t \wedge \theta \ra,
$$
we have (see \cite[Lem. 3.23]{grst})
$$
\dot C_t = 2 \la d^{\theta_t} s, F_{\theta_t} \ra + d \la a_t \wedge d^{\theta_t} s \ra = d(2 \la s, F_{\theta_t} \ra + \la a_t \wedge d^{\theta_t} s \ra),
$$
and therefore
$$
d \dot \tau_t - \dot C_t = d \dot \mu_t - \dot C_t = 0.
$$
From $\tau_0 = 0 = C_0$ it follows that $(g_t,\tau_t) \in \Aut(E_0)$ for all $t$.

We prove next formula \eqref{eq:AdPic} for the adjoint action. For $(g_j,\tau_j) \in  \Aut(E_0)$, with $j = 0,1$, denote $a_j := g_j^{-1}\theta - \theta$. Using that
$$
a^{g_0g_1} = g_1^{-1}g_0^{-1} \theta - \theta = g_1^{-1}a_0 + a_1,
$$
we obtain
\begin{align*}
 (g_0,\tau_0) (g_1,\tau_1) (g_0,\tau_0)^{-1} & =  (g_0,\tau_0) (g_1,\tau_1)(g_0^{-1},-\tau_0)\\
 & = (g_0 g_1g_0^{-1},\tau_1 + \la g_1^{-1}a_0 \wedge a_1\ra + \la g_0 a^{g_0g_1} \wedge a^{g_0^{-1}} \ra )\\
 & = (g_0 g_1g_0^{-1},\tau_1 + \la a_0 \wedge g_1^{-1}a_0\ra + \la a^{g_1^{-1}} \wedge a_0\ra + \la a_0 \wedge a_1\ra ).
\end{align*}
Assume now that $(g^1_t,\tau^1_t)$ is a one-parameter family of elements in $\Aut(E_0)$ with $(g^1_0,\tau^1_0)=(g_1,\tau_1)$, and define $(s_1,B_1)$ as in \eqref{eq:sB}. Taking derivatives in the previous expression it follows that
$$
(g_0,\tau_0)(s_1,B_1) = (g_0s_1, B_1 - \la a_0 \wedge [s_1, a_0] \ra - \la d^\theta s_1 \wedge a_0 \ra + \la a_0 \wedge d^\theta s_1 \ra),
$$
as claimed in  \eqref{eq:AdPic}.

Finally, assume that $(g^0_t,\tau^0_t)$ is a one-parameter family of elements in $ \Aut(E_0)$ with $(g^0_0,\tau^0_0)=(g_0,\tau_0)$ and define $(s_0,B_0)$ as in \eqref{eq:sB}. By taking derivatives in the last formula we have
$$
[(s_0,B_0),(s_1,B_1)] = ([s_0,s_1], - 2 \la d^\theta s_1 \wedge d^\theta s_0 \ra),
$$
which proves \eqref{eq:bracket}. 
\end{proof}

To finish, we observe from the first part of the proof of Lemma \ref{lem:LiePic} that the differential of $\sigma_{\underline{P}}$ in \eqref{eq:Picardses} applied to $s \in \Lie \; \cG_{\underline{P}} = \Omega^0(\ad \underline{P})$ vanishes identically, $d \sigma_{\underline{P}}(s) = - [d\la s, F_{\theta_c}\ra] = 0$. Therefore, at the infinitesimal level \eqref{eq:Picardses} induces a short exact sequence
\begin{equation}\label{eq:PicardsesLie}
  \xymatrix{
0 \ar[r] & \Omega^2_{\CC,cl} \ar[r] & \Lie \;  \Aut(E_Q) \ar[r] &  \ar[r]  \Omega^0(\ad \underline{P}) &  0.
  }
\end{equation}

\subsection{Hamiltonian automorphisms}\label{sec:Ham}

In this section we define a normal subgroup
$$
\Aut_{dR}(E_Q) \subset  \Aut(E_Q)
$$
by means of the de Rham cohomology of the complex manifold $X$, which is the key to our moment map picture in Section \ref{sec:mmapmod}. In the literature about Courant algebroids, elements of the group $\Aut_{dR}(E_Q)$ receive the name of \emph{inner symmetries} (see e.g. \cite{grt}), but we shall take a more fundamental approach inspired by symplectic geometry. To fix ideas, we shall think of $\Aut(E_Q)$ as an analogue of the group of symplectomorphisms of a complex symplectic manifold, while the elements in $\Aut_{dR}(E_Q)$ will play the role of complex Hamiltonian symplectomorphisms.

Our first goal is to define a Lie algebra homomorphism
\begin{equation*}\label{eq:aepplimap}
\mathbf{d} \colon \Lie \Aut(E_Q) \to H^{2}(X,\CC),
\end{equation*}
where the de Rham cohomology group $H^{2}(X,\CC)$ is regarded as an abelian Lie algebra. For this, notice that for any choice of representative $[(P,H,\theta)] = [Q] \in H^1(\cS)$ and isomorphism $Q \cong Q_0$, Lemma \ref{lem:LiePic} implies that there is a natural map
\begin{align}\label{eq:aepplimap0}
\begin{split}
\mathbf{d}_0 \colon \Lie \Aut(E_0) & \to H^{2}(X,\CC)\\
(s,B) & \mapsto [B - 2\la s,F_\theta \ra]. 
\end{split}
\end{align}

\begin{lemma}\label{lem:PicAeppli}
There is a canonical linear map
\begin{equation}\label{eq:aepplimap2}
\mathbf{d} \colon \Lie \Aut(E_Q) \to H^{2}(X,\CC),
\end{equation}
which is invariant under the adjoint action of $\Aut(E_Q)$. In particular, \eqref{eq:aepplimap2} is a Lie algebra homomorphism and there is a normal Lie subalgebra 
$$
\Ker \mathbf{d} \subset \Lie \Aut(E_Q).
$$ 
Moreover, for any choice of representative $[(P,H,\theta)] = [Q] \in H^1(\cS)$ and isomorphism $Q \cong Q_0$, the induced homomorphism $\mathbf{d}_0$ coincides with \eqref{eq:aepplimap0}.
\end{lemma}

\begin{proof}
Let $\underline{f} \in  \Aut(E_Q)$. Given an isomorphism $\psi \colon Q \to Q_0$ (for a choice of representative $(P,H,\theta)$ of $[Q] \in H^1(\cS)$), arguing as in the proof of Lemma \ref{lem:smap} we obtain an isomorphism
$$
\underline \psi : = \psi \oplus \Id_{T^{0,1} X} \oplus \Id_{(T^{0,1} X)^*} \colon E_Q \to E_0
$$
inducing an identification $\Aut(E_Q) \cong \Aut(E_0)$. Thus, by Lemma \ref{lem:LiePic}, an element $\zeta \in \Lie \Aut(E_Q)$ determines uniquely a pair $(s,B) \in \Lie \Aut(E_0)$.
Then, we define
$$
\mathbf{d}(\zeta) = \mathbf{d}_0(s,B) = [B - 2\la s,F_\theta \ra] \in H^{2}(X,\CC).
$$
To check that $\mathbf{d}$ is invariant under the adjoint $\Aut(E_Q)$-action, it is enough to check that $\mathbf{d}_0$ is invariant under the adjoint $\Aut(E_0)$-action. Following Lemma \ref{lem:LiePic}, we define a closed complex two-form
$$
D := B - \la a^g \wedge [s,a^g]\ra - 2 \la d^\theta s \wedge a^g \ra - 2\la gs, F_\theta \ra,
$$
so that $[D] = \mathbf{d}_0((g,\tau)(s,B)) \in H^{2}(X,\CC)$, and calculate
\begin{equation}\label{eq:Picinvariance}
\begin{split}
D & = B + \la [a^g,a^g],s\ra - 2d\la s , a^g \ra + 2\la s , d^\theta a^g \ra - 2\la s, F_{g^{-1}\theta} \ra\\
& = B - 2\la s, F_\theta \ra - 2d\la s, a^g\ra,
\end{split}
\end{equation}
which proves the invariance of $\mathbf{d}_0$.
Here we have used the invariance of the pairing $\la\,,\ra$ combined with
$$
g^{-1}F_\theta = F_{g^{-1}\theta} = F_\theta + d^\theta a^g + \frac{1}{2}[a^g,a^g].
$$
Checking that \eqref{eq:aepplimap2} is independent of choices is left as an exercise.

\end{proof}

We are now ready to define the normal subgroup $\Aut_{dR}(E_Q) \subset \Aut(E_Q)$. Let $\Aut_0(E_Q)$ denote the component of the identity $\Id_{E_Q}$ in $\Aut(E_Q)$. Given an element $\underline{f} \in \Aut_0(E_Q)$ and a smooth family $\underline{f}_t \in \Aut(E_Q)$ such that $\underline{f}_0 = \Id_{E_Q}$ and $\underline{f}_1 = \underline{f}$, there exists a unique family $\zeta_t \in \Lie \Aut(E_Q)$ such that
$$
\frac{d}{dt} \underline{f}_t = \zeta_t \circ \underline{f}_t.
$$
Here, we regard $\zeta_t$ as a vector field on the total space of $E_Q$.

\begin{definition}\label{def:PicA}
Define $\Aut_{dR}(E_Q) \subset \Aut(E_Q)$ as the set of elements $\underline{f} \in \Aut_0(E_Q)$ such that there exists a smooth family $\underline{f}_t \in \Aut(E_Q)$ with $t \in [0,1]$, satisfying $\underline{f}_0 = \Id_{E_Q}$, $\underline{f}_1 = \underline{f}$, and 
\begin{equation}\label{eq:Hamcond}
\mathbf{d}(\zeta_t) = 0, \quad \textrm{ for all $t$.}
\end{equation}
\end{definition}

By analogy with symplectic geometry, a family $\underline{f}_t \in \Aut(E_Q)$ satisfying \eqref{eq:Hamcond} will be called a \emph{Hamiltonian isotopy} on $\Aut(E_Q)$. Notice that any smooth family $\zeta_t \in \Lie \Aut(E_Q)$ satisfying \eqref{eq:Hamcond} generates a Hamiltonian isotopy. 

\begin{proposition}\label{prop:PicA}
The subset $\Aut_{dR}(E_Q) \subset \Aut(E_Q)$ defines a normal subgroup of $\Aut(E_Q)$ with Lie algebra $\Ker \mathbf{d}$.
\end{proposition}

\begin{proof}
The proof is a formality, following Lemma \ref{lem:PicAeppli} and \cite[Prop. 10.2]{McS}. 
\end{proof}

\begin{remark}\label{rem:flux}
By analogy with symplectic geometry, it is natural to consider a notion of \emph{flux homomorphism} on the universal cover of $\Aut(E_Q)$ (see \cite[Sec. 10.2]{McS}). We leave this interesting perspective for future work.
\end{remark}

\begin{remark}\label{rem:PicAeppli}
A different normal subgroup $\Aut_A(E_Q) \subset \Aut(E_Q)$ associated to the Aeppli cohomology group $H^{1,1}_A(X)$ will be considered in Appendix \ref{sec:moduliapp}.
\end{remark}

\section{The Chern correspondence}

\subsection{Background on Bott-Chern theory}\label{sec:BCback}

The goal of this section is to prove an analogue of the classical Chern correspondence
in the context of string algebroids. We first recall some background about Bott-Chern theory which we will need.

Let $G$ be a complex reductive Lie group. Let $P$ be a holomorphic principal $G$-bundle over a complex manifold $X$. We fix a maximal compact subgroup $K \subset G$, and an invariant non-degenerate pairing $\la\,, \ra$ on the Lie algebra $\mathfrak{g}$ of $G$.  We will assume that it satisfies the reality condition 
\begin{equation*}\label{eq:creal}
\la \mathfrak{k}\otimes \mathfrak{k}\ra \subset \R
\end{equation*}
for the Lie algebra $\mathfrak{k} \subset \mathfrak{g}$ of $K$. Given a reduction $h \in \Omega^0(P/K)$ of $P$ to $K$, there is a uniquely defined Chern connection $\theta^h$, whose curvature $F_h := F_{\theta^h}$ satisfies
\begin{equation*}\label{eq:F02}
F_h^{0,2} = F_h^{2,0} = 0.
\end{equation*}
We denote by $P_h \subset P$ the corresponding principal $K$-bundle.

The following result considers secondary characteristic classes introduced by Bott and Chern \cite{BottChern} (see also \cite{BGS,Don}). We denote by $\Omega^{1,1}_\RR$ the space of real $(1,1)$-forms on $X$. 

\begin{proposition}[\cite{BGS,Don}]\label{prop:Donaldson}
For any pair of reductions $h_0,h_1 \in \Omega^0(P/K)$ there is a secondary characteristic class
\begin{equation}\label{eq:BCinvariant}
R(h_1,h_0) \in \Omega^{1,1}_\R/\operatorname{Im}(\partial \oplus \dbar)
\end{equation}
with the following properties:
\begin{enumerate}

\item $R(h_0,h_0) = 0$, and, for any third reduction $h_2$,
$$
R(h_2,h_0) = R(h_2,h_1) + R(h_1,h_0),
$$

\item if $h$ varies in a one-parameter family $h_t$, then
\begin{equation}\label{eqref:BCinvariantder}
\frac{d}{dt}R(h_t,h_0) = -2i \la \dot h_t h_t^{-1},F_{h_t} \ra,
\end{equation}

\item the following identity holds
$$
dd^c R(h_1,h_0) = \la F_{h_1}\wedge F_{h_1}\ra - \la F_{h_0}\wedge F_{h_0}\ra.
$$
\end{enumerate}
\end{proposition}

As observed by Donaldson in \cite[Prop. 6]{Don}, the \emph{Bott-Chern class} \eqref{eq:BCinvariant} can be defined by integration of \eqref{eqref:BCinvariantder} along a path in the space of reductions of $P$. More precisely, given $h_0$ and $h_1$, one defines
\begin{equation}\label{eq:Rtilde}
\tilde R(h_1,h_0) = -2i \int_0^1 \la \dot h_t h_t^{-1},F_{h_t}\ra dt \in \Omega^{1,1}_\R,
\end{equation}
for a choice of path $h_t$ joining $h_0$ and $h_1$. For a different choice of path, $\tilde R(h_1,h_0)$ differs by an element in $\operatorname{Im}(\partial \oplus \dbar)$, and hence there is a well-defined class $R(h_1,h_0) = [\tilde R(h_1,h_0)]$ in \eqref{eq:BCinvariant}. 

The other piece of information which we will need is the following technical lemma from \cite{grst}. Given a reduction $h \in \Omega^0(P/K)$, using the polar decomposition 
\begin{equation*}\label{eq:polar}
G = \exp(i \mathfrak{k})\cdot K
\end{equation*}
we regard $h$ as a $K$-equivariant map $h \colon P \to \exp(i \mathfrak{k})$. Recall that given an element $g \in \cG_{\underline{P}}$ regarded as an equivariant map $g \colon P \to G$, there is a well-defined covariant derivative
$$
d^h g = g^*\omega^L\circ (\theta^h)^\perp \in \Omega^1(\ad \underline P) 
$$
where $\omega^L$ is the (left-invariant) Maurer-Cartan $1$-form on $G$ and $(\theta^h)^\perp$ denotes the horizontal projection with respect to the Chern connection of $h$.

\begin{lemma}[\cite{grst}]\label{lem:CSRinvariant}
Let $h,h'$ be reductions of $P$. Define $\tilde R(h',h) \in \Omega_\R^{1,1}$ as in \eqref{eq:Rtilde}, where $h' = e^{iu} h$, for $iu \in \Omega^0(i \ad P_{h})$, and $h_t = e^{tiu} h$. Then,
\begin{equation*}\label{eq:hermitianformula}
2i\partial \tilde R(h',h) + CS(\theta^{h'}) - CS(\theta^{h}) - d\la \theta^{h'} \wedge \theta^{h} \ra = d B^{2,0},
\end{equation*}
where
\begin{equation*}\label{eq:B20exp}
B^{2,0} = - \int_0^1 \la a_t \wedge \dot a_t \ra dt  \in \Omega^{2,0}
\end{equation*}
and $a_t: = \theta^{h} - \theta^{h_t} = - \partial^{h} (e^{-2tiu})$ and $\dot a_t = 2i \partial^{h_t} u$.
\end{lemma}

\subsection{Bott-Chern algebroids and compact forms}\label{sec:BCreal}

Our next goal is to study a special type of string algebroids---known as Bott-Chern algebroids---which appear in the Chern correspondence. These are tight to Bott-Chern secondary characteristic classes and a notion of `reduction to a maximal compact subgroup' for string algebroids, which we introduce next. 

A smooth Courant algebroid $(E_\R,\la\, , \ra,[\,,],\pi)$ over a smooth manifold $\underline X$ consists of a smooth vector bundle $E_\R \to \underline X$ together with a non-degenerate symmetric bilinear form $\la\, ,\ra$, a vector bundle morphism $\pi:E_\R\to T\underline X$ and a bracket $[\, ,]$ on sections 
satisfying the Courant algebroid axioms (see (D1)-(D5) in  Section \ref{sec:string}). 

For our next definition, we fix a compact Lie group $K$ and an invariant non-degenerate pairing $\la\,, \ra$ on the Lie algebra $\mathfrak{k}$ of $K$.

\begin{definition}\label{def:Courantreal}
	A \emph{real string algebroid} with structure group $K$ is a tuple $(P_\RR,E_\RR,\rho_{\RR})$, where $P_\RR$ is a smooth principal $K$-bundle over $\underline X$, $E_\RR$ is a smooth (real) Courant algebroid over $\underline X$, and $\rho_{\RR}$ is a bracket-preserving morphism inducing a short exact sequence
	\begin{equation*}\label{eq:defstringsmooth}
	\xymatrix{
		0 \ar[r] & T^*\underline X \ar[r] & E_\RR \ar[r]^{\rho_{\RR}} & A_{P_\RR} \ar[r] & 0,
	}
	\end{equation*}
	such that the induced map of Lie algebroids $\rho_{\RR} \colon A_{E_\RR} \to A_{P_\RR}$ is an isomorphism restricting to an isomorphism $\ad_{E_\RR} \cong (\ad P_\RR,\la\, ,\ra)$.
\end{definition}

Analogously to holomorphic and complex string algebroids, we denote by $H^1(\cS_\RR)$ the set of isomorphism classes of real string algebroids on $\underline{X}$ with structure group $K$. By \cite[Prop. A.6]{grt2}, elements in $H^1(\cS_\RR)$ are represented by equivalence classes of triples $(P_\RR,H_\RR,\theta_\RR)$ satisfying
$$
d H_\RR + \la F_{\theta_\RR} \wedge F_{\theta _\RR}\ra = 0,
$$
where $P_\RR$ is a principal $K$-bundle, $H_\RR$ is a real $3$-form on $\underline{X}$, and $\theta_\RR$ is a connection on $P_\RR$. The triple $(P_\RR,H_\RR,\theta_\RR)$ is related to $(P'_\RR,H'_\RR,\theta'_\RR)$ if there exists an isomorphism $g \colon P_\RR \to P'_\RR$ 
such that, for some real two-form $B\in \Omega^{2}$,
\begin{equation*}\label{eq:anomalyreal}
H'_\R  = H_\R + CS(g\theta_\RR) - CS(\theta'_\RR) - d\la g\theta_\RR \wedge \theta'_\RR\ra + dB.
\end{equation*}

When there is no possibility of confusion, a real string algebroid  $(E_\RR,P_\RR,\rho_\RR)$ will be denoted simply by $E_\RR$. We consider now a complex reductive Lie group $G$, with maximal compact $K \subset G$. Given a principal $K$-bundle $P_\RR$, we can induce uniquely a smooth principal $G$-bundle
\begin{equation}\label{eq:PKG}
\underline P = P_\RR \times_K G.
\end{equation}
Similarly, any real string algebroid over $\underline X$ induces uniquely a complex string algebroid---in the sense of Definition \ref{def:Courantcx}. The underlying principal $G$-bundle is $\underline P$ as in \eqref{eq:PKG}, the complex vector bundle is $E=E_\R\otimes \C$, and there is a commutative diagram,
	\begin{equation}\label{eq:realform}
	\begin{gathered}
	\xymatrix{
		0 \ar[r] & T^*\underline X \otimes \CC \ar[r] & E \ar[r]^{\rho_c} & A_{\underline P} \ar[r] & 0\\
		0 \ar[r] & T^* \underline X \ar[r] \ar[u]^{\cup} & E_\RR \ar[r]^{\rho_\RR}\ar[u]^{\cup}  & A_{P_\RR} \ar[r] \ar[u]& 0,
	}
	\end{gathered}
	\end{equation}
where the vertical arrows are canonical, such that the $\CC$-linear extension of the bracket, the pairing, and the morphism $\rho_{\RR}$ in the bottom sequence induce an isomorphism (this follows by using the universal property of the Atiyah algebroid $A_{\underline P}$). 
Note that the map $A_{P_\R}\to A_{\underline P}$ is not set-theoretically an inclusion, but a canonical injective map (following the definition of $A_{\underline P}$ in \eqref{eq:Atiyahseq}). This construction will be referred to as the `complexification' of $E_\RR$. Conversely, we have the following.

\begin{definition}\label{def:realform}
Let $E$ be a complex string algebroid. A \emph{compact form} of $E$ is a real string algebroid $E_\RR$ with structure group $K$ fitting into a diagram \eqref{eq:realform}. Compact forms will be denoted simply by $E_\RR \subset E$.
\end{definition}

\begin{example}\label{ex:real-form-in-standard}
 Let $E_0$ be the complex string algebroid given by $(\underline P,H_\RR,\theta_\RR)$ with $H_\RR\in \Omega^3\subset \Omega^3_\C$ a real three-form and $\theta_\RR$ a connection on $\underline{P}$ induced by a connection on some reduction $P_\RR \subset \underline P$ to the maximal compact subgroup (cf. Definition \ref{def:E0}). Then, the tuple $(P_\RR,H_\RR,\theta_\RR)$ defines a compact form
$$
 E_{0,\R}:=T\underline X \oplus \ad P_\RR \oplus  T^*\underline X \subset E_0.
$$
\end{example}

Let $Q$ be a string algebroid over a complex manifold $X$, with underlying holomorphic principal $G$-bundle $P$ and smooth manifold $\underline X$. From Lemma \ref{lem:red-for-all-Q}, $Q$ has a canonically associated complex string algebroid $E_Q$. 

\begin{definition}\label{def:BCtype}
A Bott-Chern algebroid over $X$ is a string algebroid $Q$ such that $E_Q$ admits a compact form $E_\RR \subset E_Q$.
\end{definition}

We provide next a handy characterization of the notion of Bott-Chern algebroid, which recovers the definition given originally in \cite{grst}. The proof requires the Bott-Chern classes considered in Proposition \ref{prop:Donaldson} and Lemma \ref{lem:CSRinvariant}. 

\begin{lemma}\label{lemma:BCrealforms}
A string algebroid $Q$ is Bott-Chern if and only if there exists $(\omega,h) \in \Omega^{1,1}_\RR \times \Omega^0(P/K)$ satisfying
\begin{equation}\label{eq:misteriouseqb}
dd^c \omega + \la F_h \wedge F_h \ra = 0
\end{equation}
and $[Q]=[(P,-2i\partial \omega,\theta^h)] \in H^1(\cS)$.
\end{lemma}

\begin{proof}
	Let $Q$ be represented by a tuple $(P,-2i\partial \omega,\theta^h)$. By the equality 
	\begin{equation}\label{eq:delta-tauR-dc-tauR}
	-2i\partial \omega =  d^c \omega - i \partial \omega - i \dbar \omega = d^c\omega - d(i\omega)
	\end{equation}
	 combined with Proposition \ref{lemma:deRhamCsmooth}, we have that 
$$[E_Q] = [(\underline P,-2i\partial \omega,\theta^h)] = [(\underline P,d^c\omega,\theta^h)].$$
Let $\underline{f}$ be an isomorphism of $E_Q$ with a standard $E_0$ given by $(\underline P,d^c\omega,\theta^h)$. By Example  \ref{ex:real-form-in-standard}, there exists a compact form $E_{0,\R}\subset E_0$. We then have that $\underline{f}^{-1}(E_{0,\R})$ is a compact form of $E_Q$.

For the converse, let $E_\RR\subset E_Q$ be a compact form with underlying $K$-bundle $P_h \subset P$. Then, the isomorphism class of $E_\RR$ is represented by $(P_h,H_\RR,\theta^h)$ (we can choose the connection on $P_h$ at will by changing the real three-form accordingly). Let $(P,H,\theta^h)$ represent the class of $Q$ (by Proposition \ref{lemma:deRhamC} we can choose the connection on $\cA_P$ at will by changing $H$ accordingly). In $H^1(\underline\cS)$, we have 
$$ [(\underline P,H,\theta^h)]=[(\underline P,H_\RR,\theta^h)]\in H^1(\underline\cS).$$
Therefore, by Proposition \ref{lemma:deRhamCsmooth} there exists $g \in \cG_{\underline{P}}$ and $B' \in \Omega^2_\CC$ such that
$$
H_\RR = H + CS(g \theta^h) - CS(\theta^h) - d \la g \theta^h \wedge \theta^h \ra + dB'.
$$
Notice that $g\theta^{g^{-1}h}$ defines a connection on $P_h$. Setting
$$
H_\RR' = H_\RR + CS(\theta^h) - CS(g\theta^{g^{-1}h}) - d \la \theta^h \wedge g\theta^{g^{-1}h}\ra
$$
we have that
$$
[(\underline{P},H'_\RR,g\theta^{g^{-1}h})] = [(\underline{P},H_\RR,\theta^h)] \in H^1(\underline{\cS})
$$ 
and
\begin{align*}
H'_\RR & = H + CS(g \theta^h) - CS(g\theta^{g^{-1}h}) - d \la g \theta^h \wedge g \theta^{g^{-1}h} \ra + dB\\
& = H + CS(\theta^h) - CS(\theta^{g^{-1}h}) - d \la \theta^h \wedge \theta^{g^{-1}h} \ra + dB
\end{align*}
where 
$$
B = B' - \la g \theta^h \wedge \theta^h \ra - \la \theta^h \wedge g\theta^{g^{-1}h}\ra + \la \theta^h \wedge \theta^{g^{-1}h} \ra \in \Omega^2_\CC.
$$
By Lemma \ref{lem:CSRinvariant}, there exists a real $(1,1)$-form $R \in \Omega_\RR^{1,1}$ such that
$$
H'_\RR = H - 2i\partial R + dB,
$$
possibly for a different choice of $B$. Since $H'_\RR$ is real, its $(3,0)+(2,1)$-part must equal the conjugate of its $(1,2)+(0,3)$-part, so we obtain
$$
H - 2i \partial R + dB^{2,0} + \partial B^{1,1} = d \overline{B^{0,2}} + \partial \overline{B^{1,1}},
$$
and hence
$$
H = -2i \partial (\Img B^{1,1}-R) + d(\overline{B^{0,2}} - B^{2,0}).
$$
Therefore, $[(P,H,\theta^h)] = [(P,-2i \partial (\Img B^{1,1}-R),\theta^h)] \in H^1(\cS)$, as claimed.
\end{proof}

Observe that the complexification of real string algebroids induces a well-defined map
\begin{equation*}\label{eq:inclusionstring}
c \colon H^1(\cS_\RR) \to H^1(\underline \cS).
\end{equation*}
Recall also that there is a forgetful map $s \colon H^1(\cS) \to H^1(\underline \cS)$ (see Lemma \ref{lem:smap}). We shall use the notation
$H^1_{BC}(\cS)$ for the set of classes of Bott-Chern algebroids inside $H^1(\cS)$. Then, by Definition \ref{def:BCtype},
$$
s(H^1_{BC}(\cS)) \subseteq c(H^1(\cS_\RR)).
$$
In the next proposition we show that compact forms on a Bott-Chern algebroid are unique up to isomorphism. Consequently, we can actually define a map $$r \colon H^1_{BC}(\cS)\to H^1(\cS_\RR)$$ such that $c\circ r=s$, which sends $[Q]$ to $[E_\RR]$ for any compact form $E_\RR \subset E_Q$. The class $[E_\RR]$ is closely related to the notion of real string class for principal bundles (see Remark \ref{rem:BCstringclass}). Our proof will use some results from Section \ref{sec:AeppliPic}. 

\begin{proposition}\label{prop:BCreal}	
Compact forms of a Bott-Chern algebroid are unique up to isomorphism of real string algebroids. Consequently, there exists a unique map $r:H^1_{BC}(\cS)\to H^1(\cS_\R)$ fitting into the commutative diagram
	\begin{equation*}\label{eq:triangle-Bott-Chern}
	\xymatrix @R=.5pc {
		H^1_{BC}(\cS) \ar[dr]_s \ar[rr]^r & & H^1(\cS_\R). \ar[dl]^{c}\\
		& H^1(\underline \cS)  & 
	}
	\end{equation*}
\end{proposition}

\begin{proof} 
By Proposition \ref{lem:retraction} in the Appendix, given compact forms $E_\RR,E'_\RR \subset E_Q$ there exists $g \in \Aut(E_Q)$ such that $g(E_\RR) = E'_\RR$. Restricted to $E_\RR$, $g$ induces an isomorphism of real string algebroids, proving the first part of the statement. By Lemma \ref{lemma:bricksglue} any isomorphism $\psi \colon Q \to Q'$ induces an isomorphism of complex string algebroids $\underline{\psi} \colon E_Q \to E_{Q'}$. Using this, we can define $r$ by $r([Q]) = [E_\RR]$, for any compact form $E_\RR \subset E_Q$. Uniqueness follows from the first part, which implies injectivity of $c$ on $r(H^1(\cS_\R))$.

\end{proof}

\begin{remark}\label{rem:BCstringclass}
When the holomorphic principal bundle $p \colon P \to X$ underlying $Q$ has trivial automorphisms there is a more amenable characterization of the Bott-Chern condition using \emph{real string classes}, in the sense of Redden \cite{Redden}. To see this, notice that $\cG_P = \{1\}$ implies that $[Q] \in H^1(\cS)$ determines uniquely a de Rham cohomology class
$$
[p^*H + CS(\theta)] \in H^3(P,\CC)
$$
for any choice representative $[(P,H,\theta)] = [Q]$  (see Proposition \ref{lemma:deRhamC}). Then, Lemma \ref{lemma:BCrealforms} implies that $Q$ is Bott-Chern if and only if the pullback of $[p^*H + CS(\theta)]$ to $P_h \subset P$, for any reduction $h$ of $P$, is a real string class.
\end{remark}

\subsection{Chern correspondence for string algebroids}\label{sec:Chernco}

We start by introducing the type of objects which play the role of the Chern connection in our context. We fix a complex manifold $X$.

\begin{definition}\label{def:unitarylift}
Let $E_\RR$ be a real string algebroid over $\underline X$. A \emph{horizontal lift} of $T\underline{X}$ to $E_\RR$ is given by a subbundle $W \subset E_\RR$ such that
$$
\rk W = \dim_\RR \underline X, \qquad \textrm{and} \qquad W \cap \Ker \pi = \{0\}.
$$
\end{definition}

Following \cite[Prop. 3.4]{GF}, it is not difficult to see that a horizontal lift $W \subset E_\RR$ is equivalent to a real symmetric $2$-tensor $\sigma$ on $\underline{X}$ and an isotropic splitting $\lambda \colon T\underline X \to E_\RR$ such that
\begin{equation}\label{eq:Wlift}
W = \{\lambda(V) + \sigma(V): V \in T\underline{X}\}.
\end{equation}
Recall that $\lambda$ induces a connection $\theta_\RR$ on $P_\RR$, a three-form $H_\RR$ on $\underline{X}$, and an isomorphism
\begin{equation}\label{eq:E0R}
E_\RR \cong E_{0,\RR} : = T \underline X \oplus\ad P_\RR \oplus T^*\underline{X},
\end{equation}
so that the string algebroid structure on $E_{0,\RR}$ is as in Example \ref{ex:real-form-in-standard}.

Let $E$ be a complex string algebroid with underlying smooth principal $G$-bundle $\underline{P}$. We assume that $G$ is reductive, and fix a maximal compact subgroup $K \subset G$. Given a compact form $E_\RR \subset E$ (see definition \ref{def:realform}), the Cartan involution on $\mathfrak{g}$ determined by the compact Lie subalgebra $\mathfrak{k} \subset \mathfrak{g}$ combined with the underlying reduction $P_\RR \subset \underline{P}$ induces a well-defined involution
\begin{align}\label{eq:Cartaninv}
\begin{split}
\Omega^0(\ad \underline{P}) & \to \Omega^0(\ad \underline{P})\\ 
s & \mapsto s^{*_h}
\end{split}
\end{align}
whose fixed points are given by $\Omega^0(\ad P_\RR)$.

\begin{lemma}[Chern correspondence]\label{lem:Cherncorr}
Let $(E,L)$ be a pair given by a complex string algebroid $E$ over $X$ and a lifting $L \subset E$ of $T^{0,1}X$. Then, any compact form $E_\RR \subset E$ determines uniquely a horizontal lift $W \subset \E_\RR$ such that
\begin{equation}\label{eq:L=LW}
L = \{e \in W \otimes \CC \; | \; \pi(e) \in T^{0,1} X \} \subset E.
\end{equation}
\end{lemma}

\begin{proof}
We choose an isotropic splitting $\lambda_0 \colon T\underline X \to E_\RR$. We will use the same notation for the $\CC$-linear extension of $\lambda_0$ to the complexification $E$. Via the isomorphism \eqref{eq:E0R} induced by $\lambda_0$, we obtain by complexification an isomorphism of complex string algebroids
$$
\underline{f}_0 \colon E_0 \to E
$$
inducing the identity on $A_{\underline{P}}$, and such that $\lambda_0 = (\underline{f}_0)_{|T\underline{X}}$. Then, by Lemma \ref{lemma:liftings} the lifting $L$ determines uniquely $(\gamma,\beta) \in \Omega^{1,1 + 0,2} \oplus \Omega^{0,1}(\ad \underline P)$ such that
$$
L_0 := \underline{f}_0^{-1}(L) = (-\gamma,-\beta)(T^{0,1}X).
$$
Furthermore, given a horizontal lift $W \subset \E_\RR$, there exists a uniquely determined pair $(b,a) \in \Omega^{2}\oplus \Omega^{1}(\ad P_\RR)$ and a real symmetric $2$-tensor $\sigma$ on $\underline{X}$ such that
$$
W_0 := \underline{f}_0^{-1}(W) = (-b,-a)\{V + \sigma(V): V \in T\underline{X}\} \subset E_0.
$$
The isotropic condition for \eqref{eq:L=LW} implies that $\sigma$ is a symmetric tensor of type $(1,1)$. Denote the associated hermitian form by
$$
\omega = \sigma(J,) \in \Omega^{1,1}_\RR,
$$
where $J$ denotes the almost complex structure of $X$.
Then, condition \eqref{eq:L=LW} implies
\begin{align*}
(-\gamma,-\beta)(T^{0,1}X) &= (i\omega -b,-a)(T^{0,1}X)\\
& = (i\omega -b - \la a^{0,1} \wedge a^{1,0} \ra,-a^{0,1})(T^{0,1}X)
\end{align*}
and therefore
$$
- i\omega + b^{1,1 + 0,2} + \la a^{0,1} \wedge a^{1,0} \ra = \gamma, \qquad a^{0,1} = \beta.
$$
From this it follows that
\begin{equation}\label{eq:WChern}
\begin{split}
\omega & = - \Im \; (\gamma^{1,1} - \la a^{0,1} \wedge a^{1,0} \ra)\\
b & = \Re \; (\gamma^{1,1} - \la a^{0,1} \wedge a^{1,0} \ra) + \gamma^{0,2} + \overline{\gamma^{0,2}}\\
a & = \beta + \beta^*,
\end{split}
\end{equation}
where $\beta^*$ is defined combining the involution \eqref{eq:Cartaninv} with the conjugation of complex differential forms. It is not difficult to see that \eqref{eq:WChern} is independent of the choice of splitting $\lambda_0$.
\end{proof}

\begin{remark}
Similarly as in the classical Chern correspondence for principal bundles, 
the involutivity of the lifting $L \subset E$ is not required for the proof of Lemma \ref{lem:Cherncorr}.  
\end{remark}

Let $Q$ be a Bott-Chern algebroid over $X$ with underlying principal $G$-bundle $P$. Let $(E,L,\psi)$ be a diagram for $Q$ (see Definition \ref{def:Moritabrick}).
Without loss of generality, we will assume that $Q$ and $E$ have the same underlying smooth principal $G$-bundle $\underline{P}$, with complex gauge group $\cG_{\underline{P}}$. Recall that Lemma \ref{lemma:bricks} establishes the existence of a unique diagram isomorphism  $\underline{f} \colon E \to E_Q$ between $(E,L,\psi)$ and $(E_Q,T^{0,1}X,\Id_Q)$. By Definition \ref{def:BCtype}, $E$ admits a compact form $E_\RR \subset E$ with structure group $K$. Our next result unravels the data determined by $E_\RR \subset E$ in relation to the fixed Bott-Chern algebroid $Q$, using the Chern correspondence.

\begin{proposition}\label{propo:Chernclassic}
Let $(E,L,\psi)$ be a diagram for $Q$. Then, any compact form $E_\RR \subset E$ determines uniquely a triple $(\omega,h,\varphi)$, where 
\begin{enumerate}

\item $\omega \in \Omega^{1,1}_\RR$ and $h \in \Omega^0(P/K)$ is a reduction of $P$ to $K \subset G$, such that
\begin{equation*}\label{eq:misteriouseqomwithg}
dd^c \omega + \la F_{gh} \wedge F_{gh} \ra = 0,
\end{equation*}
where $g \in \cG_{\underline{P}}$ is covered by $\underline{f} \colon E \to E_Q$, the unique diagram isomorphism between $(E,L,\psi)$ and $(E_Q,T^{0,1}X,\Id_Q)$,

\item $\varphi \colon Q_0\to Q$ is an isomorphism of string algebroids given by a commutative diagram
\begin{equation}\label{eq:unitarybrickiso}
\xymatrix{
0 \ar[r] & T^*X \ar[r] \ar[d]^{id} & Q_0 \ar[r] \ar[d]^{\varphi} & A_{P} \ar[r] \ar[d]^{id} & 0\\
0 \ar[r] & T^*X \ar[r] & Q \ar[r] & A_{P} \ar[r] & 0,
}
\end{equation}
where the string algebroid structure on $Q_0$ is given by $(P,-2i\partial \omega,\theta^{gh})$.
\end{enumerate}
Furthermore, the data $(\omega,h,\varphi)$ recovers the flag $W \subset E_\RR \subset E$, where $W$ is the horizontal lift given by $E_\R$ via the Chern correspondence, and the three-form $H_\RR$ and connection $\theta_\RR$ induced by $W$ are given by
\begin{equation}\label{eq:HthetaChern}
H_\RR = d^c \omega, \qquad \theta_\RR = g^{-1}\theta^{gh}.
\end{equation}
\end{proposition}

\begin{proof}
Given a compact form $E_\RR \subset E$, the principal $K$-bundle underlying $E_\RR$ induces a reduction $h \in \Omega^0(P/K)$. Furthermore, the horizontal lift $W$ determined by $L$ in Lemma \ref{lem:Cherncorr} is equivalent to a pair $(\omega,\lambda)$ where $\omega \in \Omega^{1,1}_\RR$ and $\lambda \colon T\underline X \to E_\RR$ is an isotropic splitting  such that
\begin{equation*}
W \otimes \CC = e^{-i\omega}\lambda(T^{1,0}X) \oplus e^{i\omega}\lambda(T^{0,1}X)
\end{equation*}
for $L = e^{i\omega}\lambda(T^{0,1}X)$. Recall that $\lambda$ induces a connection $\theta_\RR$ on $P_\RR$, a three-form $H_\RR$ on $\underline{X}$, and an isomorphism \eqref{eq:E0R}.

Via \eqref{eq:E0R}, we obtain by complexification an isomorphism of complex string algebroids
$$
\underline{f}_\lambda \colon E_0 \to E
$$
inducing the identity on $A_{\underline{P}}$, and such that $\lambda = (\underline{f}_\lambda)_{|T\underline{X}}$ and
\begin{equation*}\label{eq:V+exp}
\begin{split}
\underline{f}_\lambda^{-1}(W \otimes \CC) & = e^{-i\omega}(T^{1,0}X)\oplus L_0
\end{split}
\end{equation*}
for $L_0 = \underline{f}_\lambda^{-1}(L) = e^{i\omega}(T^{0,1}X)$.

Hence the involutivity of $L_0$ combined with Lemma \ref{lemma:liftings}, yields
$$
H_\RR^{1,2 + 0,3} - i\dbar \omega = 0, \qquad F_{\theta_\RR}^{0,2} = F_{\theta_\RR}^{2,0} = 0.
$$
Therefore, using that $H_\RR$ is real, $H_\RR = d^c \omega$. The reduction of $E_0$ by $L_0$, is the string algebroid $Q_0' = Q_{L_0}$ determined by the triple $(P',-2i\partial\omega,\theta_\RR)$ (see Proposition \ref{prop:QLexp}), for $P' = (\underline{P},\theta_\RR^{0,1})$, where we have used that
$$
H_\RR^{3,0 + 2,1} - \partial (i\omega) = - 2i \partial\omega.
$$
We obtain a string algebroid isomorphism
\begin{equation*}\label{eq:stringunitarybrickidbis}
\xymatrix{
0 \ar[r] & T^*X \ar[r] \ar[d]^{id} & Q_0' \ar[r] \ar[d]^{f_\lambda} & A_{P'} \ar[r] \ar[d]^{id} & 0\\
0 \ar[r] & T^*X \ar[r] & Q_{L} \ar[r] & A_{P'} \ar[r] & 0.
}
\end{equation*}
Consider the unique diagram isomorphism $\underline{f} \colon E \to E_Q$ between $(E,L,\psi)$ and $(E_Q,T^{0,1}X,\Id_Q)$, covering $g \in \cG_{\underline{P}}$. The condition that $g \colon P' \to P$ is an isomorphism implies now that
$$
g\theta_\RR = \theta^{gh}.
$$
We use now the notation $\equiv$ to denote the algebroids, as in Definition \ref{def:Q0} or Example \ref{ex:real-form-in-standard}, given by a tuple. By Proposition \ref{lemma:deRhamC}, $g$ induces an isomorphism
\begin{equation}\label{eq:varphi_0}
\varphi_0 \colon Q_0' \equiv (P',-2i \partial\omega,\theta_\RR) \mapsto Q_0 \equiv (P,-2i \partial\omega,\theta^{gh})
\end{equation}
and therefore $\varphi := \psi \circ f_\lambda \circ \varphi_0^{-1} \colon Q_0 \to Q$ has the required form \eqref{eq:unitarybrickiso}, which proves (1), (2), and \eqref{eq:HthetaChern}.

Conversely, given $(\omega,h,\varphi)$ as in the statement, we have a real string algebroid $E_{0,\RR} \equiv (P_h,d^c\omega,g^{-1}\theta^{gh})$ with complexification $E_0$ and lifting
$$
L_0 = e^{i\omega}T^{0,1}X
$$
such that $Q_{L_0} = Q_0' \equiv (P',-2i \partial\omega,g^{-1}\theta^{gh})$. Consider the isomorphism 
$$
\varphi_0 \colon Q_0' \to Q_0 \equiv (P,-2i \partial\omega,\theta^{gh})
$$
induced by $g$ as in \eqref{eq:varphi_0}, and the unique isomorphism $\tilde{\underline{f}} \colon E_0 \to E$ such that $\tilde{\underline{f}}(L_0) = L$, given by the diagram
\begin{equation*}\label{eq:diagunderlineftilde}
\xymatrix @R=.5pc {
		E_0 \ar@{-->}[dr] \ar[dd]_{\tilde{\underline{f}}} & & \\
		& Q_0' \ar[rd]^{\varphi \circ \varphi_0} \ar[dd]_{\tilde{f}} & \\ 
		E \ar@{-->}[dr] & & Q. \\
		& Q_L \ar[ru]_{\psi}  & 
	}
\end{equation*}
Then, we define the compact form and horizontal lift by
$$
E_\RR := \tilde{\underline{f}}(E_{0,\RR}), \qquad W := \tilde{\underline{f}}(\{V + \sigma(V): V \in T\underline{X}\}),
$$
where $\sigma$ is the symmetric tensor determined by $\omega$.
\end{proof}

\begin{remark}\label{rem:metrics}
As we will see in Lemma \ref{lemma:BQexp}, compact forms $E_\RR \subset E_Q$ extend naturally the notion of \emph{metric} on $Q$ introduced in \cite{grst}.
\end{remark}

\section{Moment maps}\label{sec:mmapmod}

\subsection{Conformally balanced metrics and moment maps}
\label{sec:balancedconf}

Let $X$ be a compact complex manifold of dimension $n$. Consider the space 
$$
\Omega^{1,1}_{>0} \subset \Omega^{1,1}_\RR
$$ of positive $(1,1)$-forms on $X$, sitting inside the vector space of real $(1,1)$-forms $\Omega^{1,1}_\RR$ as an open subspace. We will use the convention that, for $\omega \in \Omega^{1,1}_{>0}$, 
$$
\omega(V,JV)>0
$$
for any nonzero $V\in TX$, where we recall that $J$ is the almost complex structure on $X$. That is, $\omega(,J)$ defines a hermitian metric. We use the notation $(\omega,b)$ for the elements of $T\Omega^{1,1}_{>0}$, the total space of the tangent bundle
$$
T\Omega^{1,1}_{>0} \cong \Omega^{1,1}_{>0} \times \Omega^{1,1}_\RR,
$$
and the notation $(\dot \om, \dot b)$ for elements in the tangent bundle of $T\Omega^{1,1}_{>0}$ at $(\om,b)$. The space $T\Omega^{1,1}_{>0}$
has a natural integrable complex structure given by
\begin{equation}\label{eq:JTOmega}
\mathbf{J}(\dot \omega,\dot{b}) = (-\dot b, \dot \omega).
\end{equation}
 Consider the partial action of the additive group of complex two-forms
\begin{align}\label{eq:Om2Caction}
\begin{split}
\Omega^2_\CC \times T\Omega^{1,1}_{>0} &\to T\Omega^{1,1}_\RR\\ 
(B,(\omega,b)) &\mapsto (\omega + \Re \; B^{1,1},b + \Im \; B^{1,1}),
\end{split}
\end{align}
preserving the complex structure $\mathbf{J}$. This section is devoted to the study of a Hamiltonian action of the subgroup of purely imaginary two-forms $i \Omega^2 \subset \Omega^2_\CC$ for a natural family of K\"ahler structures on $T\Omega^{1,1}_{>0}$.

To define the family of symplectic structures of our interest, we fix a smooth volume form $\mu$ on $X$ compatible with the complex structure. For any $\omega \in \Omega^{1,1}_{>0}$, we define a function $f_\omega$ by
\begin{equation}\label{eq:dilaton}
\frac{\omega^n}{n!} = e^{2f_\omega} \mu.
\end{equation}
We will call $f_\omega$ the \emph{dilaton function} of the hermitian metric $\omega$ with respect to $\mu$. 

\begin{definition}\label{def:dilatonfunctional}
	Given $\ell \in \RR \backslash \{2\}$, the $\ell$-\emph{dilaton functional} on $T\Omega^{1,1}_{>0}$ is
	\begin{equation*}\label{eq:Dilfunctional}
	M_\ell(\omega,b) := \int_X e^{-\ell f_\omega} \frac{\omega^n}{n!}.
	\end{equation*}
\end{definition}

Associated to the functionals $M_\ell$ there is a family of exact $(1,1)$-forms defined by
\begin{equation}\label{eq:Omegal}
\Omega_\ell := -d\mathbf{J}d \log M_\ell.
\end{equation}
The following family of $1$-form potentials plays a key role in the present work
\begin{equation}\label{eq:lambdal}
\lambda_\ell := -\mathbf{J}d \log M_\ell = -\frac{1}{M_\ell}\mathbf{J}d M_\ell.
\end{equation}

\begin{lemma}\label{lemma:lambdaOmegal}
The forms $\lambda_\ell$ and $\Omega_\ell$, evaluated at the tangent vectors $v = (\dot \omega, \dot b)$ and $v_j = (\dot \omega_j, \dot b_j)$ at the point $(\om,b)\in T\Omega^{1,1}_{>0}$, are given by
	\begin{equation}\label{eq:lambdaOmegal}
	\begin{split}
	\lambda_\ell(v) & = \frac{ \ell-2}{2 M_\ell} \int_X \dot b\wedge e^{-\ell f_\omega} \frac{\omega^{n-1}}{(n-1)!},\\
	\Omega_\ell(v_1,v_2) & =  \frac{\ell-2}{2M_\ell} \int_X (\dot \omega_1 \wedge \dot b_2 - \dot \omega_2 \wedge \dot b_1) \wedge e^{-\ell f_\omega} \frac{\omega^{n-2}}{(n-2)!}\\
	& \phantom{ {} = } + \frac{\ell( \ell-2)}{4 M_\ell} \int_X (\Lambda_\omega \dot b_1 \Lambda_\omega \dot \omega_2 - \Lambda_\omega\dot b_2 \Lambda_\omega \dot \omega_1) e^{-\ell f_\omega} \frac{\omega^{n}}{n!}\\
	& \phantom{ {} = } + \left(\frac{\ell-2}{2M_\ell}\right)^2\left(\int_X \Lambda_\omega(\dot \omega_1) e^{-\ell f_\omega} \frac{\omega^n}{n!}\right)\left(\int_X \Lambda_\omega(\dot b_2) e^{-\ell f_\omega} \frac{\omega^n}{n!}\right)\\
	& \phantom{ {} = } - \left(\frac{\ell-2}{2M_\ell}\right)^2\left(\int_X \Lambda_\omega(\dot \omega_2) e^{-\ell f_\omega} \frac{\omega^n}{n!}\right)\left(\int_X \Lambda_\omega(\dot b_1) e^{-\ell f_\omega} \frac{\omega^n}{n!}\right).
	\end{split}
	\end{equation}
\end{lemma}

\begin{proof}
	Let $(\dot \omega,\dot b)$ denote a tangent vector at $(\omega,b) \in \Omega^{1,1}_{>0}$. Using that 
	$$
	M_\ell = \int_X e^{(2-\ell) f_\omega} \mu
	$$ 
	it follows that
	\begin{align*}
	d M_\ell(\dot \omega,\dot b) = \frac{2-\ell}{2}\int_X \Lambda_\omega(\dot \omega) e^{-\ell f_\omega} \frac{\omega^n}{n!}
	\end{align*}
	where we have used that the variation $2 \delta f_\omega(\dot \omega) = \Lambda_\omega \dot \omega$ by definition of $f_\omega$. Thus, the first part of \eqref{eq:lambdaOmegal} follows from \eqref{eq:JTOmega}. As for the second formula, we calculate
	\begin{align*}
	d\mathbf{J}d M_\ell((\dot \omega_1,&\dot b_1),(\dot \omega_2,\dot b_2)) \\ 
	& {} = \frac{2- \ell}{2} \int_X (\dot b_2 (-\ell(\Lambda_\omega \dot \omega_1)/2) - \dot b_1 (-\ell(\Lambda_\omega \dot \omega_2)/2)) \wedge e^{-\ell f_\omega} \frac{\omega^{n-1}}{(n-1)!} \\
	&  \phantom{ {} = } +  \frac{2 - \ell}{2} \int_X (\dot b_2 \wedge \dot \omega_1 - \dot b_1 \wedge \dot \omega_2) \wedge e^{-\ell f_\omega} \frac{\omega^{n-2}}{(n-2)!}\\
	& {} = \frac{\ell(2 - \ell)}{4} \int_X (\Lambda_\omega \dot b_1 \Lambda_\omega \dot \omega_2 - \Lambda_\omega\dot b_2 \Lambda_\omega \dot \omega_1) e^{-\ell f_\omega} \frac{\omega^{n}}{n!}\\
	&  \phantom{ {} = } + \frac{2 - \ell}{2} \int_X (\dot \omega_1 \wedge \dot b_2 - \dot \omega_2 \wedge \dot b_1) \wedge e^{-\ell f_\omega} \frac{\omega^{n-2}}{(n-2)!},
	\end{align*}
	and therefore \eqref{eq:lambdaOmegal} follows from
	$$
	\Omega_\ell = -\frac{1}{M_\ell} d\mathbf{J}d M_\ell + \frac{1}{(M_\ell)^2}dM_\ell \wedge  \mathbf{J}d M_\ell.
	$$
\end{proof}

We provide next a formula for the associated family of symmetric tensors, obtaining K\"ahler metrics for certain values of the parameter $\ell$. Given $(\omega,b) \in T\Omega^{1,1}_{>0}$, we denote by
$$
b_0 = b - \frac{\Lambda_\omega b}{n} \omega
$$
the primitive part of $b$ with respect to $\omega$.

\begin{lemma}\label{lemma:gl}
	The symmetric tensor $g_\ell = \Omega_\ell(,\mathbf{J})$ at $(\omega,b)$, evaluated at $v = (\dot \omega, \dot b)$, is given by
	\begin{equation}\label{eq:gl}
	\begin{split}
	g_\ell (v,v) & =  \frac{2 - \ell}{2 M_\ell} \int_X (|\dot \omega_0|^2 + |\dot b_0|^2) e^{-\ell f_\omega} \frac{\omega^{n}}{n!}\\
	& \phantom{ {} = } + \frac{2 - \ell}{2 M_\ell}\Bigg{(}\frac{\ell}{2} - \frac{n-1}{n}\Bigg{)} \int_X (|\Lambda_\omega \dot b|^2 + | \Lambda_\omega \dot \omega|^2) e^{-\ell f_\omega} \frac{\omega^{n}}{n!}\\
	& \phantom{ {} = }  + \left(\frac{2-\ell}{2M_\ell}\right)^2\left(\int_X \Lambda_\omega \dot \omega e^{-\ell f_\omega} \frac{\omega^n}{n!}\right)^2
	 + \left(\frac{2-\ell}{2M_\ell}\right)^2\left(\int_X \Lambda_\omega\dot b e^{-\ell f_\omega} \frac{\omega^n}{n!}\right)^2.
	\end{split}
	\end{equation}
	In particular, $g_\ell$ is K\"ahler if $2 - \frac{2}{n} < \ell < 2$ and $-g_\ell$ is K\"ahler if $\ell > 2$.
\end{lemma}
\begin{proof}
The proof of \eqref{eq:gl} is straightforward from \eqref{eq:JTOmega} and \eqref{eq:lambdaOmegal}. The K\"ahler property of $-g_\ell$ for $\ell > 2$ follows from the Cauchy-Schwarz inequality, which implies
$$
\frac{1}{M_\ell}\left(\int_X \Lambda_\omega\dot b\, e^{-\ell f_\omega} \frac{\omega^n}{n!}\right)^2 \leq \int_X |\Lambda_\omega \dot b|^2 e^{-\ell f_\omega} \frac{\omega^{n}}{n!}.
$$
\end{proof}

Consider the action of the additive group of purely imaginary two-forms induced by \eqref{eq:Om2Caction}
\begin{align*}
i \Omega^2 \times T\Omega^{1,1}_{>0} &\to T\Omega^{1,1}_{>0}\\ (i B,(\omega,b)) &\mapsto (\omega,b+B^{1,1}).
\end{align*}
Since the $i \Omega^2$-action preserves both $\mathbf{J}$ and $M_\ell$, it also preserves the one-form $\lambda_\ell$ (see \eqref{eq:lambdal}). Thus, by \eqref{eq:Omegal}, the action is Hamiltonian and there exists an equivariant moment map, which we calculate in the following result.

\begin{proposition}\label{prop:mmap}
	The action of $i\Omega^2$ on $T\Omega^{1,1}_{>0}$ is Hamiltonian, with equivariant moment map
\begin{equation}\label{eq:mmapcl}
\la \mu_\ell(\omega,b),B \ra = \frac{2-\ell}{2M_\ell} \int_X B\wedge e^{-\ell f_\omega} \frac{\omega^{n-1}}{(n-1)!}.
\end{equation}
Upon restriction to the subgroup $i\Omega^2_{ex} \subset i\Omega^2$ of imaginary exact $2$-forms on $X$, zeros of the moment map are given by $\ell$-conformally balanced metrics, that is,
\begin{equation*}\label{eq:lcbalanced}
d(e^{-\ell f_\omega} \omega^{n-1})= 0.
\end{equation*}
\end{proposition}

\begin{proof}
The $i\Omega^2$-action is Hamiltonian, with moment map
	$$
	\la \mu_\ell(\omega,b),iB \ra = - \lambda_\ell (iB \cdot (\omega,b)) = - \lambda_\ell (0,B), 
	$$
	where $iB \cdot (\omega,b) \in T_{(\omega,b)}T\Omega^{1,1}_{>0}$ denotes the infinitesimal action of $iB$ on $(\omega,b)$. Formula \eqref{eq:mmapcl} follows now from \eqref{eq:lambdaOmegal}. The last part of the statement is straightforward and is left to the reader.
\end{proof}

To the knowledge of the authors, the previous result provides the first moment map interpretation of the conformally balanced equation in the literature. In particular, for $\ell = 0$ we obtain a symplectic interpretation of balanced metrics. Similarly, when $X$ admits a holomorphic volume form $\Omega$ and we take
\begin{equation}\label{eq:muOmega2}
\mu = (-1)^{\frac{n(n-1)}{2}}i^n\Omega \wedge \overline{\Omega}
\end{equation}
and $\ell = 1$, Proposition \ref{prop:mmap} characterizes hermitian metrics with holonomy for the Bismut connection contained in $\operatorname{SU}(n)$ as a moment map condition (see e.g. \cite{grst}, cf. Corollary \ref{cor:mmapHS}). Observe that for these two interesting cases we cannot ensure that the metric $\pm g_\ell$ in \eqref{eq:gl} is K\"ahler.

\subsection{K\"ahler reduction and the Calabi system}\label{sec:mmapHS}

Let $E_\RR$ be a real string algebroid with underlying principal $K$-bundle $P_\RR$ over our compact complex manifold $X$ (see Section \ref{sec:BCreal}). 
Let $G$ be the complexification of $K$. Let $E$ be the complexification of $E_{\RR}$, with underlying principal $G$-bundle $\underline{P} = P_\RR \times_K G$. Given a horizontal lift $W \subset E_\RR$ of $T\underline{X}$ to $E_\RR$ (see Definition \ref{def:unitarylift}) we define
\begin{equation*}\label{eq:LW}
L_W := \{e \in W \otimes \CC \; | \; \pi(e) \in T^{0,1}\underline X \} \subset E.
\end{equation*}
Consider the set of horizontal lifts of $T\underline{X}$ to $E_\RR$ such that $L_W$ is isotropic
$$
\mathcal{W} := \{W \subset E_\RR \; | \; W \; \textrm{is a horizontal lift and} \; L_W \; \textrm{is isotropic}\}.
$$
Recall from Section \ref{sec:Chernco} that any $W \in \mathcal{W}$ induces the following data: a real $(1,1)$-form $\omega \in \Omega^{1,1}_\RR$ on $\underline{X}$, a three-form $H_\RR$, a connection $\theta_\RR$ on $P_\RR$, and an isomorphism $E_\RR \cong E_{0,\RR}$ (see \eqref{eq:E0R}), so that the Courant structure on $E_{0,\RR}$ is as in Definition \ref{def:E0}. In particular, there is a well-defined forgetful map
\begin{equation}\label{eq:WA}
\mathcal{W} \longrightarrow \Omega^{1,1}_\RR\times \cA,
\end{equation}
where $\cA$ denotes the space of principal connections on $P_\RR$. Furthermore, via $E_\RR \cong E_{0,\RR}$, we have
$$
W = \{V + \sigma(V): V \in T\underline{X}\},
$$
where $\sigma = \omega(,J)$. The following result is a straightforward consequence of the Chern correspondence in Lemma \ref{lem:Cherncorr}.

\begin{lemma}\label{lem:Cherncorrsets}
Denote by $\cL$ the set of isotropic subbundles $L \subset E$ mapping isomorphically to $T^{0,1} X$ under $\pi \colon E \to T\underline{X} \otimes \CC$. Then, there is a bijection
\begin{align}
\begin{split}\label{eq:Cherncorrsets}
\mathcal{W} &\to \cL\\
W  &\mapsto L_W.
\end{split}
\end{align}
\end{lemma}
The sets $\mathcal{W}$ and $\cL$ have natural structures of affine space modelled on the vector spaces $\Omega^{1,1}_\RR \oplus \Omega^2 \oplus \Omega^1(\ad P_\RR)$ and $\Omega^{1,1 + 0,2} \oplus \Omega^{0,1}(\ad \underline P)$, respectively (see Lemma \ref{lemma:liftings} and Lemma \ref{lem:Cherncorr}). It is not difficult to see that the map \eqref{eq:Cherncorrsets} is affine, and thus the natural complex structure on $\cL$ given by multiplication by $i$ induces a complex structure $\mathbf{J}$ on $\cW$ making \eqref{eq:Cherncorrsets} holomorphic.

\begin{lemma}\label{lem:JW}
Any element $W \in \cW$ induces a natural bijection $\cW \cong \Omega^{1,1}_\RR \oplus \Omega^2 \oplus \cA$. Via this identification, $\mathbf{J}$ is given by
 \begin{equation}\label{eq:JW}
\mathbf{J}_{|W}(\dot \omega,\dot b, \dot a) = (-\dot b^{1,1},\dot \omega + i \dot b^{0,2} - i \overline{\dot b^{0,2}},J \dot a)
\end{equation}
for $(\dot \omega,\dot b, \dot a) \in \Omega^{1,1}_\RR \oplus \Omega^2 \oplus \Omega^1(\ad P_\RR)$ and $J \dot a := i \dot a^{0,1} - i \dot a^{1,0}$. Consequently, the forgetful map $\mathcal{W} \longrightarrow \cA$ induced by \eqref{eq:WA} is holomorphic.
\end{lemma}

\begin{proof}
Without loss of generality, we fix an isotropic splitting $\lambda_0 \colon T\underline{X} \to E_\RR$, with induced connection $\theta_0$ on $P_\RR$. Via the isomorphism $E_\RR \cong E_{0,\RR}$ induced by $\lambda_0$, as in \eqref{eq:E0R}, an element $W \in \cW$ is given by a triple 
\begin{equation*}\label{eq:cWexp}
(\omega,b,\theta_\RR) \in  \Omega^{1,1}_\RR \times \Omega^2 \times \cA,
\end{equation*}
with corresponding horizontal lift
$$
W = (-b,-a) \{V + \omega(V,J): V \in T\underline{X}\},
$$
for $a = \theta_\RR - \theta_0$, and isotropic subbundle (see Lemma \ref{lem:Cherncorr})
\begin{align*}
L_W &= (i\omega -b,-a)(T^{0,1}X)\\
& = (i\omega - b^{1,1 + 0,2} - \la a^{0,1} \wedge a^{1,0} \ra,-a^{0,1})(T^{0,1}X)\\
& = (i\omega - b^{1,1 + 0,2} - \tfrac{i}{2}\la a \wedge Ja \ra^{1,1} , -a^{0,1})(T^{0,1}X).
\end{align*}
Thus, the differential of the map \eqref{eq:Cherncorrsets} at $W \equiv (\omega,b,\theta_\RR)$ can be identified with the linear map
\begin{align*}
\Omega^{1,1}_\RR \oplus \Omega^2 \oplus \Omega^1(\ad P_\RR) & \longrightarrow \Omega^{1,1 + 0,2} \oplus \Omega^{0,1}(\ad \underline P)\\
(\dot \omega,\dot b,\dot a) & \longmapsto (\dot b^{1,1 + 0,2} - i(\dot \omega - \tfrac{1}{2}\la \dot a \wedge Ja \ra^{1,1} - \tfrac{1}{2}\la a \wedge J\dot a \ra^{1,1}),\dot a^{0,1}),
\end{align*}
and the induced complex structure is given by
$$
\mathbf{J}_{|W}(\dot \omega,\dot b, \dot a) = (- \dot b^{1,1} + \la \dot a \wedge a \ra^{1,1} ,\dot \omega + \la J\dot a \wedge a \ra^{1,1} + i \dot b^{0,2} - i \overline{\dot b^{0,2}}, J \dot a).
$$
Taking now $\lambda_0$ to be the isotropic splitting induced by $W$ we have $a = 0$ and the statement follows (for the last part see e.g. \cite{Don}).
\end{proof}

Consider now the natural left action of $\Aut(E_\RR)$ on $\cW$, given by
\begin{equation}\label{eq:AutEr-left-actioncW}
\begin{split}
\Aut(E_\RR) \times \cW & \longrightarrow \cW\\
(\underline{f},W) & \longmapsto  \underline{f} \cdot W : = \underline{f}(W).
\end{split}
\end{equation}
Our goal is to find a Hamiltonian action on $\cW$ induced by \eqref{eq:AutEr-left-actioncW} and study its symplectic reduction. For this, we need a better understanding of the action \eqref{eq:AutEr-left-actioncW}. Our next result shows that \eqref{eq:AutEr-left-actioncW} preserves the complex structure $\mathbf{J}$, and furthermore it extends the classical action of the gauge group $\cG_{P_\RR}$ on the space of connections $\cA$. Recall from \cite[App. A]{grt2} that there is a well-defined group homomorphism
$$
\sigma_{P_\RR} \colon \cG_{P_\RR} \to H^3(X,\RR)
$$
defined as in Corollary \ref{cor:PicQ}, inducing an exact sequence
\begin{equation*}\label{eq:ERses}
  \xymatrix{
0 \ar[r] & \Omega^2_{cl} \ar[r] & \Aut(E_\RR) \ar[r] &  \ar[r] \Ker \sigma_{P_\RR} &  \ar[r]^{\sigma_{P_\RR} \qquad} \cG_{P_\RR} & H^3(X,\RR).
  }
\end{equation*}

\begin{lemma}\label{lem:WA}
The action \eqref{eq:AutEr-left-actioncW} preserves $\mathbf{J}$. Furthermore, the forgetful map \eqref{eq:WA} jointly with the action \eqref{eq:AutEr-left-actioncW} induce a commutative diagram
\begin{equation*}
\xymatrix @R=.5pc {
		\Aut(E_\RR) \times \cW \ar[r] \ar[dd] & \cW \ar[dd] \\
		& \\
		\Ker \sigma_{P_\RR} \times \Omega^{1,1}_\RR \times \cA \ar[r] & \Omega^{1,1}_\RR \times \cA ,
	}
\end{equation*}
where the bottom arrow is induced by the left $\cG_{P_\RR}$-action on $\Omega^{1,1}_\RR \times \cA $, given by $g\cdot(\omega,\theta_\RR) = (\omega,g\theta_\RR)$.
\end{lemma}

\begin{proof}
For the first part, observe that the map \eqref{eq:Cherncorrsets} is equivariant for the action of $\Aut(E_\RR)$ on $\cL$, given by
\begin{equation}\label{eq:AutEr-left-actioncL}
\begin{split}
\Aut(E_\RR) \times \cL & \longrightarrow \cL\\
(\underline{f},L) & \longmapsto  \underline{f} \cdot L : = \underline{f}(L).
\end{split}
\end{equation}
Using that \eqref{eq:AutEr-left-actioncL} is induced by the natural complex $\Aut(E)$-action on $\cL$ (defined by the same formula), we obtain that $\mathbf{J}$ is preserved by  \eqref{eq:AutEr-left-actioncW}.

As for the second part, without loss of generality we fix an isotropic splitting $\lambda_0 \colon T\underline{X} \to E_\RR$, with induced connection $\theta_0$ on $P_\RR$. Via the induced isomorphism $E_\RR \cong E_{0,\RR}$, as in \eqref{eq:E0R}, an element $W \in \cW$ is given by a triple 
\begin{equation*}\label{eq:cWexp2}
(\omega,b,\theta_\RR) \in  \Omega^{1,1}_\RR \times \Omega^2 \times \cA,
\end{equation*}
with corresponding horizontal lift
$$
W = (-b,\theta_0 - \theta_\RR) W_\omega
$$
for $W_\omega := \{V + \omega(V,J) \; | \; V \in T\underline{X}\}$. An element in $\Aut(E_\RR) \cong \Aut(E_{0,\RR})$ is given by a pair $(g,\tau) \in \cG_{P_\RR} \times \Omega^2$ satisfying (cf. Lemma \ref{lem:HomQQ})
$$
d \tau = CS(g^{-1}\theta_0) - CS(\theta_0) - d \la g^{-1} \theta_0 \wedge \theta_0 \ra
$$
and the action \eqref{eq:AutEr-left-actioncW} is
\begin{align*}
(g,\tau)(W) & = (\tau - b + \la a^g \wedge \theta_0 - \theta_\RR \ra,g(a^g + \theta_0 - \theta_\RR))(W_\omega)\\
& = (\tau - b + \la a^g \wedge \theta_0 - \theta_\RR \ra,\theta_0 - g\theta_\RR)(W_\omega)
\end{align*}
for $a^g = g^{-1} \theta_0 - \theta_0$. Thus, the statement follows.
\end{proof}

Consider the open subset $\Omega^{1,1}_{>0} \subset \Omega^{1,1}_\RR$ given by the positive $(1,1)$-forms on $X$. The phase space for our symplectic reduction is the following open subset of $\cW$
$$
\cW_+ = \{W \in  \cW \;|\; \omega(,J) > 0\} \subset \cW.
$$
To define our family of symplectic structures, we fix a smooth volume form $\mu$ on $X$ compatible with the complex structure. For any $\omega \in \Omega^{1,1}_{>0}$, we define the \emph{dilaton function} $f_\omega \in C^\infty(X)$ as in \eqref{eq:dilaton}.

\begin{definition}\label{def:dilatonfunctionalW}
	Given $\ell \in \RR \backslash \{2\}$, the $\ell$-\emph{dilaton functional} on $\cW_+$ is
	\begin{equation}\label{eq:DilfunctionalW}
	M_\ell(W) := \int_X e^{-\ell f_\omega} \frac{\omega^n}{n!}.
	\end{equation}
\end{definition}

Observe that $M_\ell$ is the pullback of the functional in Definition \ref{def:dilatonfunctional} by the projection $\cW \to \Omega^{1,1}_\RR$ induced by \eqref{eq:WA}. In the sequel we fix $\ell \in \RR \backslash \{2\}$. Associated to the functional $M_\ell$ there is a one-form $\lambda_\ell \in \Omega^1(\cW_+)$ on $\cW_+$, given by
\begin{equation*}\label{eq:lambdal2}
\lambda_\ell := -\mathbf{J}d \log M_\ell = -\frac{1}{M_\ell}\mathbf{J}d M_\ell.
\end{equation*}

\begin{lemma}\label{lem:lambdalW}
The one-form $\lambda_\ell$ is preserved by the $\Aut(E_\RR)$-action. Furthermore,
\begin{equation}\label{eq:lambdalW}
\lambda_{\ell|W}(\dot \omega,\dot b, \dot a) = \frac{ \ell-2}{2 M_\ell} \int_X \dot b \wedge e^{-\ell f_\omega} \frac{\omega^{n-1}}{(n-1)!}
\end{equation}
for $(\dot \om,\dot b, \dot a)\in T_W\cW_+ \cong \Omega^{1,1}_\RR \oplus \Omega^2 \oplus \Om^1(\ad P_\R)$.
\end{lemma}

\begin{proof}
The first part of the statement is a direct consequence of Lemma \ref{lem:JW} and Lemma \ref{lem:WA}. As for formula \eqref{eq:lambdalW}, without loss of generality, we fix an isotropic splitting $\lambda_0 \colon T\underline{X} \to E_\RR$ with induced connection $\theta_0$ on $P_\RR$. By the proof of Lemma \ref{lem:JW} combined with Lemma \ref{lemma:lambdaOmegal}, the one-form $\lambda_\ell$ is
\begin{equation}\label{eq:lambdalWgen}
\begin{split}
\lambda_{\ell|W}(\dot \omega,\dot b, \dot a) & = \frac{\ell-2}{2 M_\ell} \int_X (\dot b^{1,1} - \la \dot a \wedge a \ra^{1,1})\wedge e^{-\ell f_\omega} \frac{\omega^{n-1}}{(n-1)!},
\end{split}
\end{equation}
for $a = \theta_\RR - \theta_0$. Taking now $\lambda_0$ to be the isotropic splitting induced by $W$ we have $a = 0$ and the statement follows.
\end{proof}

Similarly as in Section \ref{sec:balancedconf}, we endow $\cW_+$ with an $\Aut(E_\RR)$-invariant exact $(1,1)$-form defined by
\begin{equation}\label{eq:Omegal2}
\Omega_\ell :=- d\mathbf{J}d \log M_\ell.
\end{equation}
We calculate next a formula for $\Omega_\ell$ and the symmetric two-tensor $g_\ell = \Omega_\ell(,\mathbf{J})$. We use the notation in Lemma \ref{lemma:gl} for the decomposition of two-forms into primitive and non-primitive parts.

\begin{lemma}\label{lemma:lambdaOmegalW}
The evaluation of $\Om_\ell$ and $g_\ell$, along tangent vectors $v = (\dot \om, \dot{b}, \dot a)$ and $v_j = (\dot \om_j, \dot{b}_j, \dot a_j)$ at the point $(\om, b, a)$, is given by:
	\begin{equation}\label{eq:OmegalW}
	\begin{split}
	\Omega_\ell(v_1,v_2) & {} =  \frac{\ell - 2}{M_\ell} \int_X \la \dot a_1 \wedge\dot a_2 \ra \wedge e^{-\ell f_\omega} \frac{\omega^{n-1}}{(n-1)!}\\
	& \phantom{ {} = } + \frac{ \ell-2}{2M_\ell} \int_X (\dot \omega_1 \wedge \dot b_2 - \dot \omega_2 \wedge \dot b_1) \wedge e^{-\ell f_\omega} \frac{\omega^{n-2}}{(n-2)!}\\
	& \phantom{ {} = } + \frac{\ell( \ell-2)}{4 M_\ell} \int_X (\Lambda_\omega \dot b_1 \Lambda_\omega \dot \omega_2 - \Lambda_\omega\dot b_2 \Lambda_\omega \dot \omega_1) e^{-\ell f_\omega} \frac{\omega^{n}}{n!}\\
	& \phantom{ {} = } + \left(\frac{\ell-2}{2M_\ell}\right)^2\left(\int_X \Lambda_\omega(\dot \omega_1) e^{-\ell f_\omega} \frac{\omega^n}{n!}\right)\left(\int_X \Lambda_\omega(\dot b_2) e^{-\ell f_\omega} \frac{\omega^n}{n!}\right)\\
	& \phantom{ {} = } - \left(\frac{\ell-2}{2M_\ell}\right)^2\left(\int_X \Lambda_\omega(\dot \omega_2) e^{-\ell f_\omega} \frac{\omega^n}{n!}\right)\left(\int_X \Lambda_\omega(\dot b_1) e^{-\ell f_\omega} \frac{\omega^n}{n!}\right).
	\end{split}
	\end{equation}
	\begin{equation}\label{eq:glW}
	\begin{split}
	g_\ell (v,v) & {} =  \frac{\ell - 2}{M_\ell} \int_X \la \dot a \wedge J \dot a \ra \wedge e^{-\ell f_\omega} \frac{\omega^{n-1}}{(n-1)!}\\
&  \phantom{ {} = } +  \frac{2 - \ell}{2 M_\ell} \int_X (|\dot \omega_0|^2 + |\dot b_0^{1,1}|^2) e^{-\ell f_\omega} \frac{\omega^{n}}{n!}\\
	& \phantom{ {} = } + \frac{2 - \ell}{2 M_\ell}\Bigg{(}\frac{\ell}{2} - \frac{n-1}{n}\Bigg{)} \int_X (|\Lambda_\omega \dot b|^2 + | \Lambda_\omega \dot \omega|^2) e^{-\ell f_\omega} \frac{\omega^{n}}{n!}\\
	& \phantom{ {} = }  + \left(\frac{2-\ell}{2M_\ell}\right)^2\left(\int_X \Lambda_\omega \dot \omega e^{-\ell f_\omega} \frac{\omega^n}{n!}\right)^2
	 + \left(\frac{2-\ell}{2M_\ell}\right)^2\left(\int_X \Lambda_\omega\dot b e^{-\ell f_\omega} \frac{\omega^n}{n!}\right)^2.
	\end{split}
	\end{equation}
\end{lemma}

\begin{proof}
We fix an isotropic splitting $\lambda_0 \colon T\underline{X} \to E_\RR$. Formulae \eqref{eq:OmegalW} and \eqref{eq:glW} follow by taking first the exterior derivative in \eqref{eq:lambdalWgen} and then setting $\lambda_0$ to be the splitting induced by $W$, combined with Lemma \ref{lemma:lambdaOmegal} and Lemma \ref{lem:JW}.
\end{proof}

\begin{remark}\label{rem:glnondeg}
Arguing as in the proof of Lemma \ref{lemma:gl}, one can prove that $g_\ell$ (respectively $-g_\ell$) induces a pseudo-K\"ahler metric along the subbundle $\Omega^{1,1}_\RR \oplus \Omega^{1,1}_\RR \oplus \Omega^1(\ad P_\RR) \subset T\cW_+$ provided that $2 - \frac{2}{n} < \ell < 2$ (respectively $\ell > 2$).
\end{remark}

By Lemma \ref{lem:lambdalW}, the action of $\Aut(E_\RR)$ on $(\cW_+,\Omega_\ell)$ is Hamiltonian, with moment map
$$
\la \mu_\ell(W),\zeta \ra = - \lambda_\ell (\zeta \cdot W)
$$
for $\zeta \in \Lie \Aut(E_\RR)$, where $\zeta \cdot W$ denotes the infinitesimal action. The following explicit formula follows from the proof of Lemma \ref{lem:WA}. Recall that any $W \in \cW$ determines an isotropic splitting $\lambda \colon T\underline{X} \to E_\RR$ with connection $\theta_\RR$, and via the isomorphism \eqref{eq:E0R} the Lie algebra $ \Lie \Aut(E_\RR)$ can be identified with (cf. Lemma \ref{lem:LiePic})
\begin{equation}\label{eq:LieAutER}
 \Lie \Aut(E_\RR) \cong  \{(s,B) \; | \; d(B - 2\la s, F_{\theta_\RR} \ra) = 0 \} \subset \Omega^0(\ad P_\RR) \times \Omega^2.
\end{equation}

\begin{proposition}\label{prop:mmapW}
The action of $\Aut(E_\RR)$ on $(\cW_+,\Omega_\ell)$ is Hamiltonian with equivariant moment map
	\begin{equation}\label{eq:mmapclW}
	\la \mu_\ell(W),\zeta \ra = \frac{ \ell-2}{2M_\ell} \int_X B\wedge e^{-\ell f_\omega} \frac{\omega^{n-1}}{(n-1)!}.
	\end{equation}
\end{proposition}

Consider the $\Aut(E_\RR)$-invariant subspace of `integrable' horizontal lifts
\begin{equation}\label{eq:Wint}
\cW^{0} = \{W \in \cW \; | \; [L_W,L_W] \subset L_W\} \subset\cW,
\end{equation}
and define $\cW^{0}_+ = \cW^{0} \cap \cW_+$. Via \eqref{eq:Cherncorrsets}, $\cW_+^{0}$ maps to an open set of the space of liftings of $T^{0,1}X$ to the complexification $E$ of $E_\RR$, which defines 
a complex subspace of $\cL$. Thus, $\cW_+^{0} \subset\cW_+$ is (formally) a complex submanifold, and inherits an exact $(1,1)$-form denoted also by $\Omega_\ell$. Similarly as in Section \ref{sec:Ham}, we define the following group of `Hamiltonian' automorphisms of $E_\RR$. Recall from Lemma \ref{lem:PicAeppli} that there is Lie algebra homomorphism
\begin{equation*}
\mathbf{d} \colon \Lie \Aut(E) \to H^{2}(X,\CC),
\end{equation*}
which defines a normal Lie subalgebra $\Ker \mathbf{d} \subset \Lie \Aut(E)$.

\begin{definition}\label{def:PicAR}
Define the subgroup $\cH \subset \Aut(E_\RR)$ as the set of elements $\underline{f} \in \Aut(E_\RR)$ such that there exists a smooth family $\underline{f}_t \in \Aut(E_\RR)$ with $t \in [0,1]$, satisfying $\underline{f}_0 = \Id_{E_\RR}$, $\underline{f}_1 = \underline{f}$, and 
\begin{equation}\label{eq:Hamcond2}
\mathbf{d}(\zeta_t) = 0, \quad \textrm{ for all $t$.}
\end{equation}
\end{definition}

We are ready to prove the main result of this section.

\begin{proposition}\label{prop:mmapCalabil}
	The $\cH$-action on $(\cW_+^0,\Omega_\ell)$ is Hamiltonian, with equivariant moment map induced by \eqref{eq:mmapclW}. Furthermore, zeros of the moment map are given by solutions of the Calabi system with level $\ell$, defined by
	\begin{equation}\label{eq:Calabil}
	\begin{split}
	F_{\theta_\RR}\wedge \omega^{n-1} & = 0,  \qquad \qquad \qquad \qquad \; F_{\theta_\RR}^{0,2} = 0,\\
	d (e^{-\ell f_\omega}\omega^{n-1}) & = 0, \qquad  dd^c \om + \la F_{\theta_\RR}\wedge F_{\theta_\RR} \ra = 0.
	\end{split}
	\end{equation}
\end{proposition}

\begin{proof}
The integrability condition in the definition of $\cW_+^0$ implies that the pair $(\omega,\theta_\RR)$ associated to $W \in \cW_+^0$ via \eqref{eq:WA} satisfies the two equations on the right-hand side of \eqref{eq:Calabil} (see Proposition \ref{propo:Chernclassic}). Assume that $\la \mu_\ell(W),\zeta \ra = 0$ for all $\zeta \in \Lie \; \cH$. Via the identification \eqref{eq:LieAutER}, the condition $\mathbf{d}(\zeta) = 0$ implies that (see Lemma \ref{lem:PicAeppli})
$$
B - 2\la s, F_{\theta_\RR} \ra = d\xi
$$
for some $\xi \in \Omega^1$. Furthermore, for any $\xi \in \Omega^1$ we have
$$
(s, d\xi + 2\la s, F_{\theta_\RR} \ra) \in \Lie \; \cH.
$$
The two equations on the left-hand side of \eqref{eq:Calabil} follow from Proposition \ref{prop:mmapW}.
\end{proof}

By Proposition \ref{prop:mmapCalabil}, the coupled system \eqref{eq:Calabil} can be regarded as a natural analogue of the Hermite-Yang-Mills equations for string algebroids. These equations were originally found in \cite{grst} for $\ell = 1$ in a holomorphic setting, that is, fixing the string algebroid and calculating the critical points of the dilaton functional $M_\ell$ for compact forms in a fixed Aeppli class (see Proposition \ref{prop:ApPicA}). Following \cite{grst}, we will refer to \eqref{eq:Calabil} as the \emph{Calabi system}. As a matter of fact, when the structure group $K$ is trivial, the solutions of \eqref{eq:Calabil} are in correspondence with (complexified) solutions of the Calabi problem for K\"ahler metrics on $X$ (see the proof of Corollary \ref{cor:Calabimoduli} for a precise statement)
\begin{equation}\label{eq:Calabieq}
\frac{\omega^n}{n!} = c \mu, \qquad d \omega = 0,
\end{equation}
for $c \in \RR_{>0}$, which motivates the name for these equations (see \cite{grst}). Thus, in particular, Proposition \ref{prop:mmapCalabil} yields a new moment map interpretation of this classical problem, which shall be compared with \cite{Fine}.

Assume now that $X$ is a (non-necessarily K\"ahler) Calabi-Yau manifold with holomorphic volume form $\Omega$ and we take $\mu$ as in \eqref{eq:muOmega2}
and $\ell = 1$. In this case, the dilaton functional is given by (see e.g. \cite{grst})
$$
e^{-f_\omega} = \|\Omega\|_\omega,
$$
and therefore Proposition \ref{prop:mmapCalabil} characterizes solutions of the Hull-Strominger system \cite{HullTurin,Strom} as a moment map condition.

\begin{corollary}\label{cor:mmapHS}
Let $(X,\Omega)$ be a Calabi-Yau manifold and let $\mu$ defined by \eqref{eq:muOmega2}. Then, the $\cH$-action on $(\cW_+^0,\Omega_1)$ is Hamiltonian, with equivariant moment map induced by \eqref{eq:mmapclW}. Furthermore, zeros of the moment map are given by solutions of the Hull-Strominger system
	\begin{equation}\label{eq:HS}
	\begin{split}
	F_{\theta_\RR}\wedge \omega^{n-1} & = 0,  \qquad \qquad \qquad \qquad \; F_{\theta_\RR}^{0,2} = 0,\\
	d (\|\Omega\|_\omega \omega^{n-1}) & = 0, \qquad  dd^c \om + \la F_{\theta_\RR}\wedge F_{\theta_\RR} \ra = 0.
	\end{split}
	\end{equation}
\end{corollary}

To the knowledge of the authors, this result provides the first symplectic interpretation of the Hull-Strominger system in the mathematics literature (see \cite{Waldram} for an alternative construction in the physics literature).

\section{Moduli metric and infinitesimal Donaldson-Uhlenbeck-Yau}\label{sec:modulimetric}

\subsection{Gauge fixing}\label{sec:gaugefixing}

Let $X$ be a compact complex manifold of dimension $n$. We fix a smooth volume form $\mu$ compatible with the orientation. The moduli space of solutions of the \emph{Calabi system with level $\ell$} on $(X,\mu)$ is defined as the set of classes of `gauge equivalent' solutions of \eqref{eq:Calabil}. More precisely, it is given by the symplectic quotient
\begin{equation}\label{eq:moduliMell}
\cM_\ell := \mu_\ell^{-1}(0)/\cH,
\end{equation}
where $\mu_\ell$ is the moment map in Proposition \ref{prop:mmapCalabil}.  
In this section we study some basic features of the geometry 
of $\cM_\ell$ and point out some directions for future research. We will proceed formally, ignoring subtleties coming from the theory of infinite dimensional manifolds and Lie groups. 
For simplicity, we will assume that $K$ is semi-simple.

Our first goal is to undertake a \emph{gauge fixing} for solutions of the linearized Calabi system \eqref{eq:Calabil}, whereby the complex structure \eqref{eq:JW} and the symmetric tensor $g_\ell$ in \eqref{eq:glW} descend to the moduli space via symplectic reduction. Difficulties will arise, due to the fact that $g_\ell$ is neither a definite pairing nor non-degenerate (see Remark \ref{rem:glnondeg}). Throughout this section, we fix a real string algebroid $E_\RR$ with principal $K$-bundle $P_h$, the level $\ell \in \RR$, and $W \in \cW^0_+$ solving the Calabi system \eqref{eq:Calabil}, that is, such that $\mu_\ell(W) = 0$. Recall that $W$ determines a holomorphic principal $G$-bundle $P$, a conformally balanced hermitian form $\omega \in \Omega^{1,1}_{>0}$, and a Hermite-Yang-Mills Chern connection $\theta^h$ on $P$ (via the fixed reduction $P_h \subset P$). 
 
We start by characterizing the tangent space to $\cM_\ell$ at $[W]$. By Lemma \ref{lem:JW}, an infinitesimal variation of our horizontal lift $W$ is given by 
$$
(\dot \omega,\dot b, \dot a) \in \Omega^{1,1}_\RR \oplus \Omega^2 \oplus \Omega^1(\ad P_h).
$$

\begin{lemma}\label{lem:Calabillinear+int}
The combined linearization of the Calabi system \eqref{eq:Calabil} and the integrability condition in \eqref{eq:Wint} is given by the linear equations
\begin{equation}\label{eq:Calabillinear+int}
\begin{split}
d^h \dot a \wedge \omega^{n-1} + (n-1) F_h \wedge \dot \omega \wedge \omega^{n-2} & = 0  ,\\
d \Big{(}e^{-\ell f_\omega}\Big{(}(n-1)  \dot \omega \wedge \omega^{n-2} - \frac{\ell}{2} (\Lambda_\omega \dot \omega) \omega^{n-1}\Big{)}\Big{)} & = 0,\\
\dbar \dot a^{0,1} & = 0, \\
d^c \dot \omega + 2\la \dot a, F_h \ra - d \dot b & = 0.
\end{split}
\end{equation}
\end{lemma}
\begin{proof}
The linearization of \eqref{eq:Calabil} is
\begin{equation}\label{eq:Calabillinear0}
\begin{split}
d^h \dot a \wedge \omega^{n-1} + (n-1) F_h \wedge \dot \omega \wedge \omega^{n-2} & = 0  ,\\
d \Big{(}e^{-\ell f_\omega}\Big{(}(n-1)  \dot \omega \wedge \omega^{n-2} - \frac{\ell}{2} (\Lambda_\omega \dot \omega) \omega^{n-1}\Big{)}\Big{)} & = 0,\\
\dbar \dot a^{0,1} & = 0, \\
d(d^c\dot \omega + 2 \la \dot a ,F_h \ra) & = 0,
\end{split}
\end{equation}
while the integrability condition $[L_W,L_W] \subset L_W$ (see \eqref{eq:Wint}) implies at the infinitesimal level that (see Lemma \ref{lemma:liftings} and Lemma \ref{lem:Cherncorr})
\begin{equation}\label{eq:liftingcondinf2}
	\begin{split}
\dbar \dot a^{0,1} & = 0,\\
d \dot b^{0,2} + \dbar (\dot b^{1,1 } - i\dot \omega) - 2 \la \dot a^{0,1},F_h\ra & = 0.
\end{split}
\end{equation}
The second equation in \eqref{eq:liftingcondinf2} yields
$$
d^c \dot \omega = d \dot b - 2\la \dot a, F_h \ra,
$$
and therefore \eqref{eq:liftingcondinf2} implies the last two equations in \eqref{eq:Calabillinear0}. Thus, the tangent to $\mu_\ell^{-1}(0)\subset \cW^0_+$ is characterized by the linear equations \eqref{eq:Calabillinear+int}.
\end{proof}

We denote by ${\bf L} (\dot \om, \dot b, \dot a)$ the differential operator defined by the left-hand side of equations \eqref{eq:Calabillinear+int}. We turn next to the study of the infinitesimal action, in order to define a complex. From the proof of Proposition \ref{prop:mmapCalabil}, we can identify elements $\zeta\in \Lie \cH$ with pairs 
$$
\zeta = (u,B) \in \Lie \; \Omega^0(\ad P_h) \oplus \Omega^2
$$ 
satisfying
\begin{equation}\label{eq:LieHcond}
B - 2 \la u,F_h\ra = d \xi,
\end{equation}
for a real one-form $\xi \in \Omega^1$, and the infinitesimal action at $W$ is
\begin{align}\label{eq:infinitesimalactionmod}
(u,B)\cdot W = (0,B,d^h u) = (0,d\xi + 2 \la u,F_h \ra, d^h u).
\end{align}
Define the vector space
$$
\cR := \Omega^{2n}(\ad P_h) \oplus \Omega^{2n-1} \oplus
\Omega^{0,2}(\ad \underline P) \oplus \Omega^3, 
$$
so that ${\bf L} (\dot \om, \dot b, \dot a)\in \cR$, and consider the complex of degree one differential operators
\begin{equation}
  \label{eq:TWMcomplexring}
(\widehat{S}^*) \qquad \qquad  \Omega^0(\ad P_h) \oplus  \Omega^1 \lra{\widehat{\mathbf{P}}}\Omega^{1,1}_\RR \oplus \Omega^2 \oplus \Omega^1(\ad P_h) \lra{\mathbf{L}} \cR
,
\end{equation}
where
$$
\widehat{\mathbf{P}}(u,\xi) = (0,d\xi + 2\la u,F_h \ra,d^h u).
$$
The cohomology $H^1(\widehat{S}^*) := \frac{\Ker \mathbf{L}}{\Im\, \widehat{\mathbf{P}}}$ can be formally identified with the tangent space $T_{[W}] \cM_\ell$. Observe that the elements of $\Lie \cH$ do not correspond to sections of a vector bundle, due to the condition \eqref{eq:LieHcond}, which we circumvent by introducing the operator $\widehat{\mathbf{P}}$. Our next result shows that the moduli space $\cM_\ell$ is finite dimensional. The proof builds on the infinitesimal moduli construction in \cite{grt}. 


\begin{lemma}\label{lem:sesTWM}
The sequence \eqref{eq:TWMcomplexring} is an elliptic complex of differential
operators. Consequently, the cohomology $H^1(\widehat{S}^*)$ is finite dimensional.
\end{lemma}

\begin{proof}
Ellipticity of \eqref{eq:TWMcomplexring} follows as in \cite[Prop. 4.4]{grt}.
\end{proof}

Our strategy to build a complex structure induced by \eqref{eq:JW} on the moduli space is to work orthogonally to the image of the operator $\widehat{\mathbf{P}}$ with respect to the indefinite pairing $g_\ell$ in \eqref{eq:glW} (cf. \cite{McSisca}). The existence of this complex structure will automatically yield a symmetric tensor of type $(1,1)$, since the two-form $\Omega_\ell$ in \eqref{eq:Omegal2} is well defined on the cohomology $H^1(\widehat{S}^*)$ by Proposition \ref{prop:mmapCalabil}. Our construction relies on a technical condition already found in \cite{grst}, which we explain next. Consider the indefinite $L^2$-pairing on the domain of the operator $\widehat{\mathbf{P}}$ in \eqref{eq:TWMcomplexring} induced by $\omega$ and $\la\, , \ra$
\begin{equation}\label{eq:L2ell}
	\begin{split}
\la(u,\xi),(u,\xi)\ra_\ell  & {} =  \frac{2 - \ell}{M_\ell} \Bigg{(}\int_X \la u,u \ra  \frac{\omega^n}{n!} + \frac{1}{2} \int_X \xi \wedge J \xi \wedge \frac{\omega^{n-1}}{(n-1)!}\Bigg{)},
	\end{split}
\end{equation}
where $M_\ell$ is the value of the functional \eqref{eq:DilfunctionalW} at the solution $W$.

\begin{lemma}\label{lem:P*}
The following operator provides an adjoint of $\widehat{\mathbf{P}}$ for the pairings \eqref{eq:L2ell} and \eqref{eq:glW}
$$
\widehat{\mathbf{P}}^* \colon \Omega^{1,1}_\RR \oplus \Omega^2 \oplus \Omega^1(\ad P_h) \to  \Omega^0(\ad P_h) \oplus  \Omega^1
$$ 
where $\widehat{\mathbf{P}}^* = \widehat{\mathbf{P}}^*_0 \oplus \widehat{\mathbf{P}}^*_1$ is defined by
\begin{align*}
\widehat{\mathbf{P}}^*_0(\dot \omega,\dot b,\dot a) & = \frac{1}{(n-1)!}* \Big{(}e^{-\ell f_\omega}\Big{(} d^h J \dot a \wedge \omega^{n-1} - (n-1)  F_h \wedge \dot b \wedge \omega^{n-2}\Big{)}\Big{)},\\
\widehat{\mathbf{P}}^*_1(\dot \omega,\dot b,\dot a) & =\frac{1}{(n-1)!}*d \Big{(}e^{-\ell f_\omega}\Big{(}(n-1)  \dot b^{1,1} \wedge \omega^{n-2} - \frac{\ell}{2} (\Lambda_\omega \dot b) \omega^{n-1}\Big{)}\Big{)}.
\end{align*}
\end{lemma}
 	
\begin{proof}
The proof follows from a straightforward calculation using integration by parts. Setting $v = (\dot \omega, \dot b,\dot a)$, $y = (u,\xi)$, and using \eqref{eq:glW} and\eqref{eq:Calabil} we have
	\begin{equation*}
	\begin{split}
	g_\ell (v,\widehat{\mathbf{P}}y) & {} =  \frac{\ell - 2}{M_\ell} \int_X \la \dot a \wedge J d^h u \ra \wedge e^{-\ell f_\omega} \frac{\omega^{n-1}}{(n-1)!}\\
&  \phantom{ {} = } -  \frac{2 - \ell}{2 M_\ell} \int_X \dot{b}^{1,1} \wedge (d\xi + 2 \la u,F_h \ra) \wedge e^{-\ell f_\omega} \frac{\omega^{n-2}}{(n-2)!}\\
	& \phantom{ {} = } + \frac{(2 - \ell)\ell}{4 M_\ell} \int_X (\Lambda_\omega \dot b) d\xi \wedge e^{-\ell f_\omega} \frac{\omega^{n-1}}{(n-1)!}\\
	& = \frac{2 - \ell}{M_\ell (n-1)!} \int_X \la u,d^h J \dot a \wedge \omega - (n-1) F_h \wedge \dot b \ra \wedge e^{-\ell f_\omega} \omega^{n-2}\\
&  \phantom{ {} = } -  \frac{2 - \ell}{2 M_\ell (n-1)!} \int_X \xi \wedge d \Big{(}e^{-\ell f_\omega}\Big{(}(n-1)  \dot b^{1,1} \wedge \omega^{n-2} - \frac{\ell}{2} (\Lambda_\omega \dot b) \omega^{n-1}\Big{)}\Big{)}.
	\end{split}
\end{equation*}
The statement follows from $*_{|\Omega^{2n-1}}^2 = -1$ and the action of the Hodge star operator on one-forms
$$
* \xi = J \xi \wedge \frac{\omega^{n-1}}{(n-1)!}.
$$
\end{proof}

Consider now the $L^2$-orthogonal decomposition of $\Omega^1$ induced by the de Rham differential
$$
\Omega^1 =  \Im \; d \oplus \Im \; d^* \oplus \cH^1
$$  
and define a differential operator 
\begin{equation}
   \label{eq:Loperator}
 \begin{array}{cccc}
  \cL : & \Om^0(\ad P_h) \times \Im \; d^*  & \rightarrow & \Om^0(\ad P_h) \times \Im \; d^* \\
         & (u, \xi) & \mapsto & \widehat{\mathbf{P}}^*\circ \widehat{\mathbf{P}}(u,\xi).\\
         \end{array}
 \end{equation}
We state next the key condition on the solution $W$ of \eqref{eq:Calabil} which we need to assume for our argument.

\begin{customcond}{A}\label{ConditionA}
The kernel of $\cL$ vanishes.
\end{customcond}

A geometric characterization of Condition \ref{ConditionA} is mentioned in Remark \ref{rem:ConditionAgeometry}.
On the practical side, this hypothesis will enable us to construct the complex structure on the moduli space. We build on the following result from \cite{grst}. Using $\om$ and a choice of invariant positive-definite bilinear form on $\mathfrak{k}$, we endow the domain of $\cL$ with an $L^2$ norm (possibly different from \eqref{eq:L2ell}, which may be indefinite) and extend the domain of $\cL$ to an appropriate Sobolev completion.

\begin{proposition}[\cite{grst}]\label{prop:zeroindex}
The operator $\cL$ is Fredholm with zero index.
\end{proposition}

Assuming Condition \ref{ConditionA}, we obtain a natural gauge fixing via a $g_\ell$-orthogonal decomposition
\begin{equation}\label{eq:perp}
\Omega^{1,1}_\RR \oplus \Omega^2 \oplus \Omega^1(\ad P_h) = \Im \; \widehat{\mathbf{P}} \oplus (\Im \; \widehat{\mathbf{P}})^{\perp_{g_\ell}}.
\end{equation}

\begin{lemma}\label{lem:gaugefixing}
Assume Condition \ref{ConditionA}. Then, there exists an orthogonal decomposition \eqref{eq:perp} for the pairing $g_\ell$ in \eqref{eq:glW}. Consequently, for any element $v \in \Omega^{1,1}_\RR \oplus \Omega^{1,1}_\RR \oplus \Omega^1(\ad P_h)$ there exists a unique $\Pi v \in \Im \; \widehat{\mathbf{P}}$ such that $(\dot \omega,\dot b, \dot a) = v - \Pi v$ solves the linear equations
\begin{equation*}\label{eq:gaugefixing}
\begin{split}
d \Big{(}e^{-\ell f_\omega}\Big{(}(n-1)  \dot b^{1,1} \wedge \omega^{n-2} - \frac{\ell}{2} (\Lambda_\omega \dot b) \omega^{n-1}\Big{)} & = 0,\\
d^h J \dot a \wedge \omega^{n-1} - (n-1) F_h \wedge \dot b \wedge \omega^{n-2} & = 0.
\end{split}
\end{equation*}
\end{lemma}
\begin{proof}
Notice first that from the non-degeneracy of $\la\,, \ra$, the pairing given in \eqref{eq:L2ell} is non-degenerate. Thus
$$
\ker \, {\bf{\widehat P}}^*= (\Im\,\widehat{\mathbf{P}})^{\perp_{g_\ell}}.
$$
If $v\in\Im \; \widehat{\mathbf{P}} \cap (\Im \; \widehat{\mathbf{P}})^{\perp_{g_\ell}}$, then $v = \widehat{\mathbf{P}}(y)$ for $y \in \Om^0(\ad P_h) \times \Im \; d^*$. But then $\widehat{\mathbf{P}}^*\circ \widehat{\mathbf{P}}(y)=0 $ and, by Condition \ref{ConditionA}, $v=0$. Thus
\begin{equation}\label{eq:glImpnondeg}
\Im \; \widehat{\mathbf{P}} \cap (\Im \; \widehat{\mathbf{P}})^{\perp_{g_\ell}}=\lbrace 0 \rbrace.
\end{equation}
Let $v\in \Omega^{1,1}_\RR \oplus \Omega^{1,1}_\RR \oplus \Omega^1(\ad P_h)$. The condition
$$
v-{\bf \widehat P}(y)\in (\Im \;\widehat{\mathbf{P}})^{\perp_{g_\ell}}
$$
for some $y\in \Om^0(\ad P_h) \times \Im \; d^* $
is equivalent to 
\begin{equation}
 \label{eq:solvingdirectsum}
{\bf \widehat P}^*(v)={\bf{\widehat P}}^*\circ\widehat{\mathbf{P}} (y).
\end{equation}
But by Proposition \ref{prop:zeroindex} and Condition \ref{ConditionA}, $ \bf{\widehat P}^*\circ{\bf \widehat P}$ is surjective. Then, by elliptic regularity, one can solve \eqref{eq:solvingdirectsum} for $y \in \Om^0(\ad P_h) \times \Im \; d^*$. The orthogonal decomposition follows. The last statement of the lemma comes from the expression of $\bf{\widehat P}^*$ in Lemma \ref{lem:P*}.
\end{proof}
The above lemma suggests to define the space of harmonic representatives of the complex \eqref{eq:TWMcomplexring}:
$$
\cH^1(\widehat S^*)=\ker \bf{L} \cap \ker {\bf\widehat P}^*.
$$
Our next result provides our gauge fixing mechanism for the linearization of the Calabi system \eqref{eq:Calabillinear+int}. 

\begin{proposition}\label{prop:Calabillineargaugefixed}
Assume Condition \ref{ConditionA}. Then, the inclusion $\cH^1(\widehat S^*)\subset \ker \bf{L}$ induces an isomorphism
$$
\cH^1(\widehat S^*)\simeq H^1(\widehat S^*).
$$
More precisely, any class in the cohomology $H^1(\widehat S^*)$ of the complex \eqref{eq:TWMcomplexring} admits a unique representative $(\dot \om,\dot{b},\dot a)$ solving the linear equations
\begin{equation}\label{eq:Calabillineargaugefixed}
\begin{split}
d^h \dot a \wedge \omega^{n-1} + (n-1) F_h \wedge \dot \omega \wedge \omega^{n-2} & = 0  ,\\
d \Big{(}e^{-\ell f_\omega}\Big{(}(n-1)  \dot \omega \wedge \omega^{n-2} - \frac{\ell}{2} (\Lambda_\omega \dot \omega) \omega^{n-1}\Big{)}\Big{)} & = 0,\\
\dbar \dot a^{0,1} & = 0, \\
d^c \dot \omega + 2\la \dot a, F_h \ra - d \dot b & = 0,\\
d \Big{(}e^{-\ell f_\omega}\Big{(}(n-1)  \dot b^{1,1} \wedge \omega^{n-2} - \frac{\ell}{2} (\Lambda_\omega \dot b) \omega^{n-1}\Big{)}\Big{)} & = 0,\\
d^h J \dot a \wedge \omega^{n-1} - (n-1) F_h \wedge \dot b \wedge \omega^{n-2} & = 0.
\end{split}
\end{equation}
\end{proposition}

\begin{proof}
The correspondence between $H^1(\widehat S^*)$ and the space of solutions of \eqref{eq:Calabillineargaugefixed} follows from Lemma \ref{lem:Calabillinear+int} and Lemma \ref{lem:gaugefixing}.
\end{proof}

\begin{remark}\label{rem:ConditionAgeometry}
Condition \ref{ConditionA} is secretly a geometric condition. 
To see this, denote by $E$ the complexification of $E_\RR$ and consider the $\Aut(E)$-action on the space of compact forms of $Q := Q_{L_W}$ (see Proposition \ref{lem:retraction}). By Lemma \ref{lem:LiePicactionBQ}, there is a partial inverse for the infinitesimal action which sends an infinitesimal variation $(\dot \omega + \
 \dot \upsilon,iu)$ of $E_\RR$ to the Lie algebra element  $\zeta(\dot \omega + \dot \upsilon,iu) = (i u, - i \dot \omega  + i\Im \; \dot \upsilon) \in \Lie \Aut(E)$. Denote by $\Aut(Q)$ the group of automorphisms of $Q$. Then, one can prove that a solution $W$ of the Calabi system \eqref{eq:Calabil} with $h^{0}(\ad P) = 0$ satisfies Condition \ref{ConditionA} if and only if the following holds: an infinitesimal variation $(\dot \omega + \dot \upsilon,iu)$ of $E_\RR$ along the Aeppli class $[E_\RR] \in \Sigma_A(Q,\RR)$ solves the linearization of the Calabi system 
only if $\zeta(\dot \omega + \dot \upsilon,iu) \in \Lie \Aut(Q)$. This shall be compared with a classical result in K\"ahler geometry, which states that solutions of the linearized constant scalar curvature equation, for K\"ahler metrics in a fixed K\"ahler class, are in bijective correspondence with Hamiltonian Killing vector fields. 
\end{remark}

\subsection{The moduli space metric}\label{sec:metric}

We are ready to prove our main result, which shows that the gauge fixing in Proposition \ref{prop:Calabillineargaugefixed} enables us to descend the complex structure \eqref{eq:JW} and the symmetric tensor $g_\ell$ in \eqref{eq:glW} to the moduli space $\cM_\ell$, via the symplectic reduction in Proposition \ref{prop:mmapCalabil}.

\begin{theorem}\label{thm:metric}
Assume Condition \ref{ConditionA}.
Then, the tangent space to $\cM_\ell$ at $[W]$, identified with the space of solutions of the gauge fixed linear equations
\eqref{eq:Calabillineargaugefixed}, inherits a complex structure $\mathbf{J}$ and a (possibly degenerate) metric $g_\ell$  such that $\Omega_\ell = g_\ell(\mathbf{J},)$, given respectively by \eqref{eq:JW} and \eqref{eq:glW}, and where $\Om_\ell$ stands for the restriction of \eqref{eq:OmegalW}.
\end{theorem}
\begin{proof}
The fact that $H^1(S^*)$ inherits a complex structure follows from Proposition \ref{prop:Calabillineargaugefixed}, using that $\mathbf{J}$ in \eqref{eq:JW} preserves \eqref{eq:Calabillineargaugefixed}. The formula for the metric is a direct consequence of Lemma \ref{lem:lambdalW} and Proposition \ref{prop:mmapW}.
\end{proof}

\begin{remark}
Assume that $h^{0,2}_{\dbar}(X) = 0$ and $h^{0}(\ad P) = 0$, where $h^{0,2}_{\dbar}(X)$ denotes the dimension of the $(0,2)$ Dolbeault cohomology group and $h^{0}(\ad P) = \dim H^{0}(\ad P)$. Then, it is not difficult to see that any $[(\dot \omega,\dot b,\dot a)] \in H^1(\widehat{S}^*)$ admits a representative with $\dot b = \dot b^{1,1}$. Thus, relying on Remark \ref{rem:glnondeg}, we expect that \eqref{eq:glW} leads to a non-degenerate metric at least for $\ell > 2 - \frac{2}{n}$.
\end{remark}

We study next the structure of the metric \eqref{eq:glW} along the fibres of a natural map from $\cM_\ell$ to the moduli space of holomorphic principal $G$-bundles. As we will see shortly, the moduli space metric constructed in Theorem \ref{thm:metric} is `semi-topological', in the sense that fibre-wise it can be expressed in terms of classical cohomological quantities associated to a gauge-fixed variation of the solution. Denote by
$$
\cA^0 = \{ \theta_\RR \in \cA \; | \; F_{\theta_\RR}^{0,2}  =0\}
$$
the space of integrable connections on $P_h = P_\RR$. Via the classical Chern correspondence, we can identify $\cA^0$ with the space of structures of holomorphic principal $G$-bundle on $\underline{P} := P_h \times_K G$, which we denote by $\cC^0$, obtaining a well-defined map
\begin{equation}\label{eq:modulimapWbundle}
\cM_\ell \to \cC^0/\cG_{\underline{P}}.
\end{equation}
By standard theory, $\cC^0/\cG_{\underline P}$ is the well-studied moduli space of holomorphic principal $G$-bundles over $X$ with fixed topological bundle $\underline{P}$.
 As before, we fix a solution $W$ of \eqref{eq:Calabil} and consider the corresponding point$$
[P] \in \cC^0/\cG_{\underline{P}}.
$$ 
We start by characterizing the tangent space to the fibre of \eqref{eq:modulimapWbundle} over the class $[P]$, using the gauge fixing in Proposition \ref{prop:Calabillineargaugefixed}.

\begin{lemma}\label{lem:Calabillinearfibre}
Assume Condition \ref{ConditionA} and $h^{0}(\ad P) = 0$. Then, any infinitesimal variation in the fibre of \eqref{eq:modulimapWbundle} over $[P]$ at $[W]$ admits a unique representative of its class in $H^1(S^*)$ of the form $(\dot \omega,\dot b, - Jd^hs+d^hs')$, for $s,s' \in \Omega^0(\ad P_h)$, solving the linear equations

\begin{equation}\label{eq:Calabillinearfiberfull}
\begin{split}
- d^h J d^h s \wedge \omega^{n-1} + (n-1) F_h \wedge \dot \omega \wedge \omega^{n-2} & = 0  ,\\
d \Big{(}e^{-\ell f_\omega}\Big{(}(n-1)  \dot \omega \wedge \omega^{n-2} - \frac{\ell}{2} (\Lambda_\omega \dot \omega) \omega^{n-1}\Big{)}\Big{)} & = 0,\\
d^c (\dot \omega - 2\la s, F_h \ra) - d(\dot b - 2\la s', F_h \ra) & = 0,\\
d \Big{(}e^{-\ell f_\omega}\Big{(}(n-1)  \dot b^{1,1} \wedge \omega^{n-2} - \frac{\ell}{2} (\Lambda_\omega \dot b) \omega^{n-1}\Big{)}\Big{)} & = 0,\\
- d^h J d^h s' \wedge \omega^{n-1} + (n-1) F_h \wedge \dot b \wedge \omega^{n-2} & = 0.
\end{split}
\end{equation}
\end{lemma}

\begin{proof}
Let $(\dot \omega,\dot b, \dot a) \in \Omega^{1,1}_\RR \oplus \Omega^2 \oplus \Omega^1(\ad P_h)$ be an infinitesimal variation of the solution $W$ of \eqref{eq:Calabil}. Assuming that it is tangent to the fibre over $[P]$, there exists $r \in \Omega^0(\ad \underline{P})$ such that
$$
\dot a^{0,1} = \dbar r.
$$
Then we can write uniquely
$$
\dot a = - J d^h s + d^h s'
$$
for $s , s' \in \Omega^0(\ad P_h)$. The statement follows from Proposition \ref{prop:Calabillineargaugefixed} using that $(d^h)^2 s \wedge \omega^{n-1} = [F_h,s] \wedge \omega^{n-1} = 0$ by \eqref{eq:Calabil}.
\end{proof}

\begin{remark}
Using that $\theta^h$ is Hermite-Yang-Mills and that $h^0(\ad P)$ vanishes, by the first and last equations in \eqref{eq:Calabillinearfiberfull} the elements $s$ and $s'$ are uniquely determined by $\dot \omega$ and $\dot b$. 
\end{remark}

The gauge fixed system \eqref{eq:Calabillinearfiberfull} for variations along the fibres of \eqref{eq:modulimapWbundle} allows us to define Aeppli and Bott-Chern cohomology classes. Recall that
\begin{equation}\label{eq:A-cohomology}
\begin{split}
H^{p,q}_{A}(X) & = \frac{\Ker(dd^c\colon \Omega^{p,q} \to \Omega^{p+1,q+1})}{\Im(\partial \oplus \dbar \colon \Omega^{p-1,q}\oplus \Omega^{p,q-1} \to \Omega^{p,q})},\\
H^{p,q}_{BC}(X) & = \frac{\Ker(d\colon \Omega^{p,q} \to \Omega^{p+1,q}\oplus \Omega^{p,q+1})}{\Im(dd^c\colon \Omega^{p-1,q-1} \to \Omega^{p,q})}.
\end{split}
\end{equation}
Let $(\dot \omega,\dot b, -Jd^hs+d^hs')$ be, as in Lemma \ref{lem:Calabillinearfibre}, a solution of \eqref{eq:Calabillinearfiberfull}.
From the third equation in \eqref{eq:Calabillinearfiberfull} we obtain
$$
dd^c (\dot \omega - 2\la s, F_h \ra) = 0, \qquad d d^c (\dot b - 2\la s', F_h \ra) = 0,
$$
and we can define the variation of the `complexified Aeppli class' of the solution (cf. Proposition \ref{prop:ApPicA}) by
\begin{align*}
\dot{\mathfrak{a}} & = \Re \; \dot{\mathfrak{a}}+i \Im \; \dot{\mathfrak{a}}\\
                   & = [\dot \omega - 2\la s,F_h \ra] + i[\dot b- 2\la s', F_h \ra ] \in H^{1,1}_A(X).
\end{align*}
Notice that, by Lemma \ref{lem:WA}, the balanced class
$$
\mathfrak{b} = \frac{1}{(n-1)!}[e^{-\ell f_\omega}\omega^{n-1}] \in H^{n-1,n-1}_{BC}(X,\RR)
$$
is independent of the representative in $[W] \in \cM_\ell$. Thus, using the second and fourth equations in \eqref{eq:Calabillinearfiberfull}, we define the variations of the `complexified balanced class' by
\begin{align*}
\dot{\mathfrak{b}} & = \Re \; \dot{\mathfrak{b}} + i \Im \; \dot{\mathfrak{b}}\\
                   & = [\Re \; \dot \nu] + i[\Im \; \dot \nu] \in H^{n-1,n-1}_{BC}(X),
\end{align*}
where $\dot \nu \in \Omega^{n-1,n-1}$ is defined by 
$$
(n-1)!\Re \; \dot \nu :=e^{-\ell f_\omega}(n-1)  \dot \omega_0 \wedge \omega^{n-2} + \frac{n(2- \ell)-2}{2n} e^{-\ell f_\omega} (\Lambda_\omega \dot \omega) \omega^{n-1},
$$
$$
(n-1)!\Im \; \dot \nu :=e^{-\ell f_\omega}(n-1)  \dot b_0 \wedge \omega^{n-2} + \frac{n(2- \ell)-2}{2n} e^{-\ell f_\omega} (\Lambda_\omega \dot b) \omega^{n-1}.
$$
The subscript $0$ stands for the primitive $(1,1)$-forms
$$
\dot \omega_0 = \dot \omega - \frac{1}{n} (\Lambda_\omega \dot \omega)\omega,\qquad \dot b_0 = \dot b - \frac{1}{n} (\Lambda_\omega \dot b) \omega.
$$
The variation of the balanced class $\mathfrak{b}$ of $\omega$ corresponds in our notation to $\Re \; \dot{\mathfrak{b}}$. For the next result, we use the duality pairing $H^{1,1}_A(X) \cong H^{n-1,n-1}_{BC}(X)^*$ between the Aeppli and Bott-Chern cohomologies.

\begin{lemma}
The pairing between $\Re \; \dot{\mathfrak{b}}$ and $\Re \; \dot{\mathfrak{a}}$ is given by:
\begin{equation*}\label{eq:L2normdilaton0}
\begin{split}
\Re \; \dot{\mathfrak{b}} \cdot  \Re \; \dot{\mathfrak{a}} & = - \int_X |\dot \omega_0|^2 e^{-\ell f_\omega} \frac{\omega^{n}}{n!} + \frac{n(2- \ell)-2}{2n}\int_X e^{-\ell f_\omega} |\Lambda_\omega \dot \omega|^2 \frac{\omega^n}{n!} \\
	& \phantom{ {}= } + 2 \int_X \la d^h s \wedge J d^h s\ra \wedge e^{-\ell f_\omega}\frac{\omega^{n-1}}{(n-1)!}.
\end{split}
	\end{equation*}
\end{lemma}
\begin{proof}
Define $\tilde \omega = e^{-\ell f_\omega/n-1} \omega$. Using that $\tilde \omega$ is balanced, we have
	\begin{align*}
	\Delta_{\tilde \omega}\la s,s\ra  :&= 2i\Lambda_{\tilde \omega}\dbar \partial \la s,s\ra \\
	& = 4i\la \Lambda_{\tilde \omega}\dbar \partial^hs,s\ra + 2 \Lambda_{\tilde \omega} \la(d^h)^cs \wedge d^h s\ra\\
	& = 4i\la\Lambda_{\tilde \omega}\dbar \partial^hs ,s\ra + 2 \Lambda_{\tilde \omega} \la Jd^h s \wedge d^h s\ra.
	\end{align*}
By equation $F_h \wedge \omega^{n-1} = 0$, we can express $d^hJd^h$ as follows
\begin{equation*}\label{eq:linHEHYM}
- (d^h J d^h s) \wedge \omega^{n-1} 
 = (2 i \dbar \partial^h s) \wedge \omega^{n-1} - [F_h,s] \wedge \omega^{n-1} = (2 i \dbar \partial^h s) \wedge \omega^{n-1}
\end{equation*}
and hence the first equation in \eqref{eq:Calabillinearfiberfull} gives
\begin{align*}
	\Delta_{\tilde \omega}\la s,s\ra \frac{\tilde \omega^{n}}{n!} & = - 2\la F_h,s\ra \wedge \dot \omega_0 \wedge \frac{e^{-\ell f_\omega}\omega^{n-2}}{(n-2)!} + 2 \Lambda_{\tilde \omega} \la Jd^h s \wedge d^h s\ra \frac{\tilde \omega^{n}}{n!}.
\end{align*}
Finally, we calculate
	\begin{align*}
\Re \; \dot{\mathfrak{b}} \cdot \Re \; \dot{\mathfrak{a}} 
	& =  \int_X \Re \; \dot \nu \wedge (\dot \omega_0 + (\Lambda_\omega \dot \omega)\omega/n - 2 \la s ,F_h\ra)\\
	& = \int_X \dot \omega_0 \wedge \dot \omega_0 \wedge e^{-\ell f_\omega} \frac{\omega^{n-2}}{(n-2)!} + \frac{n(2- \ell)-2}{2n}\int_X e^{-\ell f_\omega} |\Lambda_\omega \dot \omega|^2 \frac{\omega^n}{n!} \\
	& \phantom{ {} =} - 2\int_X \la s,F_h\ra\wedge e^{-\ell f_\omega}\dot \omega_0 \wedge \frac{\omega^{n-2}}{(n-2)!}\\
		& = - \int_X |\dot \omega_0|^2 e^{-\ell f_\omega} \frac{\omega^{n}}{n!} + \frac{n(2- \ell)-2}{2n}\int_X e^{-\ell f_\omega} |\Lambda_\omega \dot \omega|^2 \frac{\omega^n}{n!} \\
	& \phantom{{} =}  + 2 \int_X \Lambda_{\tilde \omega} \la d^h s \wedge J d^h s\ra \frac{\tilde \omega^{n}}{n!}.
	\end{align*}
\end{proof}

Note that we have a similar formula for the pairing $\Im \; \dot{\mathfrak{b}} \cdot \Im \; \dot{\mathfrak{a}}$.
We calculate next our formula for the metric in the fibres of \eqref{eq:modulimapWbundle}.

\begin{theorem}\label{thm:metricfibre}
Assume Condition \ref{ConditionA} and $h^{0}(\ad P) = 0$.
Let $(\dot \om, \dot b, -Jd^hs+d^hs')$ be an element in the tangent of the fiber of \eqref{eq:modulimapWbundle} solving equations \eqref{eq:Calabillinearfiberfull}. Denote by $\dot{\mathfrak{b}}$ and $\dot{\mathfrak{a}}$ the associated variations of complex Bott-Chern class and Aeppli class. Then
\begin{equation}\label{eq:metricfibre}
\begin{split}
g_\ell & = 
\frac{2 - \ell}{2 M_\ell}  \Bigg{(} \frac{2-\ell}{2M_\ell} (\Re \; \dot{\mathfrak{a}} \cdot \mathfrak{b})^2 - \Re \; \dot{\mathfrak{a}} \cdot \Re \; \dot{\mathfrak{b}} + \frac{2-\ell}{2M_\ell} (\Im \; \dot{\mathfrak{a}} \cdot \mathfrak{b})^2 - \Im \; \dot{\mathfrak{a}} \cdot\Im \; \dot{\mathfrak{b}} \Bigg{)}
\end{split}
	\end{equation}
\end{theorem}

\begin{proof}
The proof follows from Theorem \ref{thm:metric} and Lemma \ref{lem:Calabillinearfibre} by a straightforward calculation. E.g., for $v = (\dot \omega, 0,-Jd^h s)$ we have
\begin{equation*}
	\begin{split}
	g_\ell(v,v) & {} =  \frac{\ell - 2}{M_\ell} \int_X \la d^h s \wedge J d^h s \ra \wedge e^{-\ell f_\omega} \frac{\omega^{n-1}}{(n-1)!} +  \frac{2 - \ell}{2 M_\ell} \int_X |\dot \omega_0^{1,1}|^2 e^{-\ell f_\omega} \frac{\omega^{n}}{n!}\\
	& \phantom{ {} = } + \frac{2 - \ell}{2 M_\ell}\Bigg{(}\frac{\ell}{2} - \frac{n-1}{n} \Bigg{)} \int_X | \Lambda_\omega \dot \omega|^2 e^{-\ell f_\omega} \frac{\omega^{n}}{n!} + \left(\frac{2-\ell}{2M_\ell}\right)^2( \Re \; \dot{\mathfrak{a}} \cdot \mathfrak{b})^2\\
& = \frac{2- \ell}{2M_\ell} \Bigg{(} - 2\int_X \la d^h s \wedge J d^h s \ra \wedge e^{-\ell f_\omega} \frac{\omega^{n-1}}{(n-1)!} + \int_X |\dot \omega_0^{1,1}|^2 e^{-\ell f_\omega} \frac{\omega^{n}}{n!} \Bigg{)}\\
	& \phantom{ {} = } + \frac{2 - \ell}{2 M_\ell}\Bigg{(}- \frac{n(2 - \ell) - 2}{2n} \int_X | \Lambda_\omega \dot \omega|^2 e^{-\ell f_\omega} \frac{\omega^{n}}{n!}\Bigg{)} + \left(\frac{2-\ell}{2M_\ell}\right)^2( \Re \; \dot{\mathfrak{a}} \cdot \mathfrak{b})^2\\
& = \frac{2 - \ell}{2 M_\ell}(- \Re \; \dot{\mathfrak{b}} \cdot  \Re \; \dot{\mathfrak{a}} ) + \left(\frac{2-\ell}{2M_\ell}\right)^2( \Re \; \dot{\mathfrak{a}} \cdot \mathfrak{b})^2	.
	\end{split}
\end{equation*}
\end{proof}

When the structure group $K$ is trivial, the solutions of \eqref{eq:Calabil} are in correspondence with (complexified) solutions of the Calabi problem for K\"ahler metrics on $X$ (see \eqref{eq:Calabieq}). 
In the next result we show that, when $\ell < 2$, formula \eqref{eq:metricfibre} defines a positive-definite K\"ahler metric on the `complexified K\"ahler moduli space' of K\"ahler metrics on $X$ with prescribed volume form.

\begin{corollary}\label{cor:Calabimoduli}
Assume that $K$ is trivial. Then, the moduli space $\cM_\ell$ in \eqref{eq:moduliMell} is non-empty if and only if $X$ is K\"ahler and $E_\RR$ is isomorphic to the standard exact Courant algebroid $T \underline X \oplus T^* \underline X$. In that case, there is a bijection
\begin{equation}\label{eq:modulibijectionexact}
\cM_\ell \cong \mathcal{K}_X \times H^{0,2}(X)
\end{equation}
where $\mathcal{K}_X \subset H^{1,1}(X)$ denotes the complexification of the K\"ahler cone of $X$. Furthermore, provided that $h^{0,2}(X) = 0$ and $\ell < 2$, \eqref{eq:metricfibre} defines a positive-definite K\"ahler metric on $\cM_\ell$ with K\"ahler potential
\begin{equation}\label{eq:K0exact}
\mathcal{K} 
= - \frac{2-\ell}{2}\log \Big{(}(\Re \; \mathfrak{a})^n/n!\Big{)} - \frac{\ell}{2}\log \int_X \mu
\end{equation}
for $\mathfrak{a} \in H^{1,1}(X)$.
\end{corollary} 

\begin{proof}
As mentioned in Section \ref{sec:mmapHS}, when $K$ is trivial, a solution of \eqref{eq:Calabil} satisfies \eqref{eq:Calabieq}. Therefore, $\cM_\ell \neq \emptyset$ implies that $E_\RR \cong T \underline X \oplus T^* \underline X$ and that $X$ is K\"ahler (see \eqref{eq:HthetaChern}).  Using this fact combined with Lemma \ref{lem:JW}, it is not difficult to see that $\cM_\ell$ is bijective with the quotient
$$
\{(\omega,b) \in \Omega^{1,1}_{>0} \times \Omega^2 \; | \; df_\omega = 0, d\omega = 0, db = 0 \}/\{ (0,d\xi) \; | \; \xi \in \Omega^1  \}.
$$
Applying Yau's solution of the Calabi Conjecture \cite{Yau0}, it follows that $\cM_\ell \neq \emptyset$ provided that $E_\RR \cong T \underline X \oplus T^* \underline X$ and that $X$ is K\"ahler.

Using the Hodge decomposition for $H^2(X,\CC)$ it in not difficult to see that the bijection \eqref{eq:modulibijectionexact} is defined by 
$$
[(\omega,b)] \mapsto (\mathfrak{a},[b^{0,2}]) \in \mathcal{K}_X \times H^{0,2}(X),
$$
where $\mathfrak{a} := [\omega + ib^{1,1}]$ and we have used the $\partial\dbar$-Lemma to identify $[b^{1,1}] \in H^{1,1}_A(X,\RR) \cong H^{1,1}(X,\RR)$. The condition $d f_\omega = 0$ implies
$$
M_\ell = e^{-\ell f_\omega} \frac{(\Re \; \mathfrak{a})^n}{n!}, \qquad \mathfrak{b} = \frac{e^{-\ell f_\omega}}{(n-1)!}\mathfrak{a}^{n-1}
$$
for any $[(\omega,b)] \in \cM_\ell$, and also the equalities 
\begin{align*}
\Re \; \dot{\mathfrak{b}} & = e^{-\ell f_\omega}\Bigg{(}\frac{(\Re \; \dot{\mathfrak{a}}_0) \cdot (\Re \; \mathfrak{a})^{n-2}}{(n-2)!} + \frac{(n(2- \ell)-2)}{2}\frac{\Re \; \dot{\mathfrak{a}} \cdot (\Re \; \mathfrak{a})^{n-1}}{(\Re \; \mathfrak{a})^n}\frac{(\Re \; \mathfrak{a})^{n-1}}{(n-1)!}\Bigg{)},\\
\Im \; \dot{\mathfrak{b}} & = e^{-\ell f_\omega}\Bigg{(}\frac{(\Im \; \dot{\mathfrak{a}}_0) \cdot (\Re \; \mathfrak{a})^{n-2}}{(n-2)!} + \frac{(n(2- \ell)-2)}{2}\frac{\Im \; \dot{\mathfrak{a}} \cdot (\Re \; \mathfrak{a})^{n-1}}{(\Re \; \mathfrak{a})^n}\frac{(\Re \; \mathfrak{a})^{n-1}}{(n-1)!}\Bigg{)},
\end{align*}
where $\dot{\mathfrak{a}}_0$ stands for the primitive part of $\dot{\mathfrak{a}}$ via the Lefschetz decomposition, and we have used again the $\partial\dbar$-Lemma to identify $H^{1,1}_A(X) \cong H^{1,1}(X)$ and $H^{n-1,n-1}_{BC}(X) \cong H^{n-1,n-1}(X)$. By \cite[Prop. 5.16]{grst}, Condition \ref{ConditionA} holds for any point in $\cM_\ell$ and Theorem \ref{thm:metricfibre} applies. Substituting the previous formulae in \eqref{eq:metricfibre} we obtain
\begin{equation}\label{eq:gellK=1}
\begin{split}
g_\ell =  {} & - \frac{(2 - \ell)n!}{2(\Re \; \mathfrak{a})^{n}} \Bigg{(}\frac{(\Re \; \dot{\mathfrak{a}}_0)^2 \cdot (\Re \; \mathfrak{a})^{n-2}}{(n-2)!} + \frac{(\Im \; \dot{\mathfrak{a}}_0)^2 \cdot (\Re \; \mathfrak{a})^{n-2}}{(n-2)!} \Bigg{)}\\
& + \frac{n(2-\ell)}{2((\Re \; \mathfrak{a})^{n})^2} \Bigg{(} (\Re \; \dot{\mathfrak{a}} \cdot (\Re \; \mathfrak{a})^{n-1})^2 + (\Im \; \dot{\mathfrak{a}} \cdot (\Re \; \mathfrak{a})^{n-1})^2 \Bigg{)}
\end{split}
\end{equation}
and therefore, using that $ \tilde{\mathfrak{a}}_0^2 \cdot (\Re \; \mathfrak{a})^{n-2} < 0$ for any non-trivial primitive class $\tilde{\mathfrak{a}}_0 \in H^{1,1}(X,\RR)$, it follows that $g_\ell$ is positive provided that $h^{0,2}(X) = 0$ and $\ell < 2$. The metric is K\"ahler by Proposition \ref{prop:mmapCalabil} and formula \eqref{eq:moduliMell}, and \eqref{eq:K0exact} follows from
$$
M_\ell = \Bigg{(}\frac{(\Re \; \mathfrak{a})^n}{n!\int_X \mu} \Bigg{)}^{\frac{2-\ell}{2}} \int_X \mu.
$$
%

\end{proof}

A case of special interest where the previous result applies is when $X$ admits a holomorphic volume form $\Omega$ and we take $\mu$ as in \eqref{eq:muOmega2}
and $\ell = 1$. In this case, \eqref{eq:Calabil} is equivalent to the condition of $\operatorname{SU}(n)$-holonomy for the metric and \eqref{eq:gellK=1} matches (up to homothety) Strominger's formula for the special K\"ahler metric on the `complexified K\"ahler moduli' for $X$ \cite[Eq. (4.1)]{CD}. As a consequence, this classical moduli space is recovered, along with its Weil-Petersson metric, via pseudo-K\"ahler reduction in Corollary \ref{cor:Calabimoduli}. It is interesting to observe that the formula for the \emph{holomorphic prepotential} on a Calabi-Yau threefold, given by the natural cubic form on $H^{1,1}(X)$, breaks as soon as we split the K\"ahler class into the Aeppli and Bott-Chern parameters $\mathfrak{a}$ and $\mathfrak{b}$. 

On a (non-necessarily K\"ahler) Calabi-Yau threefold $(X,\Omega)$ and for a suitable choice of the structure group $K$, the equations \eqref{eq:Calabil} are equivalent to the Hull-Strominger system \cite{HullTurin,Strom} provided that $\ell = 1$ and we take $\mu$ as in \eqref{eq:muOmega2} (see Corollary \ref{cor:mmapHS}). In this case, our formula 
for the moduli space K\"ahler potential reads
\begin{equation}\label{eq:K0}
\mathcal{K} = - \log \int_X \|\Omega\|_\omega \frac{\omega^3}{6}.
\end{equation}
Formula \eqref{eq:K0} shall be compared with \cite[Eq. (1.3)]{CDM}. For this interesting system of equations, the physics of string theory predicts that the moduli space metric \eqref{eq:glW} should be positive definite along the fibres of \eqref{eq:modulimapWbundle}. More precisely, we have the following \emph{physical conjecture}.

\begin{conjecture}\label{conj:Kahlerpotential0}
Formula \eqref{eq:K0} defines the K\"ahler potential for a K\"ahler metric in the moduli space of solutions of the Hull-Strominger system, for fixed bundle $P$ and fixed Calabi-Yau threefold $(X,\Omega)$.
\end{conjecture}

Conjecture \ref{conj:Kahlerpotential0} is based on a Gukov-type formula \cite{Gukov} for the \emph{gravitino mass} in $4$-dimensional heterotic flux compatifications derived in \cite{GLM}, combined with \eqref{eq:K0} and a universal relation between the moduli K\"ahler potential, the \emph{superpotential}, and the gravitino mass. Further details on physical aspects of this conjecture will appear elsewhere.

Combined with Theorem \ref{thm:metricfibre}, Conjecture \ref{conj:Kahlerpotential0} leads us to an interesting physical prediction relating the variations of the Aeppli classes and balanced classes of solutions in the special case of the Hull-Strominger system on a Calabi-Yau threefold.

\begin{conjecture}\label{conj:theorem2}
If $(X,\Omega,P)$ admits a solution of the Hull-Strominger system, then \eqref{eq:metricfibre} is positive definite. In particular, the variations of the Aeppli and balanced classes of nearby solutions must satisfy
\begin{equation}\label{eq:ineqmoduli}
\Re \; \dot{\mathfrak{a}} \cdot \Re \; \dot{\mathfrak{b}} < \frac{1}{2\int_X \|\Omega\|_\omega \frac{\omega^3}{6}} (\Re \; \dot{\mathfrak{a}} \cdot \mathfrak{b})^2.
\end{equation}
\end{conjecture}

Formula \eqref{eq:ineqmoduli} provides a potential obstruction to the existence of solutions of the Hull-Strominger system around a given solution. For example, if we fix $\Re \; \dot{\mathfrak{a}}$, the possible variations in the balanced class $\Re \; \dot{\mathfrak{b}}$ are constrained by the duality pairing $\Re \; \dot{\mathfrak{a}} \cdot \Re \; \dot{\mathfrak{b}}$, via an effective bound in terms of the balanced class of the given solution and the value of the dilaton functional. We expect this phenomenon to be related to some global obstruction to the existence of solutions. It would be interesting to obtain a physical explanation for the inequality \eqref{eq:ineqmoduli}. 

\subsection{Infinitesimal Donaldson-Uhlenbeck-Yau}\label{sec:DUY}

We discuss next the relation between $\cM_\ell$ and the moduli space of string algebroids $Q$ over $X$ with fixed class $[E_Q] = [E] \in H^1(\underline{\cS})$ (see Lemma \ref{lem:smap}). This relation is suggested by the correspondence between the moduli space of solutions of the Hermite-Yang-Mills equations and the moduli space of polystable principal bundles, given by the Donaldson-Uhlenbeck-Yau Theorem \cite{Don,UY}. 
In the case of our interest, $E$ is the complexification of a compact form $E_\RR$. 

In order to establish this relation, notice that the proof of Lemma \ref{lem:WA} shows that \eqref{eq:AutEr-left-actioncW} extends to a left $\Aut(E)$-action
\begin{equation}\label{eq:AutE-left-actioncW}
\begin{split}
\Aut(E) \times \cW & \longrightarrow \cW\\
(\underline{f},W) & \longmapsto  \underline{f} \cdot W : = \underline{f}(W')
\end{split}
\end{equation}
where $W' := W(\underline{f}^{-1}(E_\RR),L_W) \subset \underline{f}^{-1}(E_\RR)$ is the horizontal subspace induced by the Chern correspondence in Lemma \ref{lem:Cherncorr}. Similarly as in Lemma \ref{lem:WA}, the forgetful map \eqref{eq:WA} jointly with the action \eqref{eq:AutE-left-actioncW} induce a commutative diagram
\begin{equation*}
\xymatrix @R=.5pc {
		\Aut(E) \times \cW \ar[r] \ar[dd] & \cW \ar[dd] \\
		& \\
		\Ker \sigma_{\underline P} \times \cA \ar[r] & \cA ,
	}
\end{equation*}
where $\Ker \sigma_{\underline P} \subset \cG_{\underline P}$ is as in Corollary \ref{cor:PicQ} and the bottom arrow is induced by the action of the complex gauge group $\cG_{\underline P}$ on $\cA $ (see e.g. \cite{Don}). Consider the isomorphism $\cA^0 \cong \cC^0$ between the space of integrable connections $\cA^0$ on $P_h$ and the space $\cC^0$ of holomorphic principal $G$-bundle structures on $\underline{P}$, given by the classical Chern correspondence. Consider the subgroup $\Aut_{dR}(E) \subset \Aut(E)$ as in Definition \ref{def:PicA}. Then, the set-theoretical Chern correspondence in Lemma \ref{lem:Cherncorrsets} induces a diagram
\begin{equation}\label{eq:modulidiagram}
\xymatrix @R=.5pc {
\cM_\ell \ar[r] & \cW^0/\Aut_{dR}(E) \ar[r]^{\cong} \ar[dd] & \cL^0/\Aut_{dR}(E) \ar[dd] \\
	&	& \\
	&	 \cW^0/\Aut(E) \ar[r]^{\cong} \ar[dd] & \cL^0/\Aut(E) \ar[dd] \ar[dd],\\
		&	& \\
	&	 \cA^0/\Ker \sigma_{\underline P} \ar[r]^{\cong} \ar[dd] & \cC^0/\Ker \sigma_{\underline P} \ar[dd],\\
		&	& \\
	&	 \cA^0/\cG_{\underline P} \ar[r]^{\cong} & \cC^0/\cG_{\underline P},
	}
\end{equation}
where $\cL^0$ denotes the space of liftings of $T^{0,1}X$ to $E$ and $\sigma_{\underline P}$ is as in \eqref{eq:Picardses}.

Let us analyse briefly the tower of moduli spaces on the right-hand side of the diagram  \eqref{eq:modulidiagram}. 
Firstly, $\cC^0/\cG_{\underline P}$ is the moduli space of holomorphic principal $G$-bundles over $X$ with fixed topological bundle $\underline{P}$, as considered in Section \ref{sec:metric}. The fibre of the map 
$$
\cC^0/\Ker \sigma_{\underline P} \to \cC^0/\cG_{\underline P}
$$ 
over $[P]$ is discrete (see \eqref{eq:PicardsesLie}). Assuming that the automorphism group $\cG_P$ of $P$ is trivial, the fibre is parametrized by $\Im \; \sigma_{\underline P} \subset H^3(X,\CC)$. As for the moduli space $\cL^0/\Aut(E)$, we have the following.

\begin{lemma}\label{lem:ModuliBC}
The set $\cL^0/\Aut(E)$ parametrizes isomorphism classes of string algebroids $Q$ over $X$ with $[E_Q] = [E] \in H^1(\underline{\cS})$ (see Lemma \ref{lem:smap}).
\end{lemma}

\begin{proof}
Any element $L \in \cL^0$ determines a string algebroid $Q_L$ with $[E_{Q_L}] = [E]\in H^1(\underline{\cS})$ (see Lemma \ref{lemma:bricks}). 
If $L$ and $L'$ are in the same $\Aut(E)$-orbit then by \eqref{eq:Morita2.0} it follows that $Q_L$ and $Q_{L'}$ are isomorphic. Conversely, given a string algebroid $Q$ with $[E_Q] = [E]$, then any choice of isomorphism $\underline{f} \colon E_Q \to E$ determines $L = \underline{f}(T^{0,1}X) \in \cL^0$. For a different choice of isomorphism $\underline{f}' \colon E_Q \to E$, we have $L' = \underline{f}'\circ \underline{f}^{-1} \cdot L$ which lies in the same $\Aut(E)$-orbit. Finally, if $\psi \colon Q \to Q'$ is an isomorphism of string algebroids, then there exists a unique isomorphism $\tilde{\underline{f}} \colon E_Q \to E_{Q'}$ in a diagram
\begin{equation*}\label{eq:diagunderlineftilde2}
\xymatrix @R=.5pc {
		E_Q \ar@{-->}[dr] \ar[dd]_{\tilde{\underline{f}}} & & \\
		& Q \ar[rd]^{\Id_Q} \ar[dd]_{\psi} & \\ 
		E_{Q'} \ar@{-->}[dr] & & Q\\
		& Q' \ar[ru]_{\psi}  & 
	}
\end{equation*}
which determines $\tilde L = \underline{f} \circ \tilde{\underline{f}}^{-1}(T^{0,1}X) \in \cL^0$. By Lemma \ref{lemma:bricksglue}, $\tilde L \in  \Aut(E) \cdot L$.
\end{proof}

\begin{remark}
The local geometry of the bigger moduli space of string algebroids over $X$ with varying $[E_Q] \in H^1(\underline{\cS})$ has been recently understood in \cite{grt2} via the construction of a Kuranishi slice theorem.
\end{remark}

By the previous lemma, $\cL^0/\Aut(E)$ is the moduli space of string algebroids $Q$ over $X$ with fixed complex string algebroid $E$, while $\cL^0/\Aut_{dR}(E)$ is a \emph{Teichm\"uller space} for string algebroids. We analyse next in detail the infinitesimal structure of the Teichm\"uller space when $E$ is the complexification of $E_\RR$, towards a Donaldson-Uhlenbeck-Yau type theorem for the Calabi system. For this, we fix a solution $W$ of \eqref{eq:Calabil} and consider the associated element $L_W \in \cL^0$ via the Chern correspondence in Lemma \ref{lem:Cherncorrsets}. Relying on Lemma \ref{lemma:liftings} and Section \ref{sec:metric}, the tangent space of $\cL^0/\Aut_{dR}(E)$ at $[L_W]$ is given (formally) by the cohomology of the complex 
\begin{equation}
  \label{eq:TLMcomplexring}
(\widehat{C}^*) \qquad \Omega^0(\ad \underline P) \oplus  \Omega^1_\CC \lra{\widehat{\mathbf{P}}^c} \Omega^{1,1 + 0,2} \oplus \Omega^{0,1}(\ad \underline{P}) \lra{\mathbf{L}^c} \Omega^{1,2 + 0,3} \oplus \Omega^{0,2}(\ad \underline{P})
,
\end{equation}
where 
\begin{align*}
\widehat{\mathbf{P}}^c(r,\xi) & = (d\xi^{0,1} + \dbar \xi^{1,0} + 2\la r,F_h \ra,\dbar r),\\
\mathbf{L}^c(\dot \gamma,\dot \beta) & = (d \dot \gamma^{0,2} + \dbar \dot \gamma^{1,1} - 2 \la \dot \beta,F_h \ra , \dbar \dot \beta).
\end{align*}
We show next that the Teichm\"uller space $\cL^0/\Aut_{dR}(E)$ is finite dimensional. 

\begin{lemma}\label{lem:sesTLM}
The sequence \eqref{eq:TLMcomplexring} is an elliptic complex of differential
operators. Consequently, the cohomology $H^1(\widehat{C}^*)$ of \eqref{eq:TLMcomplexring} is finite dimensional.
\end{lemma}

Ellipticity of the complex $\widehat{C}^*$ can be easily obtained via comparison with the Dolbeault complex of the holomorphic vector bundle underlying $Q_{L_W}$ (cf. \cite{AGS,OssaSvanes} and arXiv version 1503.07562v1 of reference \cite{grt}).

Our strategy to compare $H^1(\widehat{C}^*)$ with the tangent to the moduli space of solutions of the Calabi system, given by $H^1(\widehat{S}^*)$ as in Lemma \ref{lem:sesTWM}, is to work orthogonally to the image of the operator $\widehat{\mathbf{P}}^c$ with respect to the indefinite pairing $g_\ell$ in \eqref{eq:glW}. Notice here that the Chern correspondence in Lemma \ref{lem:Cherncorrsets} induces an isomorphism
\begin{equation*}\label{eq:infChern}
\begin{split}
\Upsilon \colon \Omega^{1,1}_\RR \oplus \Omega^2 \oplus \Omega^1(\ad P_h) & \longrightarrow \Omega^{1,1 + 0,2} \oplus \Omega^{0,1}(\ad \underline P)\\
(\dot \omega,\dot b,\dot a) & \longmapsto (\dot b^{1,1 + 0,2} - i\dot \omega,\dot a^{0,1}),
\end{split}
\end{equation*}
which we use to define the pairing $g_\ell$ on $\Omega^{1,1 + 0,2} \oplus \Omega^{0,1}(\ad \underline P)$.

\begin{theorem}\label{thm:DUYinfinitesimal}
Assume Condition \ref{ConditionA} and $h^{0,1}_A(X) = h^{0}(\ad P) = 0$.
Then, the cohomology of the complexes \eqref{eq:TWMcomplexring} and \eqref{eq:TLMcomplexring} are canonically isomorphic
$
H^1(\widehat{S}^*) \cong H^1(\widehat{C}^*).
$
\end{theorem}

\begin{proof}
Using the conditions $h^{0,1}_A(X) = h^{0}(\ad P) = 0$ one can easily prove that
$$
\Im \; \widehat{\mathbf{P}}\cap \mathbf{J} \Im \; \widehat{\mathbf{P}} = \{0\}.
$$
Then, via the isomorphism $\Upsilon$, we have equalities
\begin{align*}
\Upsilon^{-1}(\Im \; \widehat{\mathbf{P}}^c) & = \Im \; \widehat{\mathbf{P}}\oplus \mathbf{J} \Im \; \widehat{\mathbf{P}}\\
\Upsilon^{-1}(\Ker \mathbf{L}^c) & = \{(\dot \omega,\dot b,\dot a) \; | \; \dbar \dot a^{0,1} = 0,  d^c \dot \omega + 2\la \dot a, F_h \ra - d \dot b = 0\}
\end{align*}
Assuming Condition \ref{ConditionA}, there is a $g_\ell$-orthogonal projection $\Pi$ as in Lemma \ref{lem:gaugefixing} and we consider the map
\begin{equation*}
\begin{split}
\Pi^c \colon \Omega^{1,1}_\RR \oplus \Omega^2 \oplus \Omega^1(\ad P_h) & \longrightarrow (\Im \; \widehat{\mathbf{P}}\oplus \mathbf{J} \Im \; \widehat{\mathbf{P}})^{\perp_{g_\ell}}\\
v & \longmapsto \Pi^c v:= -\mathbf{J}(\Id - \Pi) \mathbf{J} (\Id - \Pi) v.
\end{split}
\end{equation*}
We take $y_j \in \Omega^0(\ad P_h) \oplus \Omega^1$, for $j = 1,2$, and check that it is well defined
\begin{align*}
g_\ell(\Pi^c v,\widehat{\mathbf{P}}(y_1) + \mathbf{J} \widehat{\mathbf{P}}(y_2)) & = g_\ell(v - \Pi v,\mathbf{J} \widehat{\mathbf{P}}(y_2))\\ 
& \phantom{{} =} + \Omega_\ell(\Pi \mathbf{J}(v - \Pi v), \widehat{\mathbf{P}}(y_1)) + g_\ell(\Pi \mathbf{J}(v - \Pi v),\widehat{\mathbf{P}}(y_2))\\
& = g_\ell(v - \Pi v,\mathbf{J} \widehat{\mathbf{P}}(y_2)) + g_\ell(\mathbf{J}(v - \Pi v),\widehat{\mathbf{P}}(y_2)) = 0.
\end{align*}
For the second equality we used that $\Im \; \Pi \subset \Im \; \widehat{\mathbf{P}}$, that $\mu_\ell$ is equivariant, and also $\mu_\ell(W) = 0$. By Proposition \ref{prop:Calabillineargaugefixed}, there is an equality
$$
\cH^1(\widehat S^*):= \ker \mathbf{L} \cap \ker \widehat{\mathbf{P}}^* = (\Im \; \widehat{\mathbf{P}}\oplus \mathbf{J} \Im \; \widehat{\mathbf{P}})^{\perp_{g_\ell}} \cap \Upsilon^{-1}(\Ker \mathbf{L}^c)
$$
and therefore, since $\Pi^c$ preserves $\Upsilon^{-1}(\Ker \mathbf{L}^c)$, it induces a well-defined surjective map
\begin{equation*}
\begin{split}
\Pi^c \colon \Upsilon^{-1}(\Ker \mathbf{L}^c) & \longrightarrow \cH^1(\widehat S^*).
\end{split}
\end{equation*}
We claim that this map induces an isomorphism $H^1(\widehat{C}^*) \cong \cH^1(\widehat S^*)$. To see this, notice that \eqref{eq:glImpnondeg} implies that $\Pi \mathbf{J} \Im \; \widehat{\mathbf{P}} = 0$, as 
$$
g_\ell(\widehat{\mathbf{P}}(y_1), \Pi \mathbf{J} \widehat{\mathbf{P}}(y_2)) = - \Omega_\ell(\widehat{\mathbf{P}}(y_1), \widehat{\mathbf{P}}(y_2)) = 0.
$$
for any $y_1,y_2$. Then, if $v = \widehat{\mathbf{P}}(y_1) + \mathbf{J} \widehat{\mathbf{P}}(y_2)$ it follows that
\begin{align*}
\Pi^c v & =  v - \Pi v + \mathbf{J} \Pi \mathbf{J}(v - \Pi v)\\
& = (\Id - \Pi)\mathbf{J} \widehat{\mathbf{P}}(y_2) + \mathbf{J} \Pi \mathbf{J} (\Id - \Pi)\mathbf{J}\widehat{\mathbf{P}}(y_2) = \mathbf{J} \widehat{\mathbf{P}}(y_2) - \mathbf{J} \Pi \widehat{\mathbf{P}}(y_2) = 0.
\end{align*}
Conversely, if $\Pi^c v = 0$:
$$
v = \Pi v - \mathbf{J}\Pi \mathbf{J}(v - \Pi v) \in \Im \; \widehat{\mathbf{P}}\oplus \mathbf{J} \Im \; \widehat{\mathbf{P}},
$$
and therefore $H^1(\widehat{C}^*) \cong \cH^1(\widehat S^*)$, as claimed. 
\end{proof}

Our Theorem \ref{thm:DUYinfinitesimal} can be regarded as an infinitesimal Donaldson-Uhlenbeck-Yau type theorem, relating the moduli space of solutions of the Calabi system with the Teichm\"uller space $\cL^0/\Aut_{dR}(E)$ for string algebroids. This strongly suggests that---if we shift our perspective and consider the Calabi system as equations
\begin{equation}\label{eq:CalabilQLW}
\begin{split}
F_h\wedge \omega^{n-1} & = 0,\\
d (e^{-\ell f_\omega}\omega^{n-1}) & = 0,
\end{split}
\end{equation}
for a compact form $E_\RR \subset E_Q$ on a fixed string algebroid $Q$ along a fixed Aeppli class $\mathfrak{a} \in \Sigma_A(Q,\RR)$ (see Proposition \ref{prop:ApPicA})---the existence of solutions should be related to a stability condition in the sense of Geometric Invariant Theory. This was essentially the point of view taken in \cite{grst}. The precise relation with stability in our context is still unclear, as the balanced class $\mathfrak{b} \in H^{n-1,n-1}_{BC}(X,\RR)$ of the solution varies in the moduli space $\cM_\ell$. Recall that $\mathfrak{b}$ is required to measure slope stability of the holomorphic bundle in the classical Donaldson-Uhlenbeck-Yau Theorem \cite{Don,UY} (see also \cite{lt}). The conjectural stability condition which characterizes the existence of solutions of \eqref{eq:CalabilQLW} should be for pairs given by string algebroid $Q$ of Bott-Chern type equipped with a `complexified Aeppli class' (see Appendix \ref{sec:moduliapp}). It must be closely related to the properties of the \emph{integral of the moment map} $\mu_\ell$ for compact forms in a fixed Aeppli class, given by the $\ell$-dilaton functional (cf. \cite{grst}). We speculate that there is a relation between this new form of stability and the conjectural inequality \eqref{eq:ineqmoduli}. This may lead to an obstruction to the global existence which goes beyond the slope stability of the bundle and the balanced property of the manifold (cf. \cite{Yau2}).

\subsection{Examples}\label{sec:example}

We present an interesting class of examples of solutions of the Calabi system where Condition \ref{ConditionA} holds, and Theorem \ref{thm:metric}, Theorem \ref{thm:metricfibre} and Theorem \ref{thm:DUYinfinitesimal} apply. These examples are non-K\"ahler solutions of \eqref{eq:Calabil} obtained by deformation of a solution of the Calabi problem for K\"ahler metrics, as in \eqref{eq:Calabieq}, equipped with a polystable vector bundle. Our method is inspired by the one used in \cite{AGF1} to find solutions of the Hull-Strominger system on K\"ahler Calabi-Yau manifolds. 

Let $X$ be a compact K\"ahler manifold equipped with smooth volume form $\mu$ compatible with the orientation and a K\"ahler class $\mathbf{k} \in H^{1,1}(X,\RR)$. Let $V_0$ and $V_1$ be $\mathbf{k}$-stable holomorphic vector bundles over $X$ with vanishing first Chern class and the same second Chern character
$$
ch_2(V_0) = ch_2(V_1) \in H^{2,2}(X,\RR).
$$
Given $\ell,\epsilon \in \RR$, consider the system of equations
\begin{equation}\label{eq:Calabilepsilon}
	\begin{split}
	F_{h_0}\wedge \omega^{n-1} & = 0,\\
	F_{h_1}\wedge \omega^{n-1} &= 0,\\
	d (e^{-\ell f_\omega}\omega^{n-1}) & = 0,\\
	dd^c \om - \epsilon \tr_0 F_{h_0} \wedge F_{h_0} + \epsilon \tr_1 F_{h_1} \wedge F_{h_1} & = 0.
	\end{split}
\end{equation}
for a hermitian form $\omega$ on $X$ and hermitian metrics $h_j$ in $V_j$. Taking $P$ to be the bundle of split frames of $V_0 \oplus V_1$, any solution of \eqref{eq:Calabilepsilon} provides a solution of the Calabi system \eqref{eq:Calabil} for
$$
\la\,,\ra_\epsilon = - \epsilon \tr_0  + \epsilon \tr_1.
$$ 
Combining the Donaldson-Uhlenbeck-Yau Theorem \cite{Don,UY} with Yau's solution of the Calabi Conjecture \cite{Yau0}, there exists a unique solution $(\omega_0,h_{0,0},h_{1,0})$ of \eqref{eq:Calabilepsilon} for $\epsilon = 0$ with $[\omega_0] = \mathbf{k}$. Notice here that such solution must be necessarily K\"ahler (see \cite{grst}), that is, $d \omega_0 = 0$. 

\begin{proposition}\label{prop:deformation}
Assume $\ell > 2 - \tfrac{2}{n}$ and $h^{0,1}(X) = 0$, and let $(X,V_0,V_1)$ be as above. Then, there exists $\epsilon_0 >0$ and a smooth family $(\omega_\epsilon,h_{0,\epsilon},h_{1,\epsilon})$ of solutions of \eqref{eq:Calabilepsilon} parametrized by $[0,\epsilon_0[$ such that Condition \ref{ConditionA} holds for sufficiently small $\epsilon > 0$. Furthermore, $(\omega_\epsilon,h_{0,\epsilon},h_{1,\epsilon})$ converge uniformly in $C^\infty$ norm to $(\omega_0,h_{0,0},h_{1,0})$ as $\epsilon \to 0$.
\end{proposition}

\begin{proof}
Existence of the family of solutions $(\omega_\epsilon,h_{0,\epsilon},h_{1,\epsilon})$ follows as in \cite{AGF1} by application of an implicit function theorem argument (cf. \cite[Lem. 5.17]{grst}). We prove now that any such solution satisfies Condition \ref{ConditionA} for sufficiently small $\epsilon$. Denote by $P_{h_\epsilon}$ the bundle of split unitary frames for $h_\epsilon = h_{0,\epsilon} \times h_{1,\epsilon}$. For $\cL_\epsilon$ as in \eqref{eq:Loperator} and $(u,\xi) \in \Om^0(\ad P_{h_\epsilon}) \times \Im \; d^*$, the condition $\cL_\epsilon(u,\xi) = 0$ is equivalent to
\begin{align*}
d \Big{(}e^{-\ell f_{\omega_\epsilon}}\Big{(}(n-1) ((d \xi)^{1,1} + 2 \la u,F_{h_\epsilon}\ra_\epsilon) \wedge \omega_\epsilon^{n-2} - \frac{\ell}{2} (\Lambda_{\omega_\epsilon} d \xi) \omega_\epsilon^{n-1}\Big{)}\Big{)} & = 0,\\
d^{h_\epsilon} J d^{h_\epsilon} u \wedge \omega_\epsilon^{n-1} - (n-1) F_{h_\epsilon} \wedge (d \xi + 2 \la u,F_{h_\epsilon} \ra_\epsilon) \wedge \omega_\epsilon^{n-2} & = 0.
\end{align*}
Consider the family of elliptic operators
\begin{equation*}
   \label{eq:Loperatorlambda0}
 \begin{array}{cccc}
U_{\epsilon,0} : & \Om^0(\ad P_{h_\epsilon})  & \rightarrow & \Om^{2n}(\ad P_{h_\epsilon})
         \end{array}
 \end{equation*}
defined by
\begin{align*}
U_{\epsilon,0}(u) & = d^{h_\epsilon} J d^{h_\epsilon} u \wedge \omega_\epsilon^{n-1} - (n-1) F_{h_\epsilon} \wedge (2 \la u,F_{h_\epsilon} \ra_\epsilon) \wedge \omega_\epsilon^{n-2}.
\end{align*}
By hypothesis, $\tilde U_{0,0}$ is elliptic with zero kernel, and therefore $U_{\epsilon,0}$ has vanishing kernel for sufficiently small $\epsilon$ by upper semi-continuity of $\dim \Ker U_{\epsilon,0}$. Notice that $U_{\epsilon,0}$ can be regarded as an operator $\Om^0(\ad P_{h_0}) \rightarrow \Om^0(\ad P_{h_0})$ by a gauge transformation depending only on $h_\epsilon$. Let $\epsilon > 0$ such that $\Ker U_{\epsilon,0} = \{0\}$, and assume that $(u_\epsilon,\xi_\epsilon) \in \Ker \cL_\epsilon$. Given $\lambda \in \RR$, consider the family of elliptic operators
\begin{equation*}
   \label{eq:Loperatorlambda}
 \begin{array}{cccc}
U_{\epsilon,\lambda} : & \Om^0(\ad P_{h_\epsilon})  & \rightarrow & \Om^0(\ad P_{h_\epsilon})
         \end{array}
 \end{equation*}
defined by
\begin{align*}
U_{\epsilon,\lambda}(u) & = d^{h_\epsilon} J d^{h_\epsilon} u \wedge \omega_\epsilon^{n-1} - (n-1) F_{h_\epsilon} \wedge (\lambda d \xi_\epsilon + 2 \la u,F_{h_\epsilon} \ra_\epsilon) \wedge \omega_\epsilon^{n-2}.
\end{align*} 
By upper semi-continuity of $\dim \Ker U_{\epsilon,\lambda}$ we have that $\Ker U_{\epsilon,\lambda} = \{0\}$ for sufficiently small $\lambda$. Since $\lambda u_\epsilon \in \Ker U_{\epsilon,\lambda}$, it follows that $u_\epsilon = 0$. Using now Lemma \ref{lem:P*} and setting $v = (0,d\xi_\epsilon,0)$, we have $g_\ell (v,v) = 0$ and therefore
\begin{align*}
\int_X |((d\xi)^{1,1})_0|^2  e^{-\ell f_{\omega_\epsilon}} \frac{\omega_\epsilon^{n}}{n!} + \Bigg{(}\frac{\ell}{2} - \frac{n-1}{n}\Bigg{)} \int_X | \Lambda_{\omega_\epsilon} d \xi|^2 e^{-\ell f_{\omega_\epsilon}} \frac{\omega_\epsilon^{n}}{n!} = 0.
\end{align*}
For $\ell > 2 - \tfrac{2}{n}$ this implies $(d\xi)^{1,1} = 0$, and therefore $\partial \dbar \xi^{0,1} = 0$. Finally, using that $h^{0,1}(X) = 0$ we conclude $\xi^{0,1} = \dbar \phi$ for some complex valued function $\phi$, and hence $d\xi = 0$.
\end{proof}

We finish with concrete examples where the hypotheses of Proposition \ref{prop:deformation} are satisfied. We will take $X$ to be a Calabi-Yau threefold with holomorphic volume form $\Omega$, and $\mu$ as in \eqref{eq:muOmega2}. We choose a K\"ahler class $\mathbf{k}$, and $\mathbf{k}$-stable bundles $V_0$ and $V_1$ such that
$$
c_1(V_j) = 0, \qquad c_2(V_j) = c_2(X)
$$
(see \cite{AGF1,GF2} and references therein for constructions of such bundles). In this setup, $h^{0,1}(X) = h^{0,2}(X) = 0$ and $h^0(\End V_0) = h^0(\End V_1) = 0$. Hence, the hypotheses of Proposition \ref{prop:deformation} hold, and Theorem \ref{thm:metric}, Theorem \ref{thm:metricfibre} and Theorem \ref{thm:DUYinfinitesimal} apply. 

Our choice of bundles $V_0,V_1$ can be interpreted, geometrically, as a deformation of the special K\"ahler metric on the `complexified K\"ahler moduli' for the Calabi-Yau manifold $X$ (see Corollary \ref{cor:Calabimoduli}). More precisely, Proposition \ref{prop:deformation} combined with Theorem \ref{thm:metricfibre} gives a family of pseudo-K\"ahler metrics $g_{\ell,\epsilon}$ (see \eqref{eq:metricfibre}) in a non-empty open subset of $H^{1,1}(X) \cong H^{1,1}_A(X)$, for $(\ell,\epsilon) \in ]\tfrac{4}{3},2[ \times [0,\epsilon_0[$. Here, the fibre of \eqref{eq:modulimapWbundle} over $[P]$ (for $P$ the bundle of split frames of $V_0 \oplus V_1$) is regarded as a subset of $H^{1,1}(X)$ via \eqref{eq:modulidiagram}, Lemma \ref{lem:fibreModQP}, and the $\partial\dbar$-Lemma. The special K\"ahler metric in the `complexified K\"ahler moduli' of $X$ is recovered (up to homothety) in the $\epsilon \to 0$ limit of this family (see the proof of Corollary \ref{cor:Calabimoduli}). The case of the Hull-Strominger equations corresponds to $\ell = 1$, and it is not covered by our result.

\begin{example}\label{example:CY}
Let $X$ be a complete intersection Calabi-Yau threefold. By \cite[Cor. 2.2]{Huyb}, $TX$ has unobstructed deformations parametrized by $ H^1(\End TX)$. Since $TX$ is stable for any K\"ahler class, any pair of small deformations $V_0$ and $V_1$ of $TX$ are also stable. For the quintic hypersurface $h^1(\End TX) = 224$ and we obtain a family of deformations of the special K\"ahler metric on $H^{1,1}(X)$ of dimension $450$, parametrized by a non-empty open subset of
$$
H^1(\End TX) \times H^1(\End TX) \times ]\tfrac{4}{3},2[ \times [0,\epsilon_0[.
$$
\end{example}

\appendix

\section{The space of compact forms}\label{sec:AeppliPic}

\subsection{Gauge action on compact forms}

Let $Q$ be a Bott-Chern algebroid over a complex manifold $X$, with underlying principal $G$-bundle $P$. We fix a maximal compact subgroup $K\subset G$. Consider the automorphism group $\Aut(E_Q)$ of $Q$, as defined in Section \ref{sec:picard}. This section is devoted to the study of the interplay between $\Aut(E_Q)$ and the space of compact forms with structure group $K$ on the complex string algebroid $E_Q$ (see Lemma \ref{lem:red-for-all-Q}). 
We introduce the following notation for the space of compact forms on $E_Q$ with structure group $K$
$$
\mathbf{B}_Q = \{E_\RR \subset E_Q \; | \; E_\RR \; \textrm{is a compact form}\}.
$$
There is a natural left $\Aut(E_Q)$-action
\begin{equation}\label{eq:Pic-left-actionER}
\begin{split}
\Aut(E_Q) \times \mathbf{B}_Q & \longrightarrow \mathbf{B}_Q\\
(\underline{f},E_\RR) & \longmapsto  \underline{f} \cdot E_\RR : = \underline{f}(E_\RR)
\end{split}
\end{equation}
which extends the classical action of the complex gauge group $\cG_{\underline{P}}$ on the space of reductions $\Omega^0(P/K)$. More precisely, there is a commutative diagram
\begin{equation*}
\xymatrix @R=.5pc {
		\Aut(E_Q) \times \mathbf{B}_Q \ar[r] \ar[dd] & \mathbf{B}_Q \ar[dd] \\
		& \\
		\Ker \sigma_{\underline{P}} \times \Omega^0(P/K) \ar[r] & \Omega^0(P/K),
	}
\end{equation*}
where $\Ker \sigma_{\underline{P}} \subset \cG_{\underline{P}}$ is the subgroup defined by \eqref{eq:Picardses} and the bottom arrow is induced by the left $\cG_{\underline{P}}$-action on $\Omega^0(P/K)$. In order to obtain a better understanding of this action, we start by giving a more explicit description of the space $\mathbf{B}_Q$ for the case of a string algebroid given by a triple $(P,H,\theta)$ (see Proposition \ref{lemma:deRhamC} and Definition \ref{def:Q0}). For this, we apply Proposition \ref{propo:Chernclassic} to the canonical diagram $(E_Q,T^{0,1}X,\Id_Q)$ for $Q$.

\begin{lemma}\label{lemma:BQexp}
	Let $Q_0$ be the string algebroid given by a triple $(P,H,\theta)$. Then, $\mathbf{B}_{Q_0}$ can be regarded canonically as the subset
	$$
	\mathbf{B}_{Q_0} \subset \Omega^{1,1}_\RR \oplus \Omega^{2,0} \times \Omega^0(P/K)
	$$
	given by
	\begin{equation}\label{eq:BQexp}
	\mathbf{B}_{Q_0} =    \Big\{(\omega + \upsilon,h) \st d \upsilon = H + 2i\partial \omega + CS(\theta) - CS(\theta^h) - d\la \theta \wedge \theta^h \ra \Big\}.
	\end{equation}
\end{lemma}
\begin{proof}
Let $E_\RR \in \mathbf{B}_{Q_0}$ and consider the triple $(\omega,h,\varphi)$ corresponding to the canonical diagram $(E_{Q_0},T^{0,1}X,\Id_{Q_0})$ for $Q_0$ via Proposition \ref{propo:Chernclassic}. Then, $\varphi$ is given explicitly by
	\begin{equation}\label{eq:psiexp}
	\varphi = (\upsilon,\theta - \theta^h)
	\end{equation}
acting as in \eqref{eq:BA}, where $\upsilon \in \Omega^{2,0}$ satisfies the condition in \eqref{eq:BQexp} (see Proposition \ref{lemma:deRhamC}), and therefore $E_\RR$ can be identified with a triple $(\omega + \upsilon,h)$ as in the statement. 

Observe that the compact form and horizontal subspace corresponding to a triple $(\omega + \upsilon,h)$ are given by
\begin{equation}\label{eq:ERombh}
E_\RR = (\upsilon -i\omega,\theta - \theta^h)(E_{\RR,h}) \subset E_{Q_0},
\end{equation}
where 
\begin{equation}
 \label{eq:ER}
E_{\RR,h}: = T \underline X \oplus\ad P_h \oplus T^*\underline{X},
\end{equation}
 and
\begin{equation}\label{eq:Wombh}
W \otimes \CC = (\upsilon-2i\omega,\theta - \theta^h)(T^{1,0}X) \oplus T^{0,1}X.
\end{equation}
\end{proof}

Our next result provides an explicit formula for the $\Aut(E_Q)$-action on $\mathbf{B}_Q$ in terms of the model in Lemma \ref{lemma:BQexp}. In the sequel, we will use the notation $\boldsymbol{\omega} = (\omega + \upsilon,h)$ for the elements in $\mathbf{B}_{Q_0}$ and identify $E_0 = E_{Q_0}$.

\begin{lemma}\label{lemma:PicQactexp}
	Let $Q_0$ be the string algebroid given by a triple $(P,H,\theta)$. Let $\underline{f} = (g,\tau) \in \Aut(E_0)$ (see Lemma \ref{lem:HomQQ}) and $\boldsymbol{\omega} = (\omega + \upsilon,h) \in \mathbf{B}_{Q_0}$. Then,  
\begin{equation*}\label{eq:PicQactexp}
\underline{f} \cdot \boldsymbol{\omega} = (\omega' + \upsilon',gh),
\end{equation*}
where (for $a^g = g^{-1}\theta - \theta$)
\begin{equation}\label{eq:PicQactexp2}
\begin{split}
\omega' & = \omega - \operatorname{Im} (\tau + \langle a^g  \wedge \theta - \theta^h \rangle + \langle g \theta^h - \theta \wedge \theta^{gh} - g\theta^h \rangle )^{1,1},\\
\upsilon' & = \upsilon + (\tau + \la a^g \wedge \theta - \theta^h \ra + \la g \theta^h - \theta \wedge \theta^{gh} - g\theta^h \ra)^{2,0} \\
&  \phantom{ {} = } - \overline{(\tau + \la g \theta^h - \theta \wedge \theta^{gh} - g\theta^h \ra)^{0,2}}.
\end{split}
\end{equation}
\end{lemma}
\begin{proof}
Let $\boldsymbol{\omega} = (\omega + \upsilon,h) \in \mathbf{B}_{Q_0}$ with real form \eqref{eq:ERombh}. Then, for $\underline{f} = (g,\tau)$ we have
\begin{align*}
\underline{f}(E_\RR) & = (\upsilon-i\omega + \tau + \la a^g \wedge \theta - \theta^h \ra,g(a^g + \theta - \theta^h))(E_{\RR,gh})\\
& = (\upsilon-i\omega + \tau + \la a^g \wedge \theta - \theta^h \ra, \theta - g\theta^h)(E_{\RR,gh}),
\end{align*}
where $E_{\RR,gh}$ is as in \eqref{eq:ER}. Using that $(0,g\theta^h - \theta^{gh})(E_{\RR,gh}) = E_{\RR,gh}$, we obtain
\begin{align*}
\underline{f}(E_\RR) & = (\upsilon-i\omega + \tau + \la a^g \wedge \theta - \theta^h \ra, \theta - g\theta^h)(0,g\theta^h - \theta^{gh})(E_{\RR,gh})\\
& = (\upsilon-i\omega + \tau + \la a^g \wedge \theta - \theta^h \ra + \la \theta - g\theta^h \wedge g\theta^{h} - \theta^{gh}\ra, \theta - \theta^{gh})(E_{\RR,gh}).
\end{align*}
Let $W' \subset \underline{f}(E_\RR)$ be the horizontal subspace determined by the canonical diagram $(E_{Q_0},T^{0,1}X,\Id_{Q_0})$ via the Chern correspondence. Following the proof of Lemma \ref{lem:Cherncorr}, we set
$$
(b_0,a_0) = (\upsilon-i\omega + \tau + \la a^g \wedge \theta - \theta^h \ra + \la \theta - g\theta^h \wedge g\theta^{h} - \theta^{gh}\ra, \theta - \theta^{gh}).
$$
There exists $(b,a) \in \Omega^{2}\oplus \Omega^{1}(\ad P_{\RR,gh})$ and a real symmetric tensor $\sigma'$ with associated differential form
$$
\omega' = \sigma'(J,) \in \Omega^{1,1}_\RR,
$$
uniquely determined by the condition
$$
W_0' := (-b_0,-a_0)(W') = (-b,-a)\{V + \sigma'(V): V \in T\underline{X}\} \subset E_{\RR,gh} \otimes \CC.
$$
Next, we define $(\gamma,\beta) \in \Omega^{1,1 + 0,2} \oplus \Omega^{0,1}(\ad \underline P)$ by
$$
(-b_0,-a_0)(T^{0,1}X) = (-\gamma,-\beta)(T^{0,1}X).
$$
More explicitly,
\begin{equation*}
\gamma = b_0^{1,1 + 0,2}, \qquad \beta = 0,
\end{equation*}
and therefore \eqref{eq:WChern} implies $a = 0$, and
\begin{equation}\label{eq:WChernexpQ0}
\begin{split}
\omega' & = - \Im \; b_0^{1,1} = \omega - \Im (\tau + \la a^g \wedge \theta - \theta^h \ra + \la \theta - g\theta^h \wedge g\theta^{h} - \theta^{gh}\ra)^{1,1},\\
b & = \Re \; b_0^{1,1} + b_0^{0,2} + \overline{b_0^{0,2}}\\ 
& = \Re \;(\tau + \la a^g \wedge \theta - \theta^h \ra + \la \theta - g\theta^h \wedge g\theta^{h} - \theta^{gh}\ra)^{1,1} \\
& \phantom{ {} = } + (\tau + \la \theta - g\theta^h \wedge g\theta^{h} - \theta^{gh}\ra)^{0,2} + \overline{(\tau + \la \theta - g\theta^h \wedge g\theta^{h} - \theta^{gh}\ra)^{0,2}}.
\end{split}
\end{equation}
The first equation in \eqref{eq:WChernexpQ0} gives the formula for $\omega'$ in \eqref{eq:PicQactexp2}. To obtain the formula for $\upsilon'$, notice that \eqref{eq:Wombh} implies that
$$
W' \otimes \CC = (\upsilon'-2i\omega',\theta - \theta^{gh})(T^{1,0}X) \oplus T^{0,1}X
$$
and, on the other hand,
\begin{align*}
W' \otimes \CC & = (b_0,a_0)(W_0')\\
& = ((b_0 - b)^{1,1 + 2,0} - i\omega',\theta - \theta^{gh})(T^{1,0}X) \oplus T^{0,1}X.
\end{align*}
Therefore, we conclude
\begin{align*}
\upsilon' & = (b_0 - b)^{2,0}\\
& = \upsilon + (\tau + \la a^g \wedge \theta - \theta^h \ra + \la \theta - g\theta^h \wedge g\theta^{h} - \theta^{gh}\ra)^{2,0}\\
& \phantom{ {} = } - \overline{(\tau + \la \theta - g\theta^h \wedge g\theta^{h} - \theta^{gh}\ra)^{0,2}},
\end{align*}
as claimed.
\end{proof}

\subsection{Contractibility and transitivity}


Using Lemma \ref{lemma:PicQactexp}, we want to calculate a formula for the infinitesimal $\Aut(E_Q)$-action on $\mathbf{B}_Q$ in terms of the model $Q_0$. For this, we characterize next the tangent space to $\mathbf{B}_{Q_0}$. 

\begin{lemma}\label{lemma:TBQ}
	The tangent space of $\mathbf{B}_{Q_0}$ at $\boldsymbol{\omega} = (\omega + \upsilon,h) \in \mathbf{B}_{Q_0}$ is given by the subspace
	$$
	T_{\boldsymbol{\omega}}\mathbf{B}_{Q_0} \subset \Omega^{1,1}_\RR \oplus \Omega^{2,0} \oplus \Omega^0(i \ad P_h),
	$$
	defined by
	\begin{equation}\label{eq:TBQ}
	T_{\boldsymbol{\omega}}\mathbf{B}_{Q_0} = \{(\dot \omega + \dot \upsilon,iu) \; | \; d (\dot \upsilon + 2i\la \theta^h - \theta \wedge \partial^h u \ra ) = 2i\partial (\dot \omega + 2\la u,F_h \ra) \}.
	\end{equation}
\end{lemma}

\begin{proof}
	Showing that the right-hand side of \eqref{eq:TBQ} is contained in $T_{\boldsymbol{\omega}}\mathbf{B}_{Q_0}$ is a formality, by taking derivatives along a curve $(\omega_t+\upsilon_t,h_t)$ in $\mathbf{B}_{Q_0}$. To see this, we define
	$$
	C_t = CS(\theta) - CS(\theta^{h_t}) - d\la \theta \wedge \theta^{h_t} \ra
	$$  
	and use Remark \ref{rem:CSexp} combined with Lemma \ref{lem:CSRinvariant} to calculate
	\begin{align*}
	\frac{d}{dt}_{|t = 0} C_t & = \frac{d}{dt}_{|t = 0}\Big{(}CS(\theta^h) - CS(\theta^{h_t}) - d\la \theta \wedge \theta^{h_t}\ra - d\la \theta^h \wedge \theta^{h_t} \ra + d\la \theta^h \wedge \theta^{h_t} \ra\Big{)}\\
	& = \frac{d}{dt}_{|t = 0}\Big{(} 2 \la \theta^h - \theta^{h_t} \wedge F_{h} \ra + d\la \theta^h - \theta \wedge \theta^{h_t} \ra \Big{)}\\
	& = 4i\partial \la u, F_h \ra - 2i d \la \theta^h - \theta \wedge \partial^h u \ra.
	\end{align*}
	Here, we have used the formula for the infinitesimal variation of the Chern connection with respect to $iu \in \Omega^0(i \ad P_h)$ (see Lemma \ref{lem:CSRinvariant}):
	$$
	\frac{d}{dt}_{|t = 0} \theta^{e^{itu}h} = -2i\partial^h u \in \Omega^{1,0}(i\ad P_h).
	$$
	Conversely, given $(\dot \omega + \dot \upsilon,iu)$ satisfying 
	\begin{equation*}\label{eq:ddotb}
	d (\dot \upsilon + 2i\la \theta^h - \theta \wedge \partial^h u \ra ) = 2i\partial (\dot \omega + 2\la u,F_h \ra),
	\end{equation*}
	we define, for $t \in \RR$, $\boldsymbol{\omega}_t = (\omega_t + \upsilon_t,h_t)$ by
	\begin{equation}\label{eq:pathofmetrics}
	\begin{split}
	h_t & = e^{itu}h,\\
	\omega_t & = \omega + t(\dot \omega + 2 \la u,F_{h}\ra ) - \tilde R(h_t,h),\\
	\upsilon_t & = \upsilon + t (\dot \upsilon + 2i\la \theta^h - \theta \wedge \partial^h u \ra) - \int_0^t \la \theta^{h_s} - \theta^h \wedge 2i \partial^{h_s} u \ra ds\\
	& \phantom{ {} = } - \la \theta \wedge \theta^{h_t}\ra + \la \theta \wedge \theta^{h}\ra + \la \theta^h \wedge \theta^{h_t} \ra,
	\end{split}
	\end{equation}
	where $\tilde R(h_t,h)$ is defined as in Lemma \ref{lem:CSRinvariant}. We claim that $\boldsymbol{\omega}_t$ satisfies the equation on the right-hand side of \eqref{eq:BQexp} and therefore $\boldsymbol{\omega}_t \in \mathbf{B}_{Q_0}$ for all $t$. To see this, using Lemma \ref{lem:CSRinvariant}, we calculate
	\begin{align*}
	d \upsilon_t & = H + 2i\partial \omega_t + CS(\theta) - CS(\theta^{h}) - d\la \theta \wedge \theta^{h_t}\ra + d \la \theta^h \wedge \theta^{h_t} \ra \\
	& \phantom{ {} = }+ 2i\partial \tilde R(h_t,h)  - d \left(\int_0^t \la \theta^{h_s} - \theta^h \wedge 2i \partial^{h_s} u \ra ds \right)\\
	& = H + 2i \partial \omega_t + CS(\theta) - CS(\theta^{h_t}) - d \la \theta \wedge \theta^{h_t} \ra.
	\end{align*}
	Finally, using Proposition \ref{prop:Donaldson} and Lemma \ref{lem:CSRinvariant} again,
\begin{equation}\label{eq:partialtomegat}
\begin{split}
	\dot \omega_t & : = \frac{d}{dt} \omega_t = \dot \omega + 2 \la u,F_{h}\ra - 2 \la u,F_{h_t}\ra,\\
	\dot \upsilon_t & : = \frac{d}{dt} \upsilon_t = \dot \upsilon + \la \theta^h - \theta \wedge 2i\partial^{h} u \ra - \la \theta^{h_t} - \theta \wedge 2i\partial^{h_t} u \ra,
\end{split}
\end{equation}
and thus the tangent vector of $\boldsymbol{\omega}_t := (\omega_t + \upsilon_t,h_t)$ at $t = 0$ is $(\dot \omega + \dot \upsilon,iu)$.
\end{proof}

We are ready to calculate the Lie algebra action induced by \eqref{eq:Pic-left-actionER}
\begin{equation}\label{eq:PicQaction0}
\rho \colon \Lie \Aut(E_Q) \to \Gamma (T\mathbf{B}_{Q}).
\end{equation}
Recall that, for any choice of reduction $h \in \Omega^0(P/K)$, the Cartan involution induces a well-defined involution \eqref{eq:Cartaninv}.

\begin{lemma}\label{lem:LiePicactionBQ}
The Lie algebra action \eqref{eq:PicQaction0} is surjective. Furthermore, for any choice of representative $[(P,H,\theta)] = [Q] \in H^1(\cS)$ and isomorphism $Q \cong Q_0$, the induced action $\rho_0 \colon \Lie \Aut(E_0) \to \Gamma (T\mathbf{B}_{Q_0})$ is given by
	\begin{equation}\label{eq:LiePicQ0}
	\rho_0(\zeta)_{|\boldsymbol{\omega}} = \Big{(} \dot \omega + \dot \upsilon,\tfrac{1}{2}(s - s^{*_h})\Big{)},
	\end{equation}
	for $\zeta = (s,B) \in \Lie \Aut(E_0)$ (see Lemma \ref{lem:LiePic}) and $\boldsymbol{\omega}= (\omega + \upsilon,h) \in \mathbf{B}_{Q_0}$, where
	\begin{equation}\label{eq:LiePicQ0exp}
	\begin{split}
	\dot \omega & = - \Im \; (B^{1,1} + 2 \la \theta^h - \theta \wedge \dbar s \ra), \\
	\dot \upsilon 
	& = B^{2,0} - \overline{B^{0,2}} + \la \theta^h - \theta \wedge \partial^hs^{*_h} + \partial^\theta s \ra.
	\end{split}
	\end{equation}		
\end{lemma}

\begin{proof}
We start by proving \eqref{eq:LiePicQ0}. Let $\underline{f}_t = (g_t,\tau_t) \in \Aut(E_0)$ be a one-parameter subgroup and let
$$
(s,B):= \frac{d}{dt}_{|t = 0}\underline{f}_t.
$$
Then, taking derivatives in \eqref{eq:PicQactexp2} at $t = 0$ we obtain
$$
\frac{d}{dt}_{|t = 0}\underline{f}_t \cdot \boldsymbol{\omega} = \Big{(} \dot \omega + \dot \upsilon,\tfrac{1}{2}(s - s^{*_h})\Big{)}
$$
where
\begin{equation*}
\begin{split}
\dot \omega & = - \operatorname{Im} (B + \langle d^\theta s \wedge \theta - \theta^h \rangle + \langle \theta^h - \theta \wedge d^hs \rangle )^{1,1}\\
& = - \operatorname{Im} (B + 2\langle \theta^h - \theta \wedge \dbar s \rangle)^{1,1}, \\
\dot \upsilon & = (B + \la d^\theta s \wedge \theta - \theta^h \ra + \la \theta^h - \theta \wedge -\partial^h(s - s^{*_h}) + d^h s \ra)^{2,0} \\
&  \phantom{ {} = } - \overline{B^{0,2}} - \overline{\la \theta^h - \theta \wedge -\partial^h(s - s^{*_h}) + d^h s \ra^{0,2}}\\
& = B^{2,0} - \overline{B^{0,2}} + \la \theta^h - \theta \wedge \partial^\theta s +\partial^h s^{*_h} \ra.
\end{split}
\end{equation*}
Finally, we prove the surjectivity of \eqref{eq:LiePicQ0}. Given $(\dot \omega + \dot \upsilon,iu) \in T_{\boldsymbol{\omega}}\mathbf{B}_{Q_0}$, taking imaginary parts in \eqref{eq:TBQ},
	$$
	d (- \Im \; (\dot \upsilon + 2i \la \theta^h - \theta \wedge  \partial^h u \ra) + \dot \omega + 2 \la u,F_h \ra) = 0
	$$
	and therefore 
	$$
	\zeta = (i u, - i \dot \omega  + i\Im \; (\dot \upsilon + 2i \la \theta^h - \theta \wedge  \partial^h u \ra) + i \la \theta - \theta^h \wedge d^\theta u + d^h u \ra)
	$$ 
	is an element in $\Lie \Aut(E_0)$ which satisfies $\rho_0(\zeta) \cdot \boldsymbol{\omega} = (\dot \omega + \dot \upsilon,iu)$.
\end{proof}

To finish this section we prove the main structural property of the space of compact forms $\mathbf{B}_Q$. In particular, we obtain the analogous statement to the uniqueness, up to isomorphism, of the compact form of a holomorphic principal bundle.

\begin{proposition}\label{lem:retraction}
The space $\mathbf{B}_Q$ is contractible. Consequently, the $\Aut(E_Q)$-action \eqref{eq:Pic-left-actionER} on $\mathbf{B}_Q$ is transitive.
\end{proposition}
\begin{proof}
We work with a model $Q_0 \cong Q$ as in \eqref{eq:BQexp}, and fix $\boldsymbol{\omega} = (\omega + \upsilon,h) \in \mathbf{B}_{Q_0}$. Using \eqref{eq:BQexp} and Lemma \ref{lem:CSRinvariant}, given $\boldsymbol{\omega}' = (\omega' + \upsilon',h') \in \mathbf{B}_{Q_0}$ we have
\begin{align}\label{eq:tautauprime}
\begin{split}
d(\upsilon' - \upsilon) - d\la \theta \wedge \theta^h & \ra + d\la \theta \wedge \theta^{h'} \ra 
 = 2i\partial (\omega' - \omega + \tilde R(h',h))\\
& \phantom{ {} = } + d\la \theta^h \wedge \theta^{h'} \ra + d\left(\int_0^1 \la \theta^{h} - \theta^{h_t} \wedge 2i \partial^{h_t} u \ra dt\right),
\end{split}
\end{align}
where $h_t = e^{itu}$ for $u \in \Omega^0(\ad P_h)$ such that $h' = e^{iu}h$. Therefore, setting
\begin{equation*}\label{eq:taudotbdot}
\begin{split}
\dot \omega & = \omega' - \omega + \tilde R(h',h) - 2\la u,F_h \ra, \\
\dot \upsilon & = \upsilon' - \upsilon - \la \theta \wedge \theta^h \ra + \la \theta \wedge \theta^{h'} \ra - \la \theta^h \wedge \theta^{h'} \ra\\
& \phantom{ {} = } - \int_0^1 \la \theta^{h} - \theta^{h_t} \wedge 2i \partial^{h_t} u \ra dt  - 2i\la \theta^h - \theta \wedge \partial^h u \ra,
\end{split}
\end{equation*}
it follows that $(\dot \omega + \dot \upsilon,iu) \in T_{\boldsymbol{\omega}}\mathbf{B}_{Q_0}$ (described as in \eqref{eq:TBQ}). Consider the curve $\boldsymbol{\omega}_t \in \mathbf{B}_{Q_0}$ defined as in \eqref{eq:pathofmetrics}. Explicitly, this is given by
\begin{align*}
h_t & = e^{itu}h,\\
\omega_t & = \omega + t(\omega' - \omega + \tilde R(h',h)) - \tilde R(h_t,h),\\
\upsilon_t & = \upsilon + t \left(\upsilon' - \upsilon - \la \theta \wedge \theta^h \ra + \la \theta \wedge \theta^{h'} \ra - \la \theta^h \wedge \theta^{h'} \ra - \int_0^1 \la \theta^{h} - \theta^{h_s} \wedge 2i \partial^{h_s} u \ra ds\right)\\
& \phantom{ {} = }- \int_0^t \la \theta^{h_s} - \theta^h \wedge 2i \partial^{h_s} u \ra ds - \la \theta \wedge \theta^{h_t}\ra + \la \theta \wedge \theta^{h}\ra + \la \theta^h \wedge \theta^{h_t} \ra. 
\end{align*}
For $t = 1$ we have $\boldsymbol{\omega}_{1} = \boldsymbol{\omega}'$, 
and therefore a continuous deformation retraction of $\mathbf{B}_{Q_0}$ (in $C^\infty$-topology, say) is defined by
\begin{align*}
F \colon [0,1] \times \mathbf{B}_{Q_0} & \to \mathbf{B}_{Q_0} \\ (s,\boldsymbol{\omega}') &\mapsto \boldsymbol{\omega}_{1-s}.
\end{align*}
One can then check that this retraction is independent of our choice of model $Q_0 \cong Q$.

The last part of the statement follows from the surjectivity of the infinitesimal action (Lemma \ref{lem:LiePicactionBQ}), and the contractibility of the space $\mathbf{B}_Q$.
\end{proof}

\section{Aeppli classes}\label{sec:moduliapp}

\subsection{Aeppli classes and Hamiltonian orbits}\label{sec:Aporbit}

The goal of this section is to find an explanation for the variations of `complexified Aeppli classes' appearing in formula \eqref{eq:metricfibre} for the fibre-wise moduli metric. We start by extending the notion of Aeppli class on a Bott-Chern algebroid introduced in \cite{grst} to our setup. The proof of the following result follows from \eqref{eq:BQexp} and the properties of the Bott-Chern secondary characteristic class in Proposition \ref{prop:Donaldson} and Lemma \ref{lem:CSRinvariant}. 
Observe that the Aeppli cohomology group $H^{1,1}_A(X)$ \eqref{eq:A-cohomology} admits a canonical real structure $H^{1,1}_A(X,\RR)$.

\begin{lemma}\label{lem:Apmap}
There is a well-defined map
$$
Ap \colon \mathbf{B}_Q \times \mathbf{B}_Q \to H^{1,1}_A(X,\RR)
$$
that satisfies the cocycle condition
\begin{equation}\label{eq:Apcocycle}
Ap(\boldsymbol{\omega}_2,\boldsymbol{\omega}_0) = Ap(\boldsymbol{\omega}_2,\boldsymbol{\omega}_1) + Ap(\boldsymbol{\omega}_1,\boldsymbol{\omega}_0)
\end{equation}
for any triple of elements in $\mathbf{B}_Q$. 
More explicitly, given $Q_0 \cong Q$ determined by a triple $(P,H,\theta)$ (see Definition \ref{def:Q0}),
\begin{equation}\label{eq:Apexp}
Ap_0(\boldsymbol{\omega}',\boldsymbol{\omega}) = [\omega' - \omega + \tilde R(h',h)] \in H^{1,1}_A(X,\RR),
\end{equation}
where $\tilde R(h',h)$ is as in Lemma \ref{lem:CSRinvariant}. 
\end{lemma}


As a straightforward consequence of the cocycle condition \eqref{eq:Apcocycle}, we obtain that the map $Ap$ induces an equivalence relation in $\mathbf{B}_Q$ defined by
$$
\boldsymbol{\omega} \sim_A \boldsymbol{\omega}' \quad \textrm{\emph{if and only if}} \quad Ap(\boldsymbol{\omega},\boldsymbol{\omega}') = 0.
$$

\begin{definition}\label{def:ApQ}
The set of \emph{Aeppli classes} of $Q$ is the quotient $$\Sigma_A(Q,\RR) := \mathbf{B}_Q/\sim_A.$$
\end{definition}

The set $\Sigma_A(Q,\RR) $ has a natural structure of affine space. To see this, denote by $\Omega^{2,0}_{cl}$ the sheaf of (holomorphic) closed $(2,0)$-forms on $X$. Recall from \cite[Lem. 2.11]{grt2} that there is a group homomorphism
\begin{equation}\label{eq:sigmaP}
\sigma_P \colon \cG_P \to H^1(\Omega^{2,0}_{cl}),
\end{equation}
where $H^1(\Omega^{2,0}_{cl})$ denotes the first \v Cech cohomology of the sheaf of closed $(2,0)$-forms on $X$, defined by
$$
\sigma_P(g) = [CS(g\theta^h) - CS(\theta^h) - d\la g\theta^h \wedge \theta^h \ra ] \in  H^1(\Omega^{2,0}_{cl}),
$$
for any choice of reduction $h \in \Omega^0(P/K)$. Here we use \cite{G2} (see also \cite[Lem. 3.3]{grt2}) to identify
\begin{equation}\label{eq:kcohomologyleqk}
H^1(\Omega^{2,0}_{cl}) \cong \frac{\Ker \; d \colon \Omega^{3,0} \oplus \Omega^{2,1} \to \Omega^{4,0} \oplus \Omega^{3,1} \oplus \Omega^{2,2}}{ \Im \; d \colon \Omega^{2,0} \to\Omega^{3,0} \oplus \Omega^{2,1}}.
\end{equation}
Using \eqref{eq:A-cohomology} and the isomorphism \eqref{eq:kcohomologyleqk}, we define a map 
\begin{equation}\label{eqannex:partialmap}
\partial \colon H^{1,1}_A(X) \to H^1(\Omega^{2,0}_{cl})/\Im \; \sigma_P,
\end{equation}
induced by the $\partial$ operator on forms acting on representatives. Then, $\Sigma_A(Q,\RR) $ has a natural structure of affine space modelled on the intersection of $H^{1,1}_A(X,\RR)$ with the kernel of \eqref{eqannex:partialmap} (see \cite[Sec. 3.3]{grst}).

We give next an alternative construction of the set of Aeppli classes in terms of a normal subgroup
$$
\Aut_A(E_Q) \subset \Aut(E_Q)
$$
associated to the Aepply cohomology $H^{1,1}_A(X)$. To define $\Aut_A(E_Q)$, one considers a Lie algebra homomorphism 
\begin{equation*}\label{eq:aepplimap3}
\mathbf{a} \colon \Lie \Aut(E_Q) \to H^{1,1}_A(X),
\end{equation*}
given, when choosing an isomorphism $Q\cong Q_0$, by
\begin{equation*}\label{eq:d0sB}
\mathbf{a}_0(s,B) = [B^{1,1} - 2\la s,F_\theta^{1,1} \ra] \in  H^{1,1}_A(X).
\end{equation*}
The properties of $\mathbf{a}$ are analogue to those of $\mathbf{d}$, and follow as in Lemma \ref{lem:PicAeppli} using equation \eqref{eq:Picinvariance}. The definition of $\Aut_A(E_Q)$ is as in Definition \ref{def:PicA} and its Lie algebra is $\Ker \mathbf{a}$. By \eqref{eq:Pic-left-actionER}, we obtain an induced left action
\begin{equation}\label{eq:PicA-left-actionER}
\Aut_A(E_Q) \times \mathbf{B}_Q \longrightarrow \mathbf{B}_Q.
\end{equation} 


\begin{proposition}\label{prop:ApPicA}
A pair of $\boldsymbol{\omega},\boldsymbol{\omega}' \in \mathbf{B}_Q$ are in the same $\Aut_A(E_Q)$-orbit if and only if they define the same Aeppli class
$$
[\boldsymbol{\omega}] = [\boldsymbol{\omega}'] \in \Sigma_A(Q,\RR).
$$
\end{proposition}

\begin{proof}
We fix a model $Q_0 \cong Q$ determined by a triple $(P,H,\theta)$ (see Definition \ref{def:Q0}). Let $\boldsymbol{\omega},\boldsymbol{\omega}' \in \mathbf{B}_{Q_0}$ and consider the curve $\boldsymbol{\omega}_t \in \mathbf{B}_{Q_0}$ joining $\boldsymbol{\omega}$ and $\boldsymbol{\omega}'$, constructed in Proposition \ref{lem:retraction}. Then, the Aeppli map along the curve is
\begin{align*}
Ap(\boldsymbol{\omega}_t,\boldsymbol{\omega}) & = [t(\omega' - \omega + \tilde R(h',h)) - \tilde R(h_t,h) + \tilde R(h_t,h)] = t Ap(\boldsymbol{\omega}',\boldsymbol{\omega}).
\end{align*}
Assume first that $\boldsymbol{\omega} \sim_A \boldsymbol{\omega}'$, which implies  $\boldsymbol{\omega} \sim_A \boldsymbol{\omega}_t$ for all $t$ by the previous equation. Taking derivatives along the curve we obtain (see \eqref{eq:partialtomegat})
\begin{align*}
\dot \omega_t & : =\frac{d}{dt}\omega_t 
= \omega' - \omega + \tilde R(h',h) - 2 \la u,F_{h_t}\ra,\\
\dot \upsilon_t & : = \frac{d}{dt} \upsilon_t 
= \upsilon' - \upsilon - \la \theta \wedge \theta^h \ra + \la \theta \wedge \theta^{h'} \ra - \la \theta^h \wedge \theta^{h'} \ra\\
& \phantom{ {} := \frac{d}{dt} \upsilon_t  =}  - \int_0^1 \la \theta^{h} - \theta^{h_s} \wedge 2i \partial^{h_s} u \ra ds - \la \theta^{h_t} - \theta \wedge 2i\partial^{h_t} u \ra,
\end{align*}
which corresponds to the infinitesimal action of 
$$
\zeta_t = (i u, - i \dot \omega_t  + i\Im \; (\dot \upsilon_t + 2i \la \theta^{h_t} - \theta \wedge  \partial^{h_t} u \ra) + i \la \theta - \theta^{h_t} \wedge d^\theta u + d^{h_t} u \ra).
$$
Evaluating in the Lie algebra homomorphism in Lemma \ref{lem:PicAeppli}
\begin{align*}
\mathbf{a}_0(\zeta_t) & = [- i \dot \omega_t + 2i \la \theta - \theta^{h_t} \wedge \dbar u \ra) - 2i \la u , F_\theta \ra]\\
& = -i Ap(\boldsymbol{\omega}',\boldsymbol{\omega}) + 2i[\la u , F_{h_t} \ra  + \la \theta - \theta^{h_t} \wedge \dbar u \ra) - \la u , F_\theta \ra]\\
& = 2i[\la u , F_{h_t} + \dbar (\theta - \theta^{h_t}) - F_\theta \ra] = 0,
\end{align*}
where in the last equality we have used that
\begin{equation}\label{eq:F11trans}
(F_{h_t} - F_\theta)^{1,1} = \dbar (\theta^{h_t} - \theta).
\end{equation}
Therefore, $\zeta_t \in \Lie \Aut_A(E_{Q_0})$ for all $t$ (see Definition \ref{def:PicA}), which proves the `if part' of the statement.

Conversely, assume that there exists a curve $\boldsymbol{\omega}_t = (\omega_t + \upsilon_t,h_t) \in \mathbf{B}_{Q_0}$ joining $\boldsymbol{\omega}$ and $\boldsymbol{\omega}'$, and a one-parameter family of Lie algebra elements $\zeta_t = (s_t,B_t) \in \Ker \mathbf{a}_0$, such that
$$
\rho_0(\zeta_t)_{|\boldsymbol{\omega}_t} := \Big{(} - \Im \; (B^{1,1}_t + 2 \la \theta^{h_t} - \theta \wedge \dbar s_t \ra) + \tilde \upsilon_t,\tfrac{1}{2}(s_t - s_t^{*_{h_t}})\Big{)} = (\dot \omega_t + \dot{\upsilon}_t,\dot h_t h_t^{-1}),
$$
for a suitable $(2,0)$-form $\tilde \upsilon_t$ (see \eqref{eq:LiePicQ0exp}). Taking derivatives of the Aeppli map along the curve
\begin{align*}
\frac{d}{dt}Ap(\boldsymbol{\omega}_t,\boldsymbol{\omega}) & = [\dot \omega_t -2i \la  \dot h_t h_t^{-1}, F_{h_t}\ra]\\
& = -\Im [B_t^{1,1} + 2 \la \theta^{h_t} - \theta \wedge \dbar s_t \ra - 2 \la  s_t, F_{h_t}\ra]\\
& = -\Im [B_t^{1,1} - 2 \la  s_t, F_{h_t} - \dbar(\theta^{h_t} - \theta) \ra]\\
& = -\Im \; \mathbf{a}_0 (\zeta_t) = 0,
\end{align*}
which proves the statement. For the last equality, we have used \eqref{eq:F11trans} combined with \eqref{eq:aepplimap0}, while the second equality follows from
$$
\Im \; \la  s_t, F_{h_t}\ra = -i \la  \dot h_t h_t^{-1}, F_{h_t}\ra.
$$
\end{proof}


\subsection{Complexified Aeppli classes}\label{subsec:moduliapp}

To provide the desired explanation for the variations of `complexified Aeppli classes' appearing in formula \eqref{eq:metricfibre}, we dwell further into the geometry of of the Teichm\"uller space $\cL^0/\Aut_{dR}(E)$ for string algebroids. Recall here that the infinitesimal Donaldson-Uhlenbeck-Yau type Theorem \ref{thm:DUYinfinitesimal} identifies the tangent to the moduli space $\cM_\ell$ with the tangent to the Teichm\"uller space. 

By Lemma \ref{lem:ModuliBC}, the fibre over $[P] \in \cC^0/\Ker \sigma_{\underline P}$ of the natural map 
\begin{equation}\label{eq:moduliQP}
\cL^0/\Aut(E) \to \cC^0/\Ker \sigma_{\underline P} 
\end{equation}
parametrizes isomorphism classes of string algebroids with underlying principal $G$-bundle $P$. To give a cohomological interpretation of this fibre, consider again the sheaf  $\Omega^{2,0}_{cl}$ of (holomorphic) closed $(2,0)$-forms on $X$ and the homomorphism \eqref{eq:sigmaP}. The quotient
$$
H^1(\Omega^{2,0}_{cl})/\Im \; \sigma_P.
$$
can be identified with the set of isomorphism classes of string algebroids with underlying holomorphic principal $G$-bundle $P$ (see \cite[Prop. 3.11]{grt2}). We want to describe the fibre of \eqref{eq:moduliQP} as a natural subspace of $H^1(\Omega^{2,0}_{cl})/\Im \; \sigma_P$. Consider the natural map from Aeppli to Bott-Chern cohomology induced by the $\dbar$ operator:
\begin{equation}\label{eq:dbar02}
H^{1,1}_A(X) \lra{\dbar} \overline{H^1(\Omega^{2,0}_{cl})} := \frac{\Ker \; d \colon \Omega^{1,2} \oplus \Omega^{0,3} \to \Omega^{2,2} \oplus \Omega^{1,3} \oplus \Omega^{0,4}}{ \Im \; d \colon \Omega^{0,2} \to\Omega^{1,2} \oplus \Omega^{0,3}}.
\end{equation}

\begin{lemma}\label{lem:fibreModQP}
The fibre of \eqref{eq:moduliQP} over $[P]$ is an affine space modelled on the image of the map
\begin{equation}\label{eq:partialmapind}
\partial \colon \ker \dbar \to H^1(\Omega^{2,0}_{cl})/\Im \; \sigma_P
\end{equation}
induced by \eqref{eqannex:partialmap}, where $\ker \dbar \subset H^{1,1}_A(X)$ is defined by \eqref{eq:dbar02} 
\end{lemma}

\begin{proof}
Fix a lifting $L_0 \in \cL^0$ and denote by $P$ the induced holomorphic principal $G$-bundle structure on $\underline{P}$. Without loss of generality, we fix an isotropic splitting $\lambda_0 \colon T\underline{X} \to E$ and regard 
$$
\cL^0 \subset \Omega^{1,1 + 0,2} \oplus \Omega^{0,1}(\ad \underline{P}). 
$$
We can choose $\lambda_0$ such that $L_0 = (0,0)$, with induced three-form $H \in \Omega^{3,0 + 2,1}$ and connection $\theta^h$, for some choice of reduction $h \in \Omega^0(P/K)$. Then, by Proposition \ref{prop:QLexp}, if $L = (\gamma,\beta) \in \cL^0$ induces $[P] \in \cC^0/\Ker \sigma_{\underline P}$ it follows that $\beta$ is in the $\Ker \sigma_{\underline P}$-orbit of $0$. By \eqref{eq:Picardses}, we can `gauge' $\beta$ and assume that $(\gamma,\beta) = (\gamma,0)$. Hence,
\begin{equation*}\label{eq:conjugatepartial}
d\gamma^{0,2} + \dbar \gamma^{1,1} = 0
\end{equation*}
and $\gamma$ induces a class
$$
[\gamma^{1,1}] \in \ker \dbar \subset H^{1,1}_A(X).
$$
The change in the isomorphism class of the string algebroid, from $L_0$ to $L$, is (see Proposition \ref{prop:QLexp})
$$
\partial([\gamma^{1,1}]):= [\partial\gamma^{1,1}] \in H^1(\Omega^{2,0}_{cl})/\Im \; \sigma_P.
$$
An element $(\gamma',0) \in \cL^0$ is in the same $\Aut(E)$-orbit as $(\gamma,0)$ if and only if the corresponding string algebroids are isomorphic (see Lemma \ref{lem:ModuliBC}). This is equivalent to the existence of $g \in \cG_P$ and $B \in \Omega^{2,0}$ such that (see Proposition \ref{lemma:deRhamC})
$$
\partial \gamma'^{1,1}  = \partial \gamma^{1,1} + CS(g\theta^h) - CS(\theta^h) - d\la g\theta^h \wedge \theta^h\ra + dB.
$$
Thus, the induced map from the fibre of \eqref{eq:moduliQP} over $[P]$ to $H^1(\Omega^{2,0}_{cl})/\Im \; \sigma_P$ is well defined and injective. Surjectivity onto the image of \eqref{eq:partialmapind} follows from Proposition \ref{eq:conjugatepartial}.
\end{proof}

We turn next to the study of the map $\cL^0/\Aut_{dR}(E) \to \cL^0/\Aut(E)$.
In order to make the link with Aeppli classes we shall consider instead
the map
\begin{equation}\label{eq:moduliQAQ}
\cL^0/\Aut_{A}(E) \to \cL^0/\Aut(E),
\end{equation}
where $\Aut_{A}(E)$ is the group in Section \ref{sec:Aporbit}. The tangent to $\cL^0/\Aut_{A}(E)$ at the class of $L \in \cL^0$ is given (formally) by the cohomology of the complex 
\begin{equation}
  \label{eq:TLcomplex}
    (C^*) \quad \Lie \Aut_A(E)  \lra{\mathbf{P}^c} \Omega^{1,1 + 0,2} \oplus \Omega^{0,1}(\ad \underline{P}) \lra{\mathbf{L}^c} \Omega^{1,2 + 0,3} \oplus \Omega^{0,2}(\ad \underline{P})
,
\end{equation}
where $\Lie \Aut_A(E) \subset \Omega^0(\ad \underline{P}) \oplus \Omega_\CC^2$
and
\begin{align*}
\mathbf{P}^c(r,B) & = (B^{1,1 + 0,2}, \dbar r),\\
\mathbf{L}^c(\dot \gamma,\dot \beta) & = (d \dot \gamma^{0,2} + \dbar \dot \gamma^{1,1} - 2 \la \dot \beta,F_h \ra , \dbar \dot \beta).
\end{align*}
The following result is stated without proof.

\begin{lemma}\label{lem:sesTLM-2}
Assume that $h^0(\ad P) = 0$ and $h^{0,1}_A(X) = 0$. Then, there is an exact sequence
\begin{equation*}\label{eq:sesTWM}
0 \lra{} \Ker \partial \lra{} H^1(\widehat C^*) \lra{}  H^1(C^*) \lra{} 0
\end{equation*}
where
\begin{equation}\label{eq:partial02}
H^{0,2}_{\dbar}(X) \lra{\partial} H^{1,2}_{BC}(X).
\end{equation}
\end{lemma}

We want to characterize the tangent to the fibre of \eqref{eq:moduliQAQ}. Strikingly, this infinitesimal study requires the classical Futaki invariant for the principal bundle $P$ (see \cite[App. A]{grst}). Let $\mathfrak{b} \in  H^{n-1,n-1}_{BC}(X,\RR)$ be a Bott-Chern class. Then, the Futaki invariant of $P$ is given by a Lie algebra homomorphism
$$
\cF_{\mathfrak{b}} \colon \Lie \; \cG_P \to \CC
$$
which provides an obstruction to the existence of solutions of the Hermite-Yang-Mills equations for a given balanced metric on $X$ with class $\mathfrak{b}$ (and hence in particular of \eqref{eq:Calabil}). Using the duality pairing $H^{1,1}_A(X) \cong H^{n-1,n-1}_{BC}(X)^*$ between the Aeppli and Bott-Chern cohomologies, the Futaki invariant can be regarded as the Lie algebra homomorphism
\begin{align*}
\cF \colon \Lie \; \cG_P &\to H^{1,1}_A(X) \\ s &\mapsto [\la s, F_h\ra ]
\end{align*}
for any choice of reduction $h \in \Omega^0(P/K)$. Using Lemma \ref{lem:CSRinvariant}, it is not difficult to see that \eqref{eq:partialmapind} induces a well-defined map
\begin{equation}\label{eq:partialtildeinf}
\partial \colon \ker \dbar / \Im \; \cF \to H^1(\Omega^{2,0}_{cl})/\Im \; d \sigma_P,
\end{equation}
where $\ker \dbar \subset H^{1,1}_A(X)$ is defined by \eqref{eq:dbar02}.

\begin{lemma}\label{lem:fibreModQAQ}
Let $L \in \cL^0$ with induced principal bundle $P$. Then, the tangent to the fibre of \eqref{eq:moduliQAQ} over $[L] \in \cL^0/\Aut(E)$ is isomorphic to the kernel of \eqref{eq:partialtildeinf}.
\end{lemma}

\begin{proof}
We build on the proof of Lemma \ref{lem:fibreModQP}, following the same notation. We fix a lifting $L_0 \in \cL^0$ and an isotropic splitting $\lambda_0 \colon T\underline{X} \to E$. If $(\gamma,0), (\gamma',0) \in \cL^0$ represent elements over $[L_0] \in \cL^0/\Aut(E)$ 
there exists $(g,\tau) \in \Aut(E)$ (see Lemma \ref{lem:HomQQ}) such that $g \in \cG_P \cap \Ker \cG_{\underline{P}}$ and 
$$
\gamma' = \gamma - \tau^{1,1 + 0,2}.
$$
Therefore, if $(\dot \gamma,0), (\dot \gamma',0)$ are tangent to the fibre over $[L_0]$ we have (see Definition \ref{def:PicA})
$$
\dot \gamma'^{1,1} - \dot \gamma^{1,1} - 2\la s , F_h \ra \in \Im \; \partial \oplus \dbar
$$
for $s \in \Lie \; \cG_P$. Thus, the map
$$
[(\dot \gamma,0)] \mapsto [\dot \gamma^{1,1}] \in \ker \; \partial \subset \ker \dbar/\Im \; \cF
$$
is well defined and injective. Surjectivity follows from Lemma \ref{lemma:liftings}.
\end{proof}

As a straightforward consequence of Lemma \ref{lem:fibreModQP} and Lemma \ref{lem:fibreModQAQ}, we obtain the following cohomological interpretation of the tangent space to the fibres of the map between moduli spaces
\begin{equation}\label{eq:moduliQAP}
\cL^0/\Aut_A(E) \to \cC^0/\cG_{\underline P}
\end{equation}
induced by \eqref{eq:modulidiagram}. Relying on Theorem \ref{thm:DUYinfinitesimal}, this provides the desired explanation for the `complexified Aeppli classes' appearing in formula \eqref{eq:metricfibre} for the fibre-wise moduli metric.

\begin{proposition}\label{prop:Tfibre}
The tangent space to the fibre of \eqref{eq:moduliQAP} over $[P]$ is isomorphic to $ \ker \dbar/\Im \; \cF \subset H^{1,1}_A(X)/\Im \; \cF$, where $\dbar$ is as in \eqref{eq:dbar02}.
\end{proposition}


\begin{thebibliography}\frenchspacing\smallbreak

	
\bibitem{AGF1} B. Andreas and M. Garcia-Fernandez, \emph{Solutions of the Strominger system via stable bundles on Calabi-Yau threefolds}, Comm. Math. Phys. \textbf{315} (2012), no. 1, 153--168, MR2966943, Zbl 1252.32031.
	
\bibitem{AGS} L. B. Anderson, J. Gray and E. Sharpe, \emph{Algebroids, heterotic moduli spaces and the Strominger system}, JHEP {\bf 07} (2014) 37pp.
	
	
\bibitem{AQS}
L. Anguelova, C. Quigley and S. Sethi, \emph{The leading quantum corrections to stringy K\"ahler potentials}, JHEP \textbf{10} (2010) 065, MR2780523, Zbl 1291.81284.
	
\bibitem{AOMSS} A. Ashmore, X. de la Ossa, R. Minasian, C. Strickland-Constable, E. Svanes, \emph{Finite deformations from a heterotic superpotential: holomorphic Chern--Simons and an $L^\infty$ algebra},  J. High Energy Phys. 2018, no. \textbf{10}, 179, 58 pp, MR3879705, Zbl 1402.83091. 

\bibitem{Waldram} A. Ashmore, C. Strickland-Constable, D. Tennyson, D. Waldram, \emph{Heterotic backgrounds via generalised geometry: moment maps and moduli},  J. High Energy Phys. 2020, no. \textbf{11}, 071, 45 pp, MR4204214, Zbl 1456.83087. 

\bibitem{AB}
        M. F. Atiyah and R. Bott,
        \emph{The Yang--Mills equations over Riemann surfaces}, Philos. Trans. Roy. Soc. London \textbf{A 308} (1983) 523--615, MR0702806, Zbl 0509.14014.
	
\bibitem{BeckerTsengYau} K.~Becker, L. Tseng and S.-T- Yau, \emph{Moduli space of torsional manifolds}, Nucl. Phys. {\bf B 786} (2007) 119--134, MR2358800, Zbl 1225.32029.
	
	



	
\bibitem{BGS} J.-M. Bismut, H. Gillet, and C. Soul\'e, \emph{Analytic torsion and holomorphic determinant bundles I: Bott-Chern forms and analytic torsion}, Comm. Math. Phys. {\bf B 115} (1988) 49--78, MR0929146, Zbl 0651.32017.
	
\bibitem{BottChern} R. Bott and S. S. Chern, \emph{Hermitian vector bundles and the equidistribution of the zeroes of their holomorphic sections}, Acta Math., {\bf 114} (1968) 71--112, MR0185607, Zbl 0148.31906.
	
	
	
	

\bibitem{CD} P. Candelas and X. de la Ossa, \emph{Moduli space of Calabi-Yau manifolds}, Nuclear Phys. {\bf B 355} (1991) 455--481, 
MR1103074, Zbl 0732.53056.
	
\bibitem{CDM} P. Candelas, X. de la Ossa, and J. McOrist, \emph{A Metric for Heterotic Moduli}, Comm. Math. Phys. {\bf 356} (2017) 567--612, MR3707334, Zbl 1379.58001.
	
	
	
	
	\bibitem{OssaSvanes} X. De la Ossa and E. Svanes, \emph{Holomorphic bundles and the moduli space of $N=1$ supersymmetric heterotic compactifications}, JHEP {\bf 10} (2014) 123, 54pp, MR3324801, Zbl 1333.81413.
	
	\bibitem{Don} S.K.~Donaldson, \emph{Anti-self-dual Yang--Mills connections over complex algebraic surfaces and stable vector bundles}, Proc. London Math. Soc. {\bf 50} (1985) 1--26, MR0765366, Zbl 0529.53018.
	
	
\bibitem{Fei} T. Fei, \emph{Generalized Calabi-Gray Geometry and Heterotic Superstrings}, Proceedings of the International Consortium of Chinese Mathematicians 2017, 261--281, Int. Press, Boston, MA, 2020, MR4251114, Zbl 1459.81089.
	

	
\bibitem{Fine} J. Fine, \emph{The Hamiltonian geometry of the space of unitary connections with symplectic curvature}, J. Symplectic Geom. {\bf 12} (2014) 105--123, MR3194077, Zbl 1304.53082.
	
	
	


\bibitem{GF} M. Garcia-Fernandez, \emph{Torsion-free generalized connections and heterotic supergravity}, Comm. Math. Phys. {\bf 332} (2014) 89--115, MR3253700, Zbl 1300.83045.
	
	\bibitem{GF2} M. Garcia-Fernandez, \emph{Lectures on the Strominger system}, Travaux Math\'ematiques, Special Issue: School GEOQUANT at the ICMAT, Vol. XXIV (2016) 7--61, MR3643933, Zbl 1379.35265.
	
	
	
\bibitem{grst} M. Garcia-Fernandez, R. Rubio, C. Shahbazi, and C. Tipler, \emph{Canonical metrics on holomorphic Courant algebroids}, Proc. Lond. Math. Soc. {\bf 125} (2022), no.3, 700--758, MR4480887, Zbl 1523.32043.
	
\bibitem{grt} M. Garcia-Fernandez, R. Rubio and C. Tipler, \emph{Infinitesimal moduli for the Strominger system and Killing spinors in generalized geometry}, Math. Ann. {\bf 369} (2017) 539--595, MR3694654, Zbl 1442.58010.
	
\bibitem{grt2} M. Garcia-Fernandez, R. Rubio and C. Tipler, \emph{Holomorphic string algebroids}, Trans. Amer. Math. Soc. 373 (2020), no. \textbf{10}, 7347--7382, MR4155210, Zbl 1454.53067.

\bibitem{Gastel} A. Gastel, \emph{Canonical gauges in higher gauge theory}, Commun. Math. Phys. 376 (2020), no. \textbf{2}, 1053--1071, MR4103963, Zbl 1447.81162. 


	
	
	
	
	
	

\bibitem{GrSt} M. Grutzmann and M. Sti\'enon, \emph{Matched pairs of Courant algebroids}, Indag. Math. (N.S.) {\bf 25} (2014), no. 5, 977--991, MR3264784, Zbl 1304.53079.
	
\bibitem{G2} M. Gualtieri, \emph{Generalized K\"ahler Geometry}, Comm. Math. Phys. (1) {\bf 331} (2014) 297--331, MR3232003, Zbl 1304.53080.
	
	\bibitem{GVW} S. Gukov, C. Vafa and E. Witten, \emph{CFT's from Calabi-Yau four-folds}, Nucl. Phys. B {\bf 584} (2000) 69, MR1786689, Zbl 0984.81143 [Erratum-ibid. B 608 (2001) 477].
	
	\bibitem{Gukov} S. Gukov, \emph{Solitons, superpotentials and calibrations}, Nucl. Phys. B {\bf 574} (2000) 169, MR1759777, Zbl 1056.53508.
	
	\bibitem{GLM} S. Gurrieri, A. Lukas, and A. Micu, \emph{Heterotic string compactified on half-flat manifolds.}, Phys. Rev. {\bf D70} (2004) 126009, 18pp, MR2124222.
	
	
	
	\bibitem{HullTurin} C. Hull, \emph{Superstring compactifications with torsion and space-time supersymmetry}, In Turin 1985 Proceedings ``Superunification and Extra Dimensions'' (1986) 347--375, MR0872720. 
	
\bibitem{Huyb} D.~Huybrechts, \emph{The tangent bundle of a Calabi-Yau manifold - deformations and restriction to rational curves}, Comm. Math. Phys. {\bf 171} (1995) 139--158, MR1341697, Zbl 0852.14011.
	
	
	
	
	
	
	
	
	
	
\bibitem{lt} M. L\"ubcke and A. Teleman, {\it The Kobayashi-Hitchin correspondence}, World Scientific Publishing Co. Inc. (1995), MR1370660, Zbl 0849.32020.
	
\bibitem{McS}
D. McDuff and D. Salamon,
\emph{Introduction to Symplectic Topology}, in `Oxford Mathematical Monographs', The Clarendon Press, Oxford University Press, New York, 1998, MR1698616, Zbl 1066.53137.

\bibitem{McSisca} J. McOrist and R. Sisca, \emph{Small gauge transformations and universal geometry in heterotic theories},  SIGMA Symmetry Integrability Geom. Methods Appl. \textbf{16} (2020), Paper No. 126, 48 pp., MR4181524, Zbl 1459.53037. 
	
	

\bibitem{NS} M. S. Narasimhan and C. S.Seshadri, \emph{Stable and unitary vector bundles on a compact Riemann surface}, Annals of Math. {\bf 65} (1965), 540--67, MR0184252, Zbl 0171.04803.
	
	
	
	
	
\bibitem{PPZ2} D.-H. Phong, S. Picard, and X. Zhang, \emph{New curvature flows in complex geometry}, Surveys in differential geometry 2017. Celebrating the 50th anniversary of the Journal of Differential Geometry, 331--364, Surv. Differ. Geom., 22, Int. Press, Somerville, MA, 2018, MR3838124, Zbl 1408.32027.
	
	
\bibitem{Redden}
C. Redden,
\emph{String structures and canonical $3$-forms}, Pac. J. Math. {\bf 249} (2011) 447--484, MR2782680, Zbl 1222.57023.

\bibitem{SWolf}
C. S\"amann and M. Wolf,
\emph{Six-Dimensional Superconformal Field Theories from Principal 3-Bundles over Twistor Space}, Lett. Math. Phys. {\bf 104} (2014) 1147--1188, MR3229171, Zbl 1306.53037.
	
	
	
	
\bibitem{Strom} A.~Strominger, \emph{Superstrings with torsion}, Nucl. Phys. B {\bf 274} (2) (1986) 253--284, MR0851702.
	
	
	
\bibitem{UY} K. K. Uhlenbeck and S.-T. Yau, \emph{On the existence of Hermitian--Yang--Mills connections in stable bundles}, Comm. Pure and Appl. Math. \textbf{39-S} (1986) 257--293; \textbf{42} (1989) 703--707, MR0861491, Zbl 1415.32019.
	
\bibitem{Waldorf} C. Waldorf,
		\emph{String connections and Chern-Simons theory}, Trans. Amer. Math. Soc. {\bf 365} (2013) 4393--4432, MR3055700, Zbl 1277.53024.	
	
	
\bibitem{Xu} Y. Sheng, X. Xu and C. Zhu, \emph{String principal bundles and Courant algebroids},  Int. Math. Res. Not. IMRN 2021, no. 7, 5290--5348, MR4241129, Zbl 1470.58019. 
	
\bibitem{Yau0} S.-T. Yau, \emph{Calabi's conjecture and some new results in algebraic geometry}, Proc. Natl. Acad. Sci. USA {\bf 74} (1977) 1798--1799, MR0451180, Zbl 0355.32028.
	
\bibitem{Yau2} S.-T. Yau, \emph{Metrics on complex manifolds}, Sci. China Math. \textbf{53} (2010), no. 3, 565-572, MR2608313, Zbl 1193.53151.

\end{thebibliography}
\end{document}